\documentclass[10pt]{amsart}
\usepackage{graphicx} 
\usepackage{xcolor}
\usepackage{tikz}
\usetikzlibrary{calc,backgrounds}
\usepackage{amsmath, amssymb, amsthm}
\usepackage{bm}
\usepackage{mathtools}
\usepackage{nicefrac}
\usepackage{stmaryrd}

\usepackage[shortlabels]{enumitem}
\usepackage{aliascnt}
\usepackage[bookmarks=true,pdfstartview=FitH, pdfborder={0 0 0}, colorlinks=true,citecolor=red, linkcolor=blue]{hyperref}

\theoremstyle{plain}
\newtheorem{thm}{Theorem}[section]
\newaliascnt{coro}{thm}
\newaliascnt{prop}{thm}
\newaliascnt{lem}{thm}
\newtheorem{coro}[coro]{Corollary}
\newtheorem{prop}[prop]{Proposition}
\newtheorem{lem}[lem]{Lemma}
\aliascntresetthe{coro}
\aliascntresetthe{prop}
\aliascntresetthe{lem}

\theoremstyle{definition}
\newaliascnt{defn}{thm}
\newaliascnt{asu}{thm}
\newaliascnt{con}{thm}
\newtheorem{defn}[defn]{Definition}
\newtheorem{asu}[asu]{Assumption}

\aliascntresetthe{defn}
\aliascntresetthe{asu}
\aliascntresetthe{con}

\theoremstyle{remark}
\newaliascnt{rem}{thm}
\newaliascnt{exa}{thm}
\newaliascnt{masu}{thm}
\newaliascnt{nota}{thm}
\newaliascnt{sett}{thm}
\newtheorem{rem}[rem]{Remark}

\aliascntresetthe{rem}
\aliascntresetthe{exa}
\aliascntresetthe{masu}
\aliascntresetthe{nota}
\aliascntresetthe{sett}

\newcounter{stpi}
\newcounter{stpci}
\newcounter{stpiii}

\numberwithin{equation}{section}
\setlist[enumerate]{font=\normalfont}

\newcommand{\eps}{\varepsilon}

\newcommand{\beps}{\bm{\eps}}
\newcommand{\bsigma}{\bm{\sigma}}
\newcommand{\bu}{\bm{u}}
\newcommand{\bd}{\bm{d}}
\newcommand{\bk}{\bm{k}}
\newcommand{\bx}{\bm{x}}
\newcommand{\by}{\bm{y}}
\newcommand{\bff}{\bm{f}}
\newcommand{\bv}{\bm{v}}
\newcommand{\bw}{\bm{w}}
\newcommand{\bz}{\bm{z}}
\newcommand{\bI}{\bm{I}}

\newcommand{\dbeps}{\dot{\beps}}
\newcommand{\oOmega}{\overline{\Omega}}
\newcommand{\oG}{\overline{G}}
\newcommand{\dOmega}{\del\Omega}
\newcommand{\dG}{\del G}

\newcommand{\rC}{\mathrm{C}}
\newcommand{\rL}{\mathrm{L}}
\newcommand{\rW}{\mathrm{W}}
\newcommand{\rH}{\mathrm{H}}
\newcommand{\rB}{\mathrm{B}}
\newcommand{\rD}{\mathrm{D}}
\newcommand{\rN}{\mathrm{N}}
\newcommand{\mri}{\mathrm{i}}
\newcommand{\rE}{\mathrm{E}}
\newcommand{\rI}{\mathrm{I}}
\newcommand{\rII}{\mathrm{II}}
\newcommand{\rIII}{\mathrm{III}}
\newcommand{\rX}{\mathrm{X}}
\newcommand{\rY}{\mathrm{Y}}
\newcommand{\LFI}{\mathrm{LFI}}
\newcommand{\rBUC}{\mathrm{BUC}}
\newcommand{\rh}{\mathrm{h}}
\newcommand{\rA}{\mathrm{A}}
\newcommand{\rd}{\mathrm{d}}
\newcommand{\srd}{\, \rd}
\newcommand{\ra}{\mathrm{a}}
\newcommand{\ro}{\mathrm{o}}
\newcommand{\rb}{\mathrm{b}}
\newcommand{\rR}{\mathrm{R}}
\renewcommand{\rm}{\mathrm{m}}
\newcommand{\avg}{\mathrm{avg}}
\newcommand{\rc}{\mathrm{c}}
\newcommand{\rsh}{\mathrm{sh}}

\newcommand{\rt}{\mathrm{t}}
\newcommand{\crit}{\mathrm{crit}}
\newcommand{\ice}{\mathrm{ice}}
\newcommand{\per}{\mathrm{per}}
\newcommand{\mre}{\mathrm{e}}
\newcommand{\cor}{\mathrm{cor}}

\newcommand{\bR}{\mathbb{R}}
\newcommand{\bS}{\mathbb{S}}
\newcommand{\bA}{\mathbb{A}}
\newcommand{\bC}{\mathbb{C}}
\newcommand{\bN}{\mathbb{N}}
\newcommand{\bE}{\mathbb{E}}

\newcommand{\bF}{\mathbb{F}}

\newcommand{\cA}{\mathcal{A}}
\newcommand{\cL}{\mathcal{L}}
\newcommand{\cB}{\mathcal{B}}
\newcommand{\cR}{\mathcal{R}}
\newcommand{\cH}{\mathcal{H}}
\newcommand{\cK}{\mathcal{K}}

\newcommand{\cE}{\mathcal{E}}
\newcommand{\cT}{\mathcal{T}}
\newcommand{\cS}{\mathcal{S}}

\newcommand{\Hinfty}{\cH^\infty}

\newcommand{\tP}{\Tilde{P}}
\newcommand{\tS}{\Tilde{S}}
\newcommand{\tf}{\Tilde{f}}
\newcommand{\tbv}{\Tilde{\bv}}
\renewcommand{\th}{\Tilde{h}}
\newcommand{\tA}{\Tilde{A}}
\newcommand{\tdbeps}{\Tilde{\dbeps}}
\newcommand{\tbA}{\Tilde{\bA}}
\newcommand{\tB}{\Tilde{B}}
\newcommand{\tbu}{\Tilde{\bu}}

\newcommand{\hbu}{\widehat{\bu}}
\newcommand{\hbv}{\widehat{\bv}}
\newcommand{\hh}{\widehat{h}}
\newcommand{\hA}{\widehat{A}}

\newcommand{\tin}{\enspace \text{in} \enspace}
\newcommand{\ton}{\enspace \text{on} \enspace}
\newcommand{\tfor}{\enspace \text{for} \enspace}
\newcommand{\tforall}{\enspace \text{for all} \enspace}
\newcommand{\tand}{\enspace \text{and} \enspace}
\newcommand{\tif}{\enspace \text{if} \enspace}
\newcommand{\twith}{\enspace \text{with} \enspace}

\newcommand{\taswellas}{\enspace \text{as well as} \enspace}
\newcommand{\twhere}{\enspace \text{where} \enspace}

\newcommand{\bAFI}{\bA^{\mathrm{FI}}}
\newcommand{\AFID}{A_{\rD}^{\mathrm{FI}}}
\newcommand{\ALFI}{A^{\mathrm{LFI}}}
\newcommand{\tbAFI}{\tbA^{\mathrm{FI}}}
\newcommand{\ALFIhyp}{A_{\mathrm{hyp}}^{\mathrm{LFI}}}

\DeclareMathOperator{\tr}{tr}
\DeclareMathOperator{\diag}{diag}
\renewcommand{\div}{\mathrm{div} \, }
\DeclareMathOperator{\re}{Re}
\DeclareMathOperator{\im}{Im}
\DeclareMathOperator{\Id}{Id}
\DeclareMathOperator{\dist}{dist}

\newcommand{\shadedDelta}{\mathpalette\shadedDeltaAux\relax}

\newcommand{\shadedDeltaAux}[2]{%
  \tikz[baseline=(D.base)]{%
    \node[inner sep=0pt, outer sep=0pt] (D) {$#1\Delta$};
    \begin{scope}[on background layer]
      \fill[gray!20]
        ($(D.north)!0.32!(D.center)$)
        --
        ($(D.south west)!0.28!(D.center)$)
        --
        ($(D.south east)!0.28!(D.center)$)
        -- cycle;
    \end{scope}
  }%
}

\newcommand{\del}{\partial}

\newcommand{\oB}{\overline{B}}

\begin{document}

\title{Analysis and numerical simulations of a landfast ice model}

\author{Felix Brandt}
\address{Department of Mathematics, University of Southern California, Los Angeles, 90089, CA, USA.}
\email{fbrandt@usc.edu}
\author{Carolin Mehlmann}
\address{Institute of Analysis and Numerics, Otto-von-Guericke University Magdeburg, Universit\"atsplatz 2, 39106 Magdeburg, Germany.}
\email{carolin.mehlmann@ovgu.de}


\begin{abstract}
In this manuscript, we consider a common modeling framework for Arctic landfast ice based on the work of Lemieux et al.~\cite{LDBRSF:16}, which is designed for use in large-scale climate models. 
This approach extends the classical viscous--plastic sea-ice model introduced by Hibler~\cite{Hib:79}, which remains the most widely used model for simulating large-scale sea-ice dynamics in climate science.
In particular, landfast ice refers to sea-ice that is attached to the coastline or grounded and therefore exhibits nearly vanishing motion.
The model under consideration in this manuscript augments Hibler's viscous-plastic sea-ice model by tensile strength and a threshold-dependent basal stress.
We present a rigorous analytical and numerical study of this landfast ice model. 
The main analytical contributions are the local strong well-posedness, the global strong well-posedness in the absence of external forces and for initial data close to constant equilibrium solutions, and the existence of time-periodic solutions.
The limiting case of maximal tensile strength, corresponding to $k_\rt \equiv 1$, is treated separately, since the pressure coupling then vanishes and a new smallness condition on the stress regularization parameter is required.
For the physically most relevant system without scalar diffusion, we establish local strong well-posedness through a Lagrangian formulation and anisotropic maximal $\rL^p$-regularity.
Complementing the analysis, we perform numerical simulations that illustrate key qualitative differences between landfast ice and classical viscous--plastic sea-ice models. 
In particular, the simulations reveal the formation of stationary equilibrium states characterized by vanishing ice velocity. 
These observations are consistent with the global-in-time existence result close to equilibria established in \autoref{thm:glob ex close to equil} as well as the time-periodic result in \autoref{thm:time-per sols}. 
The combined analytical and numerical results provide new insight into the structure, stability, and long-term behavior of landfast ice dynamics.
\end{abstract}

\maketitle

\section{Introduction}\label{sec:intro}

Sea-ice in polar regions plays a central role in mediating the exchange of heat and momentum between the ocean and the atmosphere, thereby influencing the climate system as a whole \cite{SN:18}. 
The sea-ice cover typically consists of several regimes, which can be divided into landfast ice, pack ice, and the marginal ice zone.
In contrast to pack ice, which is transported by winds and ocean currents, landfast ice (also referred to as fast ice in the sequel) remains stationary along coastlines or over shallow seabeds \cite{Mah:18}. 
In contrast, the marginal ice zone marks the transition from the pack ice regime to the open ocean; ice floes dynamically interact with each other \cite{CR:13}. 

In this manuscript, we consider a common framework based on the work of Lemieux et al.~\cite{LDBRSF:16} to represent Arctic landfast ice in large-scale climate models. 
This approach to model Arctic landfast ice is based on a modification of the viscous–plastic sea-ice model introduced by Hibler \cite{Hib:79}. 
The model of Hibler remains the standard for simulating large-scale sea-ice dynamics in climate models \cite{Blo:20}, where sea-ice is typically represented as a continuous viscous–plastic (VP) material. 
Recent modeling studies~\cite{LDBRSF:16, KBH:10} introduced a basal stress parameterization together with a modification of sea-ice tensile strength within the VP framework to account for anchoring effects in shallow regions. 
This combination has been shown to successfully reproduce the spatial distribution of fast ice in several Arctic coastal areas.

In this paper, for the first time, we present a rigorous analysis of a suitable regularization of the landfast ice model introduced by Lemieux et al.~\cite{LDBRSF:16}.
In fact, we show that for suitable initial data that especially lie in $\rC^1$ in space, the regularized landfast ice model admits a unique, local-in-time, strong solution.
Moreover, for initial data close to constant equilibria, and in the absence of external forcing terms, we show the global strong well-posedness, and we discuss the existence of time-periodic strong solutions for small time-periodic forcing terms.
In addition, we provide numerical simulations, revealing differences between landfast ice and viscous-plastic sea-ice \cite{Hib:79}.

The investigation of sea-ice has attracted a lot of attention in recent years in both numerical and mathematical analysis.
For numerical studies in the context of Hibler's viscous-plastic sea-ice model, we refer, e.g., to the works \cite{ZH:91, LT:09, KDL:15, MR:17, SK:18, MK:21, YP:22, SMLS:23}.
An elastic-viscous-plastic (EVP) sea-ice model was introduced in \cite{HD:97} in order to reduce the computational costs.
Concerning numerical studies of {\em landfast ice}, recently, a hybrid particle-continuum method for the simulation of landfast ice by means of subgrid iceberg interaction has been developed in \cite{MK:26}.

The {\em rigorous} mathematical analysis of sea-ice started only very recently.
In \cite{BDHH:22}, local strong well-posedness and global strong well-posedness close to constant equilibria for a fully parabolic regularized version of Hibler's model were obtained.
At the same time, using energy methods, local strong well-posedness for a different regularization, in the physically most relevant case of hyperbolic balance laws, was established in~\cite{LTT:22}. 
Later, a Lagrangian approach combined with maximal $\rL^p$-regularity yielded local strong well-posedness for Hibler's parabolic-hyperbolic model in \cite{Bra:25}, with the regularization of the stress tensor as in \cite{BDHH:22}.
Motivated by~\cite{HD:97}, the global well-posedness of the EVP sea-ice model with additional Voigt-regularization was shown in \cite{BLTT:25}.
Well-posedness for different regularizations of Hibler's model was also studied in \cite{CKK:23, CK:25, LTT:26}.
The existence and uniqueness of a weak solution to the momentum equation without convective term in Hibler's model with local cut-off was established in~\cite{DD:25}.
In \cite{DGH:25}, by introducing a broader solution concept, the authors showed the existence of energy-driven solutions to the momentum balance without convective term in Hibler's sea-ice model, especially addressing the singular limit in the stress tensor.
Both works \cite{DD:25, DGH:25} seem to be related to the total variation flow.
The analysis from \cite{BDHH:22} was also extended to the situation of time-periodic external forcing terms with applications to time-periodic wind forces~\cite{BH:23}, see also \cite[Sec.~7.2]{Bra:24}, and the analysis of a coupled atmosphere-sea-ice-ocean model~\cite{BBH:26}.

To the best knowledge of the authors, the present article is the first one addressing the rigorous analysis of landfast ice.
For the analysis of a related problem on shallow grounded ice sheets, we refer to \cite{PT:23}.

We now describe the approach in the present paper in more detail.
In order to tackle the well-posedness, we first introduce the so-called {\em fast ice operator}.
Compared to the analysis in \cite{BDHH:22}, we need to account for the adjusted rheology that differs from the one in Hibler's viscous-plastic model by the tensile strength as specified in \eqref{eq:stress tensor}.
We then employ theory for parabolic boundary value problems \cite{DHP:03} based on ellipticity properties of this differential operator as well as its $\rL^q$-realization subject to Dirichlet boundary conditions.
This leads to the maximal $\rL^p$-regularity of the $\rL^q$-realization of the fast ice operator and then also the linearized operator matrix associated with the landfast ice model.
For the well-posedness, we employ quasilinear methods based on maximal $\rL^p$-regularity theory \cite{Ama:95, PS:16}.
In particular, this requires us to establish nonlinear estimates of the basal stress $f_\rb$ that is introduced in \eqref{eq:basal stress}.
Invoking time weights, we are able to lower the regularity of the initial data and exploit instantaneous regularization of the solution.

With regard to the global well-posedness close to equilibria, we make use of the generalized principle of linearized stability \cite{PSZ:09} to deal with the fact that zero lies in the spectrum of the operator matrix due to the Neumann Laplacian operators in the regularized balance laws.
We especially manage to handle the case of tensile strength parameter $k_{\rt} \equiv 1$, which induces a smallness condition on the regularization parameter in the rheology.

For the existence of time-periodic solutions, we capitalize on the Arendt-Bu theorem \cite{AB:02}, providing a characterization of the so-called maximal periodic $\rL^p$-regularity in terms of maximal $\rL^p$-regularity and a spectral condition on the semigroup.
At this stage, to handle the lack of invertibility, we introduce a new ground space, where for the $h$- and $A$-component, we consider the space of functions with spatial average zero.
In this case, we also discuss the situation of $k_{\rt} \equiv 1$.

\subsection{Relation to previous work and main contributions}
\ 

For the sake of the reader, we now indicate the novelty of the paper and place it in context of the existing literature.
In \cite{BDHH:22}, a regularized version of Hibler's viscous-plastic model with parabolic balance laws was analyzed by means of maximal $\rL^p$-regularity and the generalized principle of linearized stability.
In \cite{LTT:22}, a differently regularized approximating model of Hibler type with hyperbolic balance laws was studied by using a high-order energy method.
In \cite{BH:23}, the existence of time-periodic strong solutions to the regularized Hibler model originally investigated in \cite{BDHH:22} was analyzed by using maximal periodic $\rL^p$-regularity.
The analysis in \cite{Bra:25} concerned the elliptic stress regularization of the viscous-plastic stress tensor from \cite{BDHH:22}, but in the situation of hyperbolic balance laws.
For this purpose, Lagrangian coordinates along with maximal $\rL^p$-regularity in anisotropic spaces were used.

Although the analytical framework builds on established maximal-regularity and stability methods, several model-specific issues have to be resolved for the landfast-ice system. 
First, the incorporation of tensile strength changes both the nonlinear second-order stress term and the pressure contribution, while the basal stress introduces a piecewise-defined, threshold-dependent, locally Lipschitz lower-order nonlinearity.
We establish the required ellipticity, maximal-regularity, and nonlinear mapping properties for this system. 
Second, we analyze the local structure of the equilibrium set near constant states and prove nonlinear stability and exponential convergence. 
The limiting case \(k_{\mathrm t}=1\) is structurally distinct because the pressure-gradient coupling vanishes; 
this requires a separate spectral and energy argument and yields a smallness condition on the stress-regularization parameter. 
Third, we prove the existence and local uniqueness of small time-periodic strong solutions by resolving the non-invertibility caused by the Neumann Laplacians on suitable mean-zero spaces. 
Fourth, for the physically more relevant system without artificial diffusion in the thickness and concentration equations, we prove local strong well-posedness using a Lagrangian formulation and anisotropic maximal regularity. 
Finally, we complement the analytical theory with numerical experiments illustrating landfast-specific anchoring, polynya formation, and long-time evolution toward stationary configurations.

To the best of our knowledge, this is the first rigorous PDE analysis of a regularized evolution system based on the landfast-ice model introduced in \cite{LDBRSF:16}, and the first study to discuss the qualitative relationship between its nonlinear stability and time-periodic theory and numerical simulations of landfast-ice formation.

This article is organized as follows.
In \autoref{sec:landfast ice model}, we introduce the governing equations of the landfast ice model from \cite{LDBRSF:16}.
\autoref{sec:well-posedness} is dedicated to the local-in-time strong well-posedness of the regularized landfast ice model as asserted in \autoref{thm:loc strong well-posedness}, and to the required underlying linear theory and nonlinear estimates.
In addition, we discuss the situation of hyperbolic balance laws, i.e., the case without parabolic regularization in the equations for the mean ice thickness and the ice concentration.
\autoref{sec:glob wp close to equil} is concerned with the global strong well-posedness close to equilibrium solutions in the absence of external forcing terms, while in \autoref{sec:time-per sols}, we prove the existence of a time-periodic strong solution.
Finally, in \autoref{sec:numerical simulations}, we present the numerical simulations, revealing key differences between landfast ice and viscous-plastic sea-ice, and reflecting the global existence result close to equilibria in \autoref{thm:glob ex close to equil} and the existence of a time-periodic solution for sufficiently small time-periodic wind forces as shown in \autoref{thm:time-per sols}.
Operator theoretic concepts such as $\cR$-sectoriality, maximal $\rL^p$-regularity, and the bounded $\Hinfty$-calculus are recalled in \autoref{sec:op theoret concepts}, while the required analytical tools such as theory for parabolic boundary value problems, quasilinear existence theory, the generalized principle of linearized theory, and the Arendt-Bu theorem are collected in \autoref{sec:det theory}.

\section{Governing equations of the landfast ice model}\label{sec:landfast ice model}

Let $G \subset \bR^2$ be a bounded domain with boundary $\del G$ of class $\rC^2$ and $(0,T)$, with $0 < T \le \infty$, the time interval on which we consider the sea-ice evolution.
By $x$, $y$, $t$, we denote the horizontal spatial coordinates and the time, respectively. 
The sea-ice  motion is described by the following three variables: the sea-ice concentration $A = A(x,y, t) \colon G \times (0,T) \to [0,1]$ (the fraction of a grid cell covered with ice), the mean ice thickness $h = h(x,y, t) \colon G \times (0,T) \to [0, \infty)$, and the sea-ice velocity $\bv = \bv(x,y,t) \colon G \times (0,T) \to \bR^2$.
The sea-ice concentration and sea-ice thickness evolve over time via transport equations, where we include artificial diffusion terms to simplify the analysis in a first step.
We will also elaborate on the possibility of omitting these diffusion terms in a second step, see \autoref{ssec:rems on parabolic-hyperbolic setting} and \autoref{ssec:proof parabolic-hyperbolic}.
For $d_\rh$, $d_\rA > 0$, the regularized balance laws are given by
\begin{equation}\label{eq:transport}
    \begin{aligned}
        \partial_t h + \div (\bv h) &= d_\rh \Delta h, &&\tin G \times (0,T), \\
        \partial_t A + \div (\bv A) &= d_\rA \Delta A, &&\tin G\times (0,T),
    \end{aligned}
\end{equation}
while the velocity is determined from the momentum equation
\begin{equation}\label{eq:mom}
    \rho h \bigl(\partial_t \bv + (\bv \cdot \nabla) \bv\bigr)=f_{\rc}+f_{\rsh}+f_{\ro}+f_{\ra}+f_{\rb}+f_\sigma, \tin G \times (0,T).
\end{equation}
where $\rho > 0$ is the (constant) sea-ice density. 
The forcing terms on the right-hand side of \eqref{eq:mom} are introduced in the sequel.
The differences with regard to Hibler's model \cite{Hib:79} are the presence of the basal stress $f_{\rb}$ that will be made precise in \eqref{eq:basal stress} as well as the form of the internal ice stress $f_\sigma$, which we will elaborate on below.
Note that the convective term $(\bv \cdot \nabla)\bv$ will be neglected in the numerical simulations in \autoref{sec:numerical simulations}.
In the above, the term
\begin{equation*}
    f_\sigma = \div \bm{\sigma}
\end{equation*}
represents the sea-ice rheology, which models the relation of the internal sea-ice stresses $\bm{\sigma}(\bv, A, h)$ and the strain rates
\begin{equation*}
    \dot{\bm{\varepsilon}} = \frac{1}{2} (\nabla \bv + \nabla \bv^\top) = \dot{\bm{\varepsilon}}' + \frac{1}{2} \tr(\dot{\bm{\varepsilon}}) \bI \in \mathbb{R}^{2 \times 2},
\end{equation*}
where $\bI \in \mathbb{R}^{2 \times 2}$ denotes the identity matrix.
Note that $\dot{\bm{\varepsilon}}'$ is trace-free.
Following Lemieux et al.~\cite{LDBRSF:16}, we consider a modified version of the viscous-plastic rheology to model landfast ice.
To make the notation more compact, we introduce the rheology in terms of the trace and deviatoric parts of the strain rate tensor, i.e.,
\begin{equation}\label{eq:stress tensor}
    \bm{\sigma} \coloneqq \frac{1}{2} \zeta \dot{\bm{\varepsilon}}'(\bv) + \zeta (\tr(\dot{\bm{\varepsilon}}(\bv))) \bI - \frac{P'}{2} \bI,
\end{equation}
The viscosity $\zeta$ and the pressure term $P'$ are given by
\begin{equation*}
    \begin{aligned}
        \zeta
        &\coloneqq \frac{P + T}{2 \max(\shadedDelta_P(\bv), \shadedDelta_{\min})}, \twhere \shadedDelta_P(\bv) \coloneqq \frac{1}{2} \dot{\bm{\varepsilon}}' : \dot{\bm{\varepsilon}}' \taswellas \shadedDelta_{\min} = 2 \times 10^{-9}, \tand\\
        P' 
        &\coloneqq P - T.
    \end{aligned}
\end{equation*}
In the above, the ``$:$'' in $\shadedDelta_P(\bv)$ denotes the Frobenius inner product between tensors, and $e \ge 1$ is the ratio of the main axes of the elliptic yield curve. 
Following \cite{KBH:10}, the ice strength $P$ and the tensile strength~$T$ are modeled as
\begin{equation}\label{eq:ice strength & tensile strength}
    P = h \, P^* \exp(-c^*(1 - A)), \qquad T = k_{\rt} h \, P^* \exp(-c^*(1 - A))=k_{\rt}P,
\end{equation}
where $P^* > 0$ is an ice strength parameter, $c^* > 0$ is an ice concentration parameter, and $k_{\rt} \in [0,1]$ is the isotropic tensile strength parameter. 
As we will see below, the case $k_{\rt} \equiv 1$ poses additional analytical difficulties for the long-term dynamics as well as the existence of time-periodic solutions.
For convenience of the reader, we briefly compare the stress tensor from the present landfast ice model to that of Hibler's model.
In fact, in Hibler's model, there is no tensile strength, i.e., this model corresponds to the case $k_{\rt} = 0$.
Thus, for $\shadedDelta_P(\bv)$ as well as $\shadedDelta_{\min}$ as made precise above, and for the ice strength $P$ from~\eqref{eq:ice strength & tensile strength}, the viscosity $\zeta_\mathrm{Hib}$ is given by $\zeta_\mathrm{Hib} = \frac{P}{2 \max(\shadedDelta_P(\bv), \shadedDelta_{\min})}$, and the resulting stress tensor is
\begin{equation*}
    \bm{\sigma}_\mathrm{Hib} = \frac{1}{2} \zeta_\mathrm{Hib} \dot{\bm{\varepsilon}}'(\bv) + \zeta_\mathrm{Hib} (\tr(\dot{\bm{\varepsilon}}(\bv))) \bI - \frac{P}{2} \bI.
\end{equation*}
The associated internal ice stress in Hibler's model thus is $f_{\sigma_\mathrm{Hib}} = \div \bm{\sigma}_\mathrm{Hib}$.

To ensure a smooth transition between the viscous closure, $\shadedDelta_{\min}$, and plastic regimes, $\shadedDelta_P(\bv)$, we follow~\cite{KHLFG:00, MR:17} and define the regularized strain rate invariant by
\begin{equation}\label{eq:reg strain rate}
    \shadedDelta(\dot{\bm{\varepsilon}}) \coloneqq \sqrt{\shadedDelta_P(\bv) + \shadedDelta_{\min}^2}. 
\end{equation}

The atmospheric $f_{\ra}$ and oceanic drag $f_{\ro}$  are modeled as
\begin{equation*}
    f_{\ra} = C_{\ra} \rho_{\ra} \| \bv_{\ra} \|_2 \bv_{\ra} \tand f_{\ro} = C_{\ro} \rho_{\ro} \| \bv - \bv_{\ro} \|_2 (\bv - \bv_{\ro}),
\end{equation*}
where $C_{\ra}$, $C_{\ro} > 0$ are the oceanic and atmospheric drag coefficients, $\rho_{\ra}$, $\rho_{\ro} > 0$ are the oceanic and atmospheric densities, and $\bv_{\ra}$, $\bv_{\ro}$ represent the velocity fields of near surface atmospheric and oceanic flows, respectively.

The forces due to the Coriolis force $f_{\rc}$ and the changing sea surface height, $f_{\rsh}$, are given in their common form by
\begin{equation*}
    f_{\rc} = -\rho h f \mathbf{k} \times \bv \tand f_{\rsh} = -\rho h g \nabla H_{\rd},
\end{equation*}
with Coriolis parameter $C_{\cor} > 0$, the unit upward pointing unit vector $\mathbf{k}$, the gravity constant $g$, and the surface height $H_{\rd}$. 
We follow a common approach and approximate the changing sea-surface height by 
\begin{equation*}
    f_{\rsh} \approx C_{\cor} \bk \times \bv_\ro.
\end{equation*}

The basal stress $f_{\rb}$ is given by
\begin{equation}\label{eq:basal stress}
    f_{\rb} = 
    \begin{cases}
        0, & \text{if } h \leq h_{\crit}, \\
        k_2 \left( \frac{-\bv}{|\bv| + v_0} \right)(h - h_{\crit}) \exp\left(-\alpha_{\rb}(1 - A)\right), & \text{if } h > h_{\crit},
    \end{cases}
\end{equation}
where $h$ is the mean thickness in a grid cell (or volume per unit area), $h_{\crit}$ the critical mean thickness,  $k_2$ a free parameter that determines the maximum basal stress, $A$ the ice concentration, $v_0$ a small velocity parameter, and $\alpha_{\rb}$ the basal stress ice concentration parameter. 
The inclusion of the basal drag, $f_{\rb}$ in~\eqref{eq:mom} indicates that the basal stress is zero when $h \leq h_{\crit}$, or, in other words, that the parameterized ridge is not deep enough to reach the seafloor. 
We emphasize again that the basal stress $f_{\rb}$ is not present in the momentum equation of Hibler's model \cite{Hib:79}.
All parameters of the problem are collected in \autoref{tab:1}.
\begin{table}[h!]
\centering
\begin{tabular}{lll}
\textbf{Symbol} & \textbf{Definition} & \textbf{Value} \\
\hline
$e$      & Ellipse aspect ratio & $ 2.0$ \\
$k_{\rt}$    & Isotropic tensile strength parameter & $0.15$ \\
$h_{\crit}$    & Critical thickness & $2 \, \mathrm{m}$ \\
$k_2$    & Maximum basal stress parameter & $5\,\mathrm{N\,m^{-3}}$ \\
$\alpha_{\rb}$ & Basal stress ice concentration parameter & $20$ \\
$v_0$    & Basal stress velocity parameter & $5 \times 10^{-8}\,\mathrm{m\,s^{-1}}$ \\
$\rho$   & Sea-ice density & $900\,\mathrm{kg\,m^{-3}}$ \\
$\rho_{\ra}$ & Air density & $1.3\,\mathrm{kg\,m^{-3}}$ \\
$\rho_{\ro}$ & Water density & $1026\,\mathrm{kg\,m^{-3}}$ \\
$C_{\ra}$ & Air drag coefficient & $1.2 \times 10^{-3}$ \\
$C_{\ro}$ & Water drag coefficient & $5.5 \times 10^{-3}$ \\
$C_{\cor}$      & Coriolis parameter & $1.46 \times 10^{-4}\,\mathrm{s^{-1}}$ \\
$P^\ast$ & Ice strength parameter & $27.5 \times 10^{3}\,\mathrm{N\,m^{-2}}$ \\
$c^*$      & Ice concentration parameter & $20$ \\
\end{tabular}
\caption{Model parameters used in the simulations in \autoref{sec:numerical simulations}.\label{tab:1}}
\end{table}

The sea-ice system is completed by prescribing the initial and boundary conditions as
\begin{equation*}
    \bv(0) = \bv_0, \enspace h(0) = h_0 \enspace \text{and } A(0) = A_0, \ton G
\end{equation*}
as well as
\begin{equation*}
    \bv = 0 \tand \del_\nu h = \del_\nu A = 0, \enspace \text{on} \enspace \del G \times (0,T).
\end{equation*}

\begin{rem}
The homogeneous Neumann boundary conditions for h and A are associated with the parabolic regularization of the corresponding transport equations due to the introduction of the terms $d_\rh \Delta h$ and $d_\rA \Delta A$, respectively. 
We emphasize that these stabilized transport equations were already introduced in the original work by Hibler \cite[Eqs.~(13)--(14)]{Hib:79}.
Since the diffusion terms are introduced only as an artificial regularization, it is natural to require that they do not induce an additional flux through the boundary. 
In view of the no-slip condition for the sea-ice velocity $\bv = 0$ on $\del G$, the zero-flux conditions for the regularized balance laws reduce precisely to $\del_\nu h = \del_\nu A = 0$.
In particular, these boundary conditions preserve the integral balance relations for $h$ and $A$ and ensure that the artificial diffusion contributes only dissipative bulk terms in the corresponding energy estimates.
These boundary conditions have already been used in the analysis of a regularized version of Hibler's sea-ice model in \cite{BDHH:22}.
\end{rem}

\section{Linear theory and local well-posedness}\label{sec:well-posedness}

In this section, we address the local-in-time strong well-posedness of the landfast ice model as introduced in \autoref{sec:landfast ice model}.
After stating the corresponding result, \autoref{thm:loc strong well-posedness}, in \autoref{ssec:lin theory}, we address the linear theory by showing that the differential operator associated with the (regularized) updated strain rate stress tensor from \eqref{eq:stress tensor}, see also \eqref{eq:reg strain rate} for the regularization, admits maximal $\rL^p$-$\rL^q$-regularity, paving the way for the maximal $\rL^p$-regularity of the complete linearized system.
In \autoref{ssec:nonlin ests & proof of local wp}, we then establish suitable nonlinear estimates, where we especially focus on the basal stress $f_{\rb}$, and we then use quasilinear methods, see \autoref{lem:loc wp quasilin ACP}, to deduce the local well-posedness result.
We finally discuss the local strong well-posedness of the system without parabolic regularization in the balance laws in \autoref{ssec:rems on parabolic-hyperbolic setting}.

We start by making precise the functional analytic setting.
In fact, the ground space $\rX_0$ and regularity space $\rX_1$ are given by
\begin{equation}\label{eq:ground & regularity space}
    \rX_0 \coloneqq \rL^q(G)^2 \times \rL^q(G) \times \rL^q(G) \tand \rX_1 \coloneqq \rW^{2,q}(G)^2 \cap \rW_0^{1,q}(G)^2 \times \rW_\rN^{2,q}(G) \times \rW_\rN^{2,q}(G),
\end{equation}
where $\rW_\rN^{2,q}(G) \coloneqq \{f \in \rW^{2,q}(G) : \del_\nu f = 0 \text{ on } \del G\}$.
For $p$, $q \in (1,\infty)$ and $\mu \in (\frac{1}{p},1]$, we also introduce the time trace space $\rX_{\gamma,\mu} = (\rX_0,\rX_1)_{\mu - \nicefrac{1}{p},p}$.
If
\begin{equation}\label{eq:cond on p, q and mu}
    \frac{1}{2} + \frac{1}{p} + \frac{1}{q} < \mu,
\end{equation}
this space takes the form 
\begin{equation*}
    \rX_{\gamma,\mu} = \rB_{qp,\rD}^{2(\mu - \frac{1}{p})}(G)^2 \times \rB_{qp,\rN}^{2(\mu - \frac{1}{p})}(G) \times \rB_{qp,\rN}^{2(\mu - \frac{1}{p})}(G),
\end{equation*}
where the subscripts $_\rD$ and $_\rN$ indicate Dirichlet and Neumann boundary conditions, respectively.
Note that \eqref{eq:cond on p, q and mu} ensures that the trace and the normal derivative are well-defined.
Indeed, if $p$, $q$ and $\mu$ satisfy~\eqref{eq:cond on p, q and mu}, it follows from classical theory, see, e.g., \cite[Thm.4.6.1]{Tri:78}, that there exists $\alpha > 0$ sufficiently small such that
\begin{equation*}
    \rX_{\gamma,\mu} \hookrightarrow \rC^{1,\alpha}(\overline{G})^4.
\end{equation*}
As we will also see in \autoref{ssec:lin theory}, the reason for making this assumption on $p$, $q$ and $\mu$ is to guarantee that the coefficients of the differential operator are H\"older-continuous, enabling us to prove the boundedness of the $\cH^\infty$-calculus and thus also the maximal $\rL^p$-regularity of the fast ice operator.

To make sure that the $h$ and $A$ take values in their physically reasonable ranges, and that the mean ice thickness stays uniformly bounded from below, we define the open subset $V_\mu \subset \rX_{\gamma,\mu}$ by
\begin{equation}\label{eq:open set V_mu}
    V_\mu \coloneqq \{\bu = (\bv,h,A) \in \rX_{\gamma,\mu} : h > \kappa \tand A \in (0,1+\delta)\}.
\end{equation}
Here $\delta > 0$ was introduced to allow for $A = 1$, which corresponds to the case of thick ice in the complete control area.
In the sequel, we also use $\bu = (\bv,h,A)$ to denote the principal variable.

We provide the complete system of equations below.
As we assume that $h \ge \kappa$ for a sufficiently small parameter $\kappa > 0$, we may divide the momentum equation by the ice mass $\rho h$, leading to
\begin{equation}\label{eq:landfast ice model}
    \left\{
    \begin{aligned}
        \del_t \bv + (\bv \cdot \nabla) \bv
        &= \frac{1}{\rho h} \div(\bsigma) - C_{\cor} \bk \times (\bv - \bv_\ro)\\
        &\quad + \frac{1}{\rho h}\bigl(f_\ra + f_\ro(\bv) + f_\rb(\bv,h,A)\bigr), &&\tin G \times (0,T),\\
        \del_t h + \div(\bv h)
        &= d_\rh \Delta h, &&\tin G \times (0,T),\\
        \del_t A + \div(\bv A)
        &= d_\rA \Delta A, &&\tin G \times (0,T),\\
        \bv = 0, \enspace \del_\nu h&= \del_\nu A = 0, &&\ton \del G \times (0,T),\\
        \bv(0) = \bv_0, \enspace h(0)
        &= h_0, \enspace A(0) = A_0, &&\tin G,
    \end{aligned}
    \right.
\end{equation}
where for the constitutive relation for $\bsigma$, we refer to \eqref{eq:stress tensor}--\eqref{eq:reg strain rate}.

Before stating the first main result of this paper, we invoke the concept of time weights.
In fact, for a time interval $(0,T)$, where $0 < T \le \infty$, and for $\rX_i$, $i = 0,1$, we define 
\begin{equation*}
    \rL_\mu^p(0,T;\rX_i) \coloneqq \left\{f \colon (0,T) \to \rX_i : t^{1-\mu} f \in \rL^p(0,T;\rX_i)\right\}, \twith \| f \|_{\rL_\mu^p(0,T;\rX_i)} \coloneqq \left(\int_0^T \| t^{1-\mu} f \|_{\rX_i}^p \srd t\right)^{\frac{1}{p}}.
\end{equation*}
Likewise, we define
\begin{equation*}
    \begin{aligned}
        \rW_\mu^{1,p}(0,T;\rX_0) 
        &\coloneqq \left\{f \in \rL_\mu^p(0,T;\rX_0) \cap \rW_{\mathrm{loc}}^{1,1}(0,T;\rX_0) : \del_t f \in \rL_\mu^p(0,T;\rX_0)\right\}, \twith\\
        \| f \|_{\rW_\mu^{1,p}(0,T;\rX_0)}
        &\coloneqq \left(\| f \|_{\rL_\mu^p(0,T;\rX_0)}^p + \| \del_t f \|_{\rL_\mu^p(0,T;\rX_0)}^p\right)^{\frac{1}{p}}.
    \end{aligned}
\end{equation*}
We are now in position to state the first main result of this paper on the local strong well-posedness of the landfast ice model \eqref{eq:landfast ice model}.
The contribution of the local theory lies in establishing the operator-theoretic hypotheses for the landfast-ice system, in particular the ellipticity and maximal \(\rL^p\)-regularity properties of the fast-ice operator, and in treating the threshold-dependent basal-stress nonlinearity within the quasilinear framework.

\begin{thm}\label{thm:loc strong well-posedness}
Let $p$, $q \in (1,\infty)$ and $\mu \in (\frac{1}{p},1]$ such that \eqref{eq:cond on p, q and mu} is satisfied.
Moreover, for $V_\mu$ as made precise in \eqref{eq:open set V_mu}, consider $\bu_0 = (\bv_0,h_0,A_0) \in V_\mu$, and for some given $T > 0$, assume that $\bv_a$ and $\bv_{\ro}$ satisfy $\bv_a$, $\bv_{\ro} \in \rL^{2p}(0,T;\rL^{2q}(G)^2)$.

Then there exist $T' = T'(\bu_0) \in (0,T]$ and $r = r(\bu_0) > 0$ with $\overline{B}_{\rX_{\gamma,\mu}}(\bu_0,r) \subset V_\mu$, where $\overline{B}_{\rX_{\gamma,\mu}}(\bu_0,r)$ denotes the closed ball with center $\bu_0$ and radius $r > 0$ in  $\rX_{\gamma,\mu}$, such that \eqref{eq:landfast ice model} admits a unique solution
\begin{equation*}
    \bu = (\bv,h,A) \in \rW_\mu^{1,p}(0,T';\rX_0) \cap \rL_\mu^p(0,T';\rX_1) \cap \rC([0,T'];V_\mu) \eqqcolon \bE_{1,\mu}(0,T')
\end{equation*}
on $[0,T']$ for any initial value $\bw_1 \in \overline{B}_{\rX_{\gamma,\mu}}(\bu_0,r)$.
The solution also has the following properties:
\begin{enumerate}[(a)]
    \item There is a constant $c = c(\bu_0) > 0$ such that for all $\bw_1$, $\bw_2 \in \overline{B}_{\rX_{\gamma,\mu}}(\bu_0,r)$, we have
    \begin{equation*}
        \| \bu(\cdot;\bw_1) - \bu(\cdot;\bw_2) \|_{\bE_{1,\mu}(0,T')} \le c \cdot \| \bw_1 - \bw_2 \|_{\rX_{\gamma,\mu}},
    \end{equation*}
    i.e., the solution $\bu$ depends continuously on the initial data.
    \item For every $\tau \in (0,T')$, we find that 
    \begin{equation*}
        \bu \in \bE_1(\tau,T') \coloneqq \bE_{1,1}(\tau,T') \hookrightarrow \rC([\tau,T'];\rX_\gamma),
    \end{equation*}
    so the solution regularizes instantaneously.
    \item If we assume that $\bv_a$, $\bv_{\ro} \in \rL^{2p}(0,\infty;\rL^{2q}(G)^2)$, then the solution $\bu$ exists on a maximal time interval $J(\bu_0) = [0,t_+(\bu_0))$ that is characterized by the following alternatives:
    \begin{enumerate}[(i)]
        \item global existence, so $t_+(\bu_0) = \infty$,
        \item $\lim_{t \to t_+(\bu_0)} \dist_{\rX_{\gamma,\mu}}(\bu(t),\del V_\mu) = 0$, or
        \item the limit $\lim_{t \to t_+(\bu_0)} \bu(t)$ does not exist in $\rX_{\gamma,\mu}$.
    \end{enumerate}
\end{enumerate}
\end{thm}

In the following two subsections, we carry out the proof of \autoref{thm:loc strong well-posedness}.

\subsection{Linear theory}\label{ssec:lin theory}
\ 

In this subsection, we address the linear theory.
A key step here is to establish suitable ellipticity properties of the operator associated with the internal ice stress in the landfast ice.
Note that due to the tensile strength $T$ as introduced in \eqref{eq:ice strength & tensile strength}, it is different from Hibler's sea-ice model.

For convenience, we briefly recall the regularized stress tensor as introduced in \autoref{sec:landfast ice model}.
By a slight abuse of notation, we still denote the viscosity with the regularized strain rate $\shadedDelta(\dbeps)$ from \eqref{eq:reg strain rate} by $\zeta$, i.e.,
\begin{equation*}
    \zeta = \frac{\tP}{2 \shadedDelta(\dbeps)}, \twhere \tP \coloneqq P + T = (1+k_{\rt}) h P^* \exp(-c^*(1-A)).
\end{equation*}
The corresponding regularized stress tensor then takes the form
\begin{equation*}
    \bsigma = \frac{1}{2} \zeta \dbeps'(\bv) + \zeta (\tr(\dbeps(\bv))) \bI - \frac{P'}{2} \bI, \twhere P' \coloneqq P - T = (1-k_{\rt}) h P^* \exp(-c^*(1-A)).
\end{equation*}
In order to deduce the shape of the quasilinear second order operator arising from $\div \bsigma$, we introduce a matrix $\bS \colon \bR^{2 \times 2} \to \bR^{2 \times 2}$ such that
\begin{equation*}
    \bS \dbeps = \begin{pmatrix}
    \left(1 + \frac{1}{e^2}\right) \dbeps_{11} + \left(1 - \frac{1}{e^2}\right) \dbeps_{22} & \frac{1}{e^2} \left(\dbeps_{12} + \dbeps_{21}\right) \\
    \frac{1}{e^2} \left(\dbeps_{12} + \dbeps_{21}\right) & \left(1 - \frac{1}{e^2}\right) \dbeps_{11} + \left(1 + \frac{1}{e^2}\right) \dbeps_{22}
    \end{pmatrix}.
\end{equation*}
Upon identifying $\dbeps \in \bR^{2 \times 2}$ with the vector $(\dbeps_{11},\dbeps_{12},\dbeps_{21},\dbeps_{22})^\top \in \bR^4$, the action of $\bS$ to $\dbeps$ amounts to the multiplication by the matrix
\begin{equation*}
    \bS = \left(\bS_{ij}^{kl}\right) = \begin{pmatrix}
	    1+\frac{1}{e^2} & 0 & 0 & 1-\frac{1}{e^2} \\ 
		0 & \frac{1}{e^2}  & \frac{1}{e^2} & 0   \\
		0 & \frac{1}{e^2} & \frac{1}{e^2} & 0  \\
	     1-\frac{1}{e^2} & 0 & 0 & 1 + \frac{1}{e^2} 
	\end{pmatrix}.
\end{equation*}
With $\tS(\dbeps,\tP) \coloneqq \frac{\tP}{2} \frac{\bS \dbeps}{\shadedDelta(\dbeps)}$, the stress tensor $\bsigma$ admits the representation $\bsigma = \tS(\dbeps,\tP) - \frac{P'}{2}\bI$.
The \textit{fast ice operator} $\bAFI$ is then defined by 
\begin{equation*}
    \bAFI \bv \coloneqq \frac{1}{\rho h} \div \tS(\dbeps,\tP) = \frac{1}{\rho h} \div\left(\frac{\tP}{2} \frac{\bS \dbeps}{\shadedDelta(\dbeps)}\right).
\end{equation*}
For the following representation of the landfast operator, we interpret it as a second-order differential operator in non-divergence form.
Straightforward calculations reveal that
\begin{equation}\label{eq:fast ice op}
    (\bAFI \bv)_i \coloneqq -\sum \limits_{j,k,l=1}^2 \frac{\tP}{2 \rho h}\frac{1}{\shadedDelta(\dbeps)}\left(\bS_{ij}^{kl} - \frac{1}{\shadedDelta^2 (\dbeps)}(\bS \dbeps)_{ik} (\bS \dbeps)_{jl} \right) \rD_k \rD_l \bv_j + \frac{1}{2 \rho h \shadedDelta(\dbeps)}\sum \limits_{j=1}^2 (\del_j \tP) (\bS \dbeps)_{ij},
\end{equation}
where $i = 1,2$ and $\rD_m = - \mri \del_m$.
The coefficients of the principal part of $\bAFI$ are given by
\begin{equation}\label{eq:coeffs principal part fast ice op}
    a_{ij}^{kl}(\dbeps,\tP) \coloneqq -\frac{\tP}{2 \rho h} \frac{1}{\shadedDelta(\dbeps)} \left(\bS_{ij}^{kl} - \frac{1}{\shadedDelta^2(\dbeps)}(\bS \dbeps)_{ik} (\bS \dbeps)_{jl}\right).
\end{equation}
For $\bu_0 = (\bv_0,h_0,A_0) \in \rC^1(\oG)^2 \times \rC(\oG) \times \rC(\oG)$ with $h_0 \ge \kappa$, we freeze the coefficients at $\bu_0$.
The \textit{linearized fast ice operator} then takes the form
\begin{equation}\label{eq:lin fast ice op}
    \begin{aligned}
        [\bAFI(\bu_0) \bv]_i
        &= \sum_{j,k,l=1}^2 a_{ij}^{kl}(\dbeps(\bv_0),\tP(h_0,A_0)) \rD_k \rD_l \bv_j + \frac{1}{2 \rho h_0 \shadedDelta(\dbeps(\bv_0))} \sum_{j=1}^2 (\del_j \tP(h_0,A_0)) (\bS \dbeps(\bv))_{ij}.
    \end{aligned}
\end{equation}
More precisely, $\bAFI(\bu_0)$ denotes the second-order operator obtained by freezing the coefficients at $\bu_0$, and it should not be confused with the full Fr\'echet linearization of the nonlinear mapping $\bu \mapsto \bAFI(\bu) \bu$, which contains additional lower-order terms, and which will be discussed in the context of the stability of equilibria in \autoref{sec:glob wp close to equil}.
For more details on the derivation of the Hibler operator, which is closely related to the present fast ice operator, we also refer to \cite[Sec.~3]{BDHH:22} and \cite[Sec.~3.3]{Bra:24}.

We now analyze the linearized fast ice operator $\bAFI(\bu_0)$ from \eqref{eq:lin fast ice op}.
In the light of \autoref{lem:max reg boundary value problem} on the boundedness of the $\Hinfty$-calculus of the $\rL^q$-realization and \autoref{lem:strong normal ell implies Lopatinskii-Shapiro} on the Lopatinskii-Shapiro condition, the task is to establish the strong normal ellipticity as made precise in \autoref{def:strong normal ell}.
This is what we will address in the following.
We start with the strong ellipticity and the parameter-ellipticity.

\begin{prop}\label{prop:ell of fast ice op}
Let $\bu_0 \in \rC^1(\oG)^2 \times \rC(\oG) \times \rC(\oG)$ be such that $h_0 \ge \kappa$.
Then for all $x \in \oG$, the principal part of the negative linearized fast ice operator $-\bAFI(\bu_0)$ as introduced in \eqref{eq:lin fast ice op} is strongly elliptic and parameter-elliptic of angle $\phi_{-\bAFI(\bu_0)} = 0$.
\end{prop}

\begin{proof}
First, we recall the principal part of $-\bAFI(\bu_0)$ given by $-\bAFI_{\#}(x,\xi) = \sum_{k,l=1}^2 -a_{ij}^{kl}(x) \xi_k \xi_l$, where the coefficients $a_{ij}^{kl}$ were introduced in \eqref{eq:coeffs principal part fast ice op}.
It can be verified that the $a_{ij}^{kl}$ satisfy the symmetries 
\begin{equation}\label{eq:symms coeffs fast ice op}
    a_{ij}^{kl} = a_{ji}^{lk} = a_{kl}^{ij} = a_{kj}^{il} = a_{il}^{kj} = a_{lk}^{ji}.
\end{equation}
Thus, the symbol of the principal part can be represented as
\begin{equation*}
    \bAFI_{\#}(x,\xi) = \begin{pmatrix}
    a_{11}^{11} \xi_1 ^2 + 2 a_{11}^{12} \xi_1 \xi_2 + a_{11}^{22} \xi_2 ^2 & a_{11}^{12} \xi_1 ^2 + (a_{12}^{12} + a_{11}^{22}) \xi_1 \xi_2 + a_{12}^{22} \xi_2 ^2 \\
    a_{11}^{12} \xi_1 ^2 + (a_{12}^{12} + a_{11}^{22}) \xi_1 \xi_2 + a_{12}^{22} \xi_2 ^2 & a_{11}^{22} \xi_1 ^2 + 2 a_{12}^{22} \xi_1 \xi_2 + a_{22}^{22} \xi_2 ^2
    \end{pmatrix}.
\end{equation*}
To simplify notation, for $\bd \in \bR^{2 \times 2}$, we introduce $\bd_\rI \coloneqq \bd_{11} + \bd_{22}$, $\bd_\rII \coloneqq \bd_{11} - \bd_{22}$ and $\bd_{\rIII} \coloneqq \frac{\bd_{12} + \bd_{21}}{2}$.
As in \cite[(4.4) and (4.5)]{BDHH:22}, we then obtain 
\begin{equation}\label{eq:tri^2}
    \bd^\top \bS \bd = \bd_\rI^2 + \frac{1}{e^2}(\bd_\rII^2 + 4 \bd_\rIII^2) \eqqcolon \shadedDelta^2(\bd) \tand (\bd^\top (\bS \dbeps))^2 \le \shadedDelta^2(\bd) \shadedDelta^2(\dbeps).
\end{equation}
By virtue of $h_0 \ge \kappa > 0$, we find that $\frac{\tP(h_0,A_0)}{2 \rho h_0 \shadedDelta^3(\dbeps)}$ is real-valued, bounded, continuous, and bounded from below by a constant $c > 0$, so the previous estimate leads to
\begin{equation}\label{eq:aux est ell}
    \sum_{i,j,k,l=1}^2 -a_{ij}^{kl} \bd_{ik} \bd_{jl} = \frac{\tP(h_0,A_0)}{2 \rho h_0 \shadedDelta^3(\dbeps)} \bigl(\shadedDelta^2(\dbeps) \bd^\top \bS \bd - (\bd^\top \bS \dbeps)^2\bigr) \ge c \shadedDelta_{\min}^2 \shadedDelta^2(\bd).
\end{equation}
We now capitalize on \eqref{eq:aux est ell} to show the strong ellipticity.
Let $\xi \in \bR^2$ and $\eta \in \bC^2$ such that $|\xi| = |\eta| = 1$, and for $i=1,2$, define $\eta_i \coloneqq x_i + \mri y_i$.
Thanks to the symmetries of the coefficients \eqref{eq:symms coeffs fast ice op}, and using the estimate from \eqref{eq:aux est ell} joint with $|\xi| = 1$ and $|\eta| = 1$, which imply $\shadedDelta^2(\xi \otimes x) = (\xi \cdot x)^2 + \frac{1}{e^2} |x|^2$ and $\shadedDelta^2(\xi \otimes x) + \shadedDelta^2(\xi \otimes y) \ge \frac{1}{e^2}$, respectively, we find that
\begin{equation*}
    \begin{aligned}
        \re(-\bAFI_{\#}(x,\xi) \eta | \eta) 
        &= \sum_{i,j,k,l=1}^2 -a_{ij}^{kl} (\xi \otimes x)_{jl} (\xi \otimes x)_{ik} -a_{ij}^{kl} (\xi \otimes y)_{jl} (\xi \otimes y)_{ik}\\
        &\ge c \shadedDelta_{\min}^2 \bigl((\shadedDelta^2(\xi \otimes x) + \shadedDelta^2(\xi \otimes y) \bigr) \ge \frac{c \shadedDelta_{\min}^2}{e^2}.
    \end{aligned}
\end{equation*}
This shows that $-\bAFI(\bu_0)$ is strongly elliptic. 
By the symmetry of $-\bAFI_{\#}$, we get $\sigma(-\bAFI_{\#}(x,\xi)) \subset \bR_+$ for all $x \in \oG$ and $\xi \in \bR^2$ with $|\xi| = 1$.
Thus, $-\bAFI(\bu_0)$ is parameter-elliptic with $\phi_{-\bAFI(\bu_0)} = 0$.
\end{proof}

Next, we discuss the strong normal ellipticity.

\begin{lem}\label{lem:strong normal ell of the Hibler op}
Let $\bu_0 \in \rC^1(\oG)^2 \times \rC(\oG) \times \rC(\oG)$ be such that $h_0 \ge \kappa$, and recall from \eqref{eq:coeffs principal part fast ice op} the coefficients $a_{ij}^{kl}(\bu_0)$ of the principal part of the linearized fast ice operator~$\bAFI(\bu_0)$.
Then for $x \in \dG$, $\xi$, $\nu \in \bR^2$ such that $|\xi| = |\nu| = 1$ as well as~$(\xi|\nu) = 0$, and $u$, $v \in \bC^2$, it holds that
\begin{equation*}
    \begin{aligned}
        \re\left(\sum_{i,j,k,l=1}^2 -a_{ij}^{kl}(\bu_0)(\xi_l u_j - \nu_l v_j) \overline{(\xi_k u_i - \nu_k v_i)}\right) 
        &\ge 0, \tand\\
        \re\left(\sum_{i,j,k,l=1}^2 -a_{ij}^{kl}(\bu_0)(\xi_l u_j - \nu_l v_j) \overline{(\xi_k u_i - \nu_k v_i)}\right) 
        &> 0, \tif \im(u|v) \neq 0.
    \end{aligned}
\end{equation*}
In particular, the negative fast ice operator is strongly normally elliptic in the sense of \autoref{def:strong normal ell}.
\end{lem}

\begin{proof}
We consider $x \in \dG$, $\xi$, $\nu \in \bR^2$ with $|\xi| = |\nu| = 1$ and $(\xi | \nu) = 0$ as well as $u$, $v \in \bC^2$.
In the sequel, we also write $u_i = x_i + \mri y_i$ and $v_i = \tilde{x}_i + \mri \tilde{y}_i$.
The symmetries of the $a_{ij}^{kl}$ from \eqref{eq:symms coeffs fast ice op} and \eqref{eq:aux est ell} yield
\begin{equation*}
    \re \Bigl(\sum \limits_{i,j,k,l=1}^2 -a_{ij}^{kl}(\xi_l u_j - \nu_l v_j)\overline{(\xi_k u_i - \nu_k v_i)}\Bigr)
    \ge c \shadedDelta_{\min}^2 \Bigl((\shadedDelta^2(\xi \otimes x - \nu \otimes \tilde{x}) + \shadedDelta^2(\xi \otimes y - \nu \otimes \tilde{y}) \Bigr) \ge 0,
\end{equation*}
so the first part of the assertion is implied.
With regard to the second part, we focus on the case of equality with zero in the preceding estimate and infer that 
\begin{equation}\label{eq:aux rel strong normal ell}
    \im(u \vert v) = \tilde{x}_1 y_1 - x_1 \tilde{y}_1 +  \tilde{x}_2 y_2 - x_2 \tilde{y}_2 = 0.
\end{equation}
For $\bd \in \bR^{2 \times 2}$, it follows from $\shadedDelta^2(\bd) = 0$ that $\bd_{11} = \bd_{22} = 0$.
Thus, the equality of the above estimate with zero results in
\begin{equation}\label{eq:second aux rel strong normal ell}
    \xi_1 x_1 - \nu_1 \tilde{x}_1 =  \xi_2 x_2 - \nu_2 \tilde{x}_2=  \xi_1 y_1 - \nu_1 \tilde{y}_1=  \xi_2 y_2 - \nu_2 \tilde{y}_2 = 0.
\end{equation}
By $| \xi | = 1$, we have either $\xi_1 \neq 0$ or $\xi_2 \neq 0$.
Without loss of generality, suppose that $\xi_1 \neq 0$, since the other case can be handled analogously.
Then $(\xi | \nu) = 0$ and $| \nu | = 1$ require that $\nu_2 \neq 0$.
By \eqref{eq:second aux rel strong normal ell}, we get $x_1 = \frac{\nu_1}{\xi_1} \tilde{x}_1$, $\tilde{x}_2 = \frac{\xi_2}{\nu_2} x_2$, $y_1 = \frac{\nu_1}{\xi_1} \tilde{y}_1$ and $\tilde{y}_2 = \frac{\xi_2}{\nu_2} y_2$.
An insertion into \eqref{eq:aux rel strong normal ell} reveals the validity of the second part of the assertion.
\end{proof}

We are now in position to discuss the maximal $\rL^p$-regularity of the $\rL^q$-realization of the fast ice operator defined as follows:
For $\bu_0 \in \rC^1(\oG)^2 \times \rC(\oG) \times \rC(\oG)$ with $h_0 \ge \kappa$, the $\rL^q$-realization of the linearized fast ice operator $\bAFI(\bu_0)$ subject to Dirichlet boundary conditions on $\dG$ is defined by
\begin{equation}\label{eq:Lq-realization of the fast ice op}
        \left[\AFID(\bu_0)\right] \bv \coloneqq \left[\bAFI(\bu_0)\right] \bv, \twith\\
        \rD(\AFID(\bu_0)) \coloneqq \rW^{2,q}(G)^2 \cap \rW_0^{1,q}(G)^2.
\end{equation}

The proposition below asserts the boundedness of the $\cH^\infty$-calculus of $-\AFID(\bu_0)$ up to a shift.
For an overview of the concepts in the proposition below, we refer to \autoref{sec:op theoret concepts}.

\begin{prop}\label{prop:Hinfty fast ice op}
Let $q \in (1,\infty)$, consider $\bu_0 \in \rC^{1,\alpha}(\oG)^2 \times \rC^\alpha(\oG) \times \rC^\alpha(\oG)$ for some $\alpha > 0$, and with $h_0 \ge \kappa$, and recall $\AFID(\bu_0)$ from \eqref{eq:Lq-realization of the fast ice op}.  
Then there is $\omega_1 \in \bR$ so that for all $\omega > \omega_1$, we have $-\AFID(\bu_0) + \omega \in \cH^\infty(\rL^q(G)^2)$ with $\phi_{-\AFID(\bu_0) + \omega}^\infty < \nicefrac{\pi}{2}$.
In particular, $-\AFID(\bu_0) + \omega$ is $\cR$-sectorial on~$\rL^q(G)^2$, with $\cR$-angle $\phi_{-\AFID(\bu_0) + \omega}^\cR \le \phi_{-\AFID(\bu_0) + \omega}^\infty < \nicefrac{\pi}{2}$, so $-\AFID(\bu_0) + \omega$ has maximal $\rL^p$-regularity.
\end{prop}

\begin{proof}
The proof is a consequence of \autoref{lem:max reg boundary value problem} upon noting that the strong normal ellipticity in the sense of \autoref{def:Lopatinskii Shapiro} has been verified in \autoref{lem:strong normal ell of the Hibler op}, which in turn implies the parameter-ellipticity and the validity of the Lopatinskii-Shapiro condition in the present case by \autoref{lem:strong normal ell implies Lopatinskii-Shapiro}.
Moreover, with regard to the form of the top-order coefficients $a_{ij}^{kl}(\dbeps,\tP)$ as introduced in \eqref{eq:coeffs principal part fast ice op}, the assumption that $\bu_0 \in \rC^{1,\alpha}(\oG)^2 \times \rC^\alpha(\oG) \times \rC^\alpha(\oG)$ for some $\alpha > 0$ implies the H\"older-continuity of the top-order coefficients.
Similarly, the boundedness of the lower-order coefficients required for \autoref{lem:max reg boundary value problem} is implied.
\end{proof}

We observe that the assumptions on $\bu_0$ can be relaxed if we are merely interested in the $\cR$-sectoriality, or, equivalently, the maximal $\rL^p$-regularity, of $-\AFID(\bu_0)$, see also \cite[Theorem~8.2]{DHP:03} and the discussion preceding \autoref{lem:max reg boundary value problem}.

\begin{rem}\label{rem:max reg fast ice op}
Let $q \in (1,\infty)$, consider $\bu_0 \in \rC^1(\oG)^2 \times \rC(\oG) \times \rC(\oG)$ with $h_0 \ge \kappa$, and recall $\AFID(\bu_0)$ from \eqref{eq:Lq-realization of the fast ice op}.  
Then there exists $\omega_1 \in \bR$ such that for all $\omega > \omega_1$, the shifted operator $-\AFID(\bu_0) + \omega$ is $\cR$-sectorial on $\rL^q(G)^2$, with $\cR$-angle $\phi_{-\AFID(\bu_0) + \omega} < \nicefrac{\pi}{2}$.
In other words, $-\AFID(\bu_0) + \omega$ has maximal $\rL^p$-regularity on $\rL^q(G)^2$.
\end{rem}

After addressing the fast ice operator, we now tackle the complete linearized problem.
To this end, we invoke the Neumann Laplacian operator that is relevant for the terms associated with the parabolic regularization in the balance laws.
It is defined by
\begin{equation}\label{eq:Neumann Laplacian}
    \Delta_\rN f \coloneqq \Delta f \tfor f \in \rD(\Delta_\rN) \coloneqq \rW_\rN^{2,q}(G).
\end{equation}

For $\bu_0$ as above, the linearized fast ice operator $\AFID(\bu_0)$ as in \eqref{eq:lin fast ice op}, and the Neumann Laplacian as made precise in \eqref{eq:Neumann Laplacian}, the operator matrix under consideration is given by
\begin{equation}\label{eq:op matrix landfast ice}
    \begin{aligned}
        \ALFI(\bu_0) 
        &= \begin{pmatrix}
            -\AFID(\bu_0) & \frac{\del_h P'(h_0,A_0)}{2 \rho h_0} \nabla & \frac{\del_A P'(h_0,A_0)}{2 \rho h_0} \nabla\\
            0 & -d_\rh \Delta_\rN & 0\\
            0 & 0 & -d_\rA \Delta_\rN
        \end{pmatrix}, \twith \rD(\ALFI(\bu_0)) = \rX_1.
    \end{aligned}
\end{equation}

Below, we discuss the properties of $\ALFI(\bu_0)$.

\begin{prop}\label{prop:bdd Hoo-calculus op matrix landfast ice}
Consider $q \in (1,\infty)$ as well as $\bu_0 \in \rC^{1,\alpha}(\oG)^2 \times \rC^\alpha(\oG) \times \rC^\alpha(\oG)$ for some $\alpha > 0$, and with $h_0 \ge \kappa$.
Then there exists some $\omega_0 \in \bR$ such that for all $\omega > \omega_0$, we have $\ALFI(\bu_0) + \omega \in \cH^\infty(\rX_0)$.
In particular, $\ALFI(\bu_0) + \omega$ has maximal $\rL^p$-regularity on $\rX_0$ for all $p \in (1,\infty)$.
\end{prop}

We observe that the maximal regularity of $\ALFI(\bu_0) + \omega$ can be obtained under the same assumptions on $\bu_0$ as in \autoref{rem:max reg fast ice op}.

\begin{proof}[Proof of \autoref{prop:bdd Hoo-calculus op matrix landfast ice}]
The idea to show the boundedness of the $\cH^\infty$-calculus is to capitalize on the respective properties of the fast ice operator $\AFID(\bu_0)$ and the Neumann Laplacian operator $\Delta_\rN$.
Thanks to its ellipticity properties, it follows from \autoref{lem:max reg boundary value problem} that $-\Delta_\rN + \omega \in \cH^\infty(\rL^q(G))$ with $\phi_{-\Delta_\rN + \omega}^\infty < \nicefrac{\pi}{2}$ for all $\omega > 0$.
Thus, from \autoref{prop:Hinfty fast ice op}, we deduce that there is $\omega_0 \ge 0$ so that for all $\omega > \omega_0$, we have 
\begin{equation*}
    A_{\diag} \coloneqq \diag(-\AFID(\bu_0),-d_\rh \Delta_\rN,-d_\rA \Delta_\rN) + \omega \in \cH^\infty(\rX_0), \twith \phi_{A_{\diag}}^\infty < \nicefrac{\pi}{2}.
\end{equation*}
As a consequence, the fractional power domains are isomorphic to the complex interpolation spaces, see, for example, \cite[Sec.~3.3]{PS:16}.
Together with standard interpolation theory, see, e.g., \cite[Ch.~2]{Tri:78}, this yields
\begin{equation*}
    \rD(A_{\diag}^{\nicefrac{1}{2}}) \simeq [\rX_0,\rX_1]_{\nicefrac{1}{2}} = \rW_0^{1,q}(G)^2 \times \rW^{1,q}(G) \times \rW^{1,q}(G).
\end{equation*}
Now, let $\bu \in \rD(A_{\diag}^{\nicefrac{1}{2}})$.
Thanks to the assumptions on $\bu_0$, we find that
\begin{equation*}
    \left\| \frac{\del_h P'(h_0,A_0)}{2 \rho h_0} \nabla h \right\|_{\rL^q(G)} \le C \cdot \| h \|_{\rW^{1,q}(G)} \le C \cdot \| (-d_\rh \Delta_\rN + \omega)^{\nicefrac{1}{2}} h \|_{\rL^q(G)} \le C \cdot \| A_{\diag}^{\nicefrac{1}{2}} \bu \|_{\rX_0}
\end{equation*}
for some (generic) constant $C > 0$.
The other off-diagonal term can be estimated in the exact same way.
Thus, we have shown that the operator $\ALFI(\bu_0)$ is relatively bounded with respect to a fractional power of $A_{\diag}$.
Perturbation theory for the boundedness of the $\cH^\infty$-calculus, see, e.g., \cite[Prop.~13.1]{KW:04}, then yields that $\ALFI(\bu_0) + \omega \in \cH^\infty(\rX_0)$ with $\phi_{\ALFI(\bu_0) + \omega}^\infty < \nicefrac{\pi}{2}$ provided $\omega$ is chosen sufficiently large.
\end{proof}

\subsection{Nonlinear estimates and proof of \autoref{thm:loc strong well-posedness}}\label{ssec:nonlin ests & proof of local wp}
\ 

This subsection is dedicated to proving estimates of the nonlinear terms, and to showing the local strong well-posedness result, \autoref{thm:loc strong well-posedness} by invoking the quasilinear theory as recalled in \autoref{lem:loc wp quasilin ACP}.

For this purpose, we first reformulate \eqref{eq:landfast ice model} as an evolution equation.
In fact, recalling the principal variable $\bu = (\bv,h,A)$ we write
\begin{equation}\label{eq:fast ice as quasilin evol eq}
    \bu'(t) + A^\LFI(\bu(t))\bu(t) = F(\bu(t)), \tfor t > 0, \tand \bu(0) = \bu_0,
\end{equation}
where
\begin{equation*}
    \begin{aligned}
        A^\LFI(\bu) \bu 
        &= \begin{pmatrix}
        -\AFID(\bu) \bv + \frac{\del_h P'(h,A)}{2 \rho h} \nabla h + \frac{\del_A P'(h,A)}{2 \rho h} \nabla A\\
        -d_\rh \Delta_\rN h\\
        -d_\rA \Delta_\rN A
        \end{pmatrix} \tand\\
        F(\bu)
        &= \begin{pmatrix}
            -(\bv \cdot \nabla) \bv - C_{\cor} \bk \times (\bv - \bv_\ro) + \frac{1}{\rho h}\bigl(f_\ra + f_\ro(\bv) + f_\rb(\bv,h,A)\bigr)\\
            -\div(\bv h)\\
            -\div(\bv A)
        \end{pmatrix}.
    \end{aligned}
\end{equation*}

First, we collect the estimates of the nonlinear terms in the lemma below.

\begin{lem}\label{lem:nonlin ests}
Let $p$, $q \in (1,\infty)$ and $\mu \in (\frac{1}{p},1]$ such that \eqref{eq:cond on p, q and mu} holds true, and recall $V_\mu$ from \eqref{eq:open set V_mu} as well as $\rX_0$ and $\rX_1$ from \eqref{eq:ground & regularity space}.
Then for the nonlinear terms $A^\LFI$ and $F$, it follows that
\begin{equation*}
    (A^\LFI,F) \in \rC^{0,1}(V_\mu;\cL(\rX_1,\rX_0) \times \rX_0),
\end{equation*}
i.e., $A^\LFI$ and $F$ are Lipschitz continuous as maps from $V_\mu$ to the space of bounded linear operators~$\cL(\rX_1,\rX_0)$ from~$\rX_1$ to $\rX_0$, and to the ground space $\rX_0$, respectively.
\end{lem}

\begin{proof}
The proof is analogous to the one of \cite[Lemma~6.2]{BDHH:22} for most terms, so we only discuss the term~$f_\rb$ related to the basal stress in more detail.

Thus, let $\bu_i = (\bv_i,h_i,A_i) \in V_\mu$, $i=1,2$.
H\"older's inequality in conjunction with the mean value theorem as well as $h_i > \kappa$ yields that
\begin{equation*}
    \begin{aligned}
        &\quad \left\| \frac{1}{\rho h_1}f_\rb(\bv_1,h_1,A_1) - \frac{1}{\rho h_2}f_\rb(\bv_2,h_2,A_2)\right\|_{\rL^q(G)}\\
        &\le C\left(\| h_1 - h_2 \|_{\rL^\infty(G)} \cdot \| f_\rb(\bv_1,h_1,A_1) \|_{\rL^q(G)} + \| f_\rb(\bv_1,h_1,A_1) - f_\rb(\bv_2,h_2,A_2) \|_{\rL^q(G)}\right).
    \end{aligned}
\end{equation*}
Recalling $f_\rb$ from \eqref{eq:basal stress}, and making use of $\bu_1 \in V_\mu$, we find that
\begin{equation}\label{eq:est of basal stress}
    \| f_\rb(\bv_1,h_1,A_1) \|_{\rL^q(G)} \le C \cdot \| h_1 \|_{\rL^q(G)} \le C \cdot \| \bu_1 \|_{\rX_{\gamma,\mu}}.
\end{equation}
On the other hand, additionally invoking that the conditions on $p$, $q$ and $\mu$ as specified in \eqref{eq:cond on p, q and mu} imply that $\rB_{qp}^{2(\mu - \nicefrac{1}{p})}(G) \hookrightarrow \rL^\infty(G)$, we infer that
\begin{equation}\label{eq:diff of basal stress}
    \begin{aligned}
        &\quad \| f_\rb(\bv_1,h_1,A_1) - f_\rb(\bv_2,h_2,A_2) \|_{\rL^q(G)}\\
        &\le C\left(\| \bv_1 - \bv_2 \|_{\rL^\infty(G)}(\| h_1 \|_{\rL^q(G)} + 1) + \| h_1 - h_2 \|_{\rL^\infty(G)} + (\| h_2 \|_{\rL^q(G)} + 1)\| A_1 - A_2 \|_{\rL^\infty(G)}\right)\\
        &\le C\left(\| \bu_1 \|_{\rX_{\gamma,\mu}} + \| \bu_2 \|_{\rX_{\gamma,\mu}} + 1\right)\| \bu_1 - \bu_2 \|_{\rX_{\gamma,\mu}}.
    \end{aligned}
\end{equation}
Plugging these estimates into the above one, we find that
\begin{equation*}
    \left\| \frac{1}{\rho h_1}f_\rb(\bv_1,h_1,A_1) - \frac{1}{\rho h_2}f_\rb(\bv_2,h_2,A_2)\right\|_{\rL^q(G)} \le C\left(\| \bu_1 \|_{\rX_{\gamma,\mu}} + \| \bu_2 \|_{\rX_{\gamma,\mu}} + 1\right)\| \bu_1 - \bu_2 \|_{\rX_{\gamma,\mu}},
\end{equation*}
showing the desired Lipschitz continuity of $f_\rb$ and thus completing the proof.
\end{proof}

\begin{proof}[Proof of \autoref{thm:loc strong well-posedness}]
Thanks to the maximal $\rL^p$-regularity of $A^\LFI(\bu_0) + \omega$ as revealed in \autoref{prop:bdd Hoo-calculus op matrix landfast ice} and the estimates of the nonlinear terms $A^\LFI(\bu)$ and $F(\bu)$ as shown in \eqref{lem:nonlin ests}, the assertion of \autoref{thm:loc strong well-posedness} is a consequence of quasilinear theory as recalled in \autoref{lem:loc wp quasilin ACP} upon taking into account the shift for the right-hand side $F(\bu)$ as well.
Note that this does not affect the present local-in-time well-posedness result.
This first yields the result for the quasilinear evolution equation \eqref{eq:fast ice as quasilin evol eq} and carries over to the equivalent PDE system \eqref{eq:landfast ice model}.
\end{proof}

\subsection{Local strong well-posedness in the parabolic-hyperbolic setting}\label{ssec:rems on parabolic-hyperbolic setting}
\

In this section, we introduce the parabolic-hyperbolic variant of the landfast ice model, and we state the local strong well-posedness result in this context.

The model without parabolic regularization in the balance laws is given by
\begin{equation}\label{eq:landfast ice model para-hyper}
    \left\{
    \begin{aligned}
        \del_t \bv + (\bv \cdot \nabla) \bv
        &= \frac{1}{\rho h} \div(\bsigma) - C_{\cor} \bk \times (\bv - \bv_\ro)\\
        &\quad + \frac{1}{\rho h}\bigl(f_\ra + f_\ro(\bv) + f_\rb(\bv,h,A)\bigr), &&\tin G \times (0,T),\\
        \del_t h + \div(\bv h)
        &= 0, &&\tin G \times (0,T),\\
        \del_t A + \div(\bv A)
        &= 0, &&\tin G \times (0,T),\\
        \bv 
        &= 0, &&\ton \del G \times (0,T),\\
        \bv(0) = \bv_0, \enspace h(0)
        &= h_0, \enspace A(0) = A_0, &&\tin G.
    \end{aligned}
    \right.
\end{equation}
In contrast to the model in \eqref{eq:landfast ice model}, no boundary conditions on~$h$ and $A$ are imposed.
The functional analytic setting suited for this problem is as follows:
We define the ground space $\rY_0 \coloneqq \rL^q(G)^2 \times \rW^{1,q}(G) \times \rW^{1,q}(G)$ and the regularity space $\rY_1 \coloneqq \rW^{2,q}(G)^2 \cap \rW_0^{1,q}(G)^2 \times \rW^{1,q}(G) \times \rW^{1,q}(G)$.
In this context, we also introduce the trace space $\rY_\gamma = (\rY_0,\rY_1)_{1-\nicefrac{1}{p},p} = \rB_{qp,\rD}^{2-\nicefrac{2}{p}}(G)^2 \times \rW^{1,q}(G) \times \rW^{1,q}(G)$, where the subscript $_\rD$ represents Dirichlet boundary conditions in the situation $2 - \frac{2}{p} > \frac{1}{q}$.
We remark that this is the threshold of the regularity parameter in the Besov space $\rB_{qp}^{2-\nicefrac{2}{p}}(G)$ such that the trace is defined in a strong sense, see, e.g., \cite[Sec.~2.9]{Tri:78}.
Similarly as for the fully parabolic problem, for $\delta > 0$ sufficiently small, we define the open set $V \subset \rY_\gamma$ by
\begin{equation}\label{eq:open set V para-hyper}
    V \coloneqq \{\bu = (\bv,h,A) \in \rY_\gamma : h > \kappa \tand A \in (0,1+\delta)\}.
\end{equation}

We now state the local strong well-posedness result for \eqref{eq:landfast ice model para-hyper}.

\begin{thm}\label{thm:loc strong wp para-hyper}
Consider $p$, $q \in (1,\infty)$ such that $\frac{1}{p} + \frac{1}{q} < \frac{1}{2}$, and let $\bu_0 = (\bv_0,h_0,A_0) \in V$, where $V \subset \rY_\gamma$ has been defined in \eqref{eq:open set V para-hyper}.
Moreover, assume that the wind and ocean velocities $\bv_\ra$ and $\bv_\ro$ fulfill $\bv_\ra$, $\bv_\ro \in \rL^\infty(0,T;\rL^{2q}(G)^2)$ for all $T > 0$.
Then there exists $T' > 0$ such that \eqref{eq:landfast ice model para-hyper} has a unique solution $\bu = (\bv,h,A) \in \rW^{1,p}(0,T';\rY_0) \cap \rL^p(0,T';\rY_1)$.
\end{thm}

We now compare our approach to the one used in \cite{LTT:22} in the context of the analysis of a regularized version of Hibler's model.
Liu, Thomas and Titi consider a different approximating stress tensor in which the symmetric-gradient and divergence contributions are regularized separately.
More precisely, for the ice strength $P_{\mathrm{LTT}}$ of the form $P_{\mathrm{LTT}} = h P^* \exp(c^* A)$, and for a regularization parameter $\eps > 0$, they consider the regularized stress
\[
    \sigma_{\mathrm{LTT}} = P_{\mathrm{LTT}}\frac{\nabla \bv + \nabla \bv^\top}{\sqrt{|\nabla \bv + \nabla \bv^\top|^2 + \eps^2}} + P_{\mathrm{LTT}} \frac{\div \bv \bI}{\sqrt{|\div \bv|^2 + \eps^2}},
\]
see also \cite[Eq.~(1.4a)]{LTT:22},  which is structurally different from the original viscous-plastic stress introduced by Hibler and the regularization studied in \cite{BDHH:22}.
Their high-order energy argument uses the monotonicity structure of this particular stress map and is carried out on the periodic domain \(\mathbb T^2\). 
To the best of our knowledge, the extension of their argument to the original elliptic Hibler rheology has not been established. 
The present model instead retains the elliptic yield-curve stress regularization, incorporates tensile strength, and is considered on a bounded domain with no-slip boundary conditions. 
Consequently, the argument of \cite{LTT:22} does not apply directly or verbatim. 
We instead use the Lagrangian and anisotropic maximal-\(\rL^p\)-regularity framework developed in \cite{Bra:25}, together with additional estimates for the tensile-strength coefficients and the basal-stress term.
We do not exclude that a different adaptation of the energy approach could be developed.

\subsection{Proof of \autoref{thm:loc strong wp para-hyper}}\label{ssec:proof parabolic-hyperbolic}
\ 

In this section, we prove the local strong well-posedness result, \autoref{thm:loc strong wp para-hyper}, for the parabolic-hyperbolic fast ice model.
The proof is carried out in three steps:
First, we transform the problem to Lagrangian coordinates and state the well-posedness result in this setting, see \autoref{prop:local strong wp in the Lagrangian setting} below.
In a second step, we analyze the underlying linear problem.
Finally, we rephrase the well-posedness question as a fixed point problem and provide the nonlinear estimates to complete the proof of \autoref{prop:local strong wp in the Lagrangian setting} and then also of \autoref{thm:loc strong wp para-hyper}.

\subsubsection{Transformation to Lagrangian coordinates and reformulation as a fixed point problem}
\ 

In order to handle the hyperbolic terms in the balance laws, we rewrite \eqref{eq:landfast ice model para-hyper} in Lagrangian coordinates.
To this end, we invoke the characteristics $X$ solving
\begin{equation*}
    \del_t X(t,\by) = \bv(t,X(t,\by)), \tfor t > 0, \enspace X(0,\by) = \by, \tfor \by \in \bR^2.
\end{equation*}
For $t \in (0,T)$, we denote by $Y(t,\cdot) = [X(t,\cdot)]^{-1}$ the inverse of $X(t,\cdot)$.
We then introduce the variables in Lagrangian coordinates $\tbv(t,\by) \coloneqq \bv(t,X(t,\by))$, $\th(t,\by) \coloneqq h(t,X(t,\by))$, and $\tA(t,\by) \coloneqq A(t,X(t,\by))$.
We use $\tbu = (\tbv,\th,\tA)$ to denote the principal variable in Lagrangian coordinates.
For simplicity, and by a slight abuse of notation, we still denote the atmosphere and ocean velocities in Lagrangian coordinates by $\bv_\ra$ and $\bv_\ro$, respectively.

In the following, we provide the transformed terms in the Lagrangian setting.
The chain rule first yields $2 \dbeps_{ij}(\bv) = \sum_{k=1}^2 \bigl((\del_i Y_k) \del_k \tbv_j + (\del_j Y_k) \del_k \tbv_i\bigr) \eqqcolon 2 \tdbeps_{ij}(\tbv)$.
In this context, with the coefficients $a_{ij}^{kl}(\dbeps,\tP)$ as made precise in \eqref{eq:coeffs principal part fast ice op}, we also introduce new coefficients, namely
\begin{equation*}
    a_{ij}^{klm}(\tdbeps(\tbv),\tP(\th,\tA)) \coloneqq (\del_k Y_m) a_{ij}^{kl}(\tdbeps(\tbv),\tP(\th,\tA)),
\end{equation*}
and we calculate
\begin{equation*}
    \del_m \tdbeps_{jl}(\tbv) = \frac{1}{2} \sum_{n=1}^2 \bigl((\del_m \del_j Y_n) \del_n \tbv_l + (\del_j Y_n) \del_m \del_n \tbv_l + (\del_m \del_l Y_n) \del_n \tbv_j + (\del_l Y_n) \del_m \del_n \tbv_j\bigr).
\end{equation*}
With regard to \eqref{eq:fast ice op}, we then get the transformed fast ice operator $\tbAFI$ given by
\begin{equation*}
    \begin{aligned}
        (\tbAFI(\tbu) \tbv)_i 
        &= \sum_{j,k,l,m=1}^2 a_{ij}^{klm}(\tdbeps(\tbv),\tP(\th,\tA)) \del_m \tdbeps_{jl}(\tbv)\\
        &\quad + \frac{1}{2 \rho \th \shadedDelta(\tdbeps(\tbv))} \sum_{j,k=1}^2 (\del_j Y_k)(1+k_\rt) \bigl(\del_k \th \cdot P^* \exp(-c^*(1-\tA)) + \del_k \tA \cdot c^* P(\th,\tA)\bigr) (\bS \tdbeps(\tbv))_{ij}.
    \end{aligned}
\end{equation*}
On the other hand, the transformed terms corresponding to the horizontal divergence of $P'$ are given by
\begin{equation*}
    (\tB_\rh(\tbu) \th)_i \coloneqq \frac{\del_h P'(\th,\tA)}{2 \rho \th} \sum_{j=1}^2 (\del_i Y_j) \del_j \th \tand (\tB_\rA(\tbu) \tA)_i \coloneqq \frac{\del_A P'(\th,\tA)}{2 \rho \th} \sum_{j=1}^2 (\del_i Y_j) \del_j \tA.
\end{equation*}
In order to shorten notation, we will also use $\tB(\tbu) \binom{\th}{\tA} \coloneqq \tB_\rh(\tbu) \th + \tB_\rA(\tbu) \tA$ in what follows.
Finally, we find that $\div \bv = \sum_{j,k=1}^2 (\del_j Y_k) \del_k \tbv_j$.
For the other terms, we note that there are no derivatives involved, so the transformed terms simply arise from plugging in the respective transformed variables.
Putting everything together, we obtain the system in Lagrangian coordinates
\begin{equation}\label{eq:landfast ice Lagrangian coords}
    \left\{
    \begin{aligned}
        \del_t \tbv
        &= \tbAFI(\tbu) \tbv - \tB(\tbu) \binom{\th}{\tA} - C_{\cor} \bk \times (\tbv - \bv_\ro)\\
        &\quad + \frac{1}{\rho \th}\bigl(f_\ra + f_\ro(\tbv) + f_\rb(\tbv,\th,\tA)\bigr), &&\tin G \times (0,T),\\
        \del_t \th
        &= -\th \sum_{j,k=1}^2 (\del_j Y_k) \del_k \tbv_j, &&\tin G \times (0,T),\\
        \del_t \tA
        &= -\tA \sum_{j,k=1}^2 (\del_j Y_k) \del_k \tbv_j, &&\tin G \times (0,T),\\
        \tbv
        &= 0, &&\ton \del G \times (0,T),\\
        \tbv(0) 
        &= \bv_0, \enspace \th(0) = h_0, \enspace \tA(0) = A_0, &&\tin G.
    \end{aligned}
    \right.
\end{equation}

The well-posedness result, \autoref{thm:loc strong wp para-hyper}, can now be reformulated for the system in Lagrangian coordinates \eqref{eq:landfast ice Lagrangian coords}.
This result reads as follows.

\begin{prop}\label{prop:local strong wp in the Lagrangian setting}
Suppose that $p$, $q \in (1,\infty)$ satisfy $\frac{1}{p} + \frac{1}{q} < \frac{1}{2}$, and consider $\bu_0 \in V$, where $V \subset \rY_\gamma$ has been defined in \eqref{eq:open set V para-hyper}.
Besides, assume that $\bv_\ra$ and $\bv_\ro$ are as made precise in \autoref{thm:loc strong wp para-hyper}.
Then there is $T' > 0$ so that \eqref{eq:landfast ice Lagrangian coords} has a unique solution $\tbu = (\tbv,\th,\tA) \in \rW^{1,p}(0,T';\rY_0) \cap \rL^p(0,T';\rY_1) \cap \rC([0,T'];V)$.
\end{prop}

\subsubsection{Analysis of the linearized problem}
\ 

The next step is to analyze the linearized problem associated with \eqref{eq:landfast ice Lagrangian coords}.
Before providing the linearized problem under consideration, let us introduce an auxiliary piece of notation, namely
\begin{equation*}
    B_1(\bu_1) \binom{h}{A} = \frac{\del_h P'(h_1,A_1)}{2 \rho h_1} \nabla h + \frac{\del_A P'(h_1,A_1)}{2 \rho h_1} \nabla A.
\end{equation*}
For a positive time $T > 0$, a shift $\omega \ge 0$, in order to account for the fact that we merely get the maximal $\rL^p$-regularity of the fast ice operator up to such a shift, for $f_1 \colon G \times (0,T) \to \bR^2$ as well as $f_2$, $f_3 \colon G \times (0,T) \to \bR$, for $\bu_1 \in \rC^1(\oG)^2 \times \rC(\oG) \times \rC(\oG)$ with $h_1 \ge \kappa$, and for suitable initial data $\bv_0$, $h_0$ and~$A_0$ to be made precise, we investigate the linearized problem
\begin{equation}\label{eq:lin probl landfast sea-ice}
    \left\{
    \begin{aligned}
        \del_t \bv - \bigl(\bAFI(\bu_1) - \omega\bigr) \bv + B_1(\bu_1) \binom{h}{A}
        &= f_1, &&\tin G \times (0,T),\\
        \del_t h + h_1 \div \bv + \omega h
        &= f_2, &&\tin G \times (0,T),\\
        \del_t A + A_1 \div \bv + \omega A
        &= f_3, &&\tin G \times (0,T),\\
        \bv
        &= 0, &&\ton \del G \times (0,T),\\
        \bv(0) = \bv_0, \enspace h(0) = h_0, \enspace A(0)
        &= A_0, &&\tin G.
    \end{aligned}
    \right.
\end{equation}
The associated operator matrix on the ground space $\rY_0$ in the unshifted case $\omega = 0$ then is
\begin{equation}\label{eq:op matrix landfast sea-ice}
    \begin{aligned}
        \ALFIhyp(\bu_1) 
        &= \begin{pmatrix}
            -\AFID(\bu_1) & \frac{\del_h P'(h_1,A_1)}{2 \rho h_1} \nabla & \frac{\del_A P'(h_1,A_1)}{2 \rho h_1} \nabla\\
            h_1 \div & 0 & 0\\
            A_1 \div & 0 & 0
        \end{pmatrix}, \twith \rD(\ALFIhyp(\bu_1)) = \rY_1.
    \end{aligned}
\end{equation}

In the proposition below, we establish the maximal $\rL^p$-regularity of the linearized problem \eqref{eq:lin probl landfast sea-ice}.
Prior to stating this result, we introduce some further spaces.
In fact, for the above spaces $\rY_0$ and $\rY_1$, we define the data space $\bF_0$ and maximal regularity space $\bF_1$ by
\begin{equation}\label{eq:data & max reg space}
    \bF_0 \coloneqq \rL^p(0,T;\rY_0) \tand \bF_1 \coloneqq \rW^{1,p}(0,T;\rY_0) \cap \rL^p(0,T;\rY_1),
\end{equation}
respectively.
The result on the maximal $\rL^p$-regularity of the linearized problem \eqref{eq:lin probl landfast sea-ice} now reads as follows.

\begin{prop}\label{prop:max reg lin syst}
Consider $T \in (0,\infty]$, $p \in (1,\infty)$ and $q \in (2,\infty)$ so that $\nicefrac{2}{p} + \nicefrac{1}{q} \neq 2$, and suppose that $\bu_1 = (\bv_1,h_1,A_1) \in \rC^1(\oG)^2 \times \rW^{1,q}(\oG) \times \rW^{1,q}(\oG)$ with $h_1 \ge \kappa$.
Moreover, for $\rY_\gamma = (\rY_0,\rY_1)_{1-\nicefrac{1}{p},p}$ and $\bF_0$ as in \eqref{eq:data & max reg space}, take into account $\bu_0 = (\bv_0,h_0,A_0) \in \rY_\gamma$ and $(f_1,f_2,f_3) \in \bF_0$.

Then there is $\omega_0 \in \bR$ such that for all $\omega > \omega_0$, the linearized problem \eqref{eq:lin probl landfast sea-ice} admits a unique solution $\bu = (\bv,h,A) \in \bF_1$, and there exists a constant $C > 0$ such that $\| \bu \|_{\bF_1} \le C\bigl(\| (f_1,f_2,f_3) \|_{\bF_0} + \| \bu_0 \|_{\rY_\gamma}\bigr)$.
\end{prop}

\begin{proof}
First, we decompose $\ALFIhyp(\bu_1)$ from \eqref{eq:op matrix landfast sea-ice} into $\ALFIhyp(\bu_1) = A_1^\LFI(\bu_1) + B(\bu_1)$, where
\begin{equation*}
    A_1^\LFI(\bu_1) \coloneqq \begin{pmatrix}
        -\AFID(\bu_1) & 0 & 0\\
        h_1 \div & 0 & 0\\
        A_1 \div & 0 & 0
    \end{pmatrix} \tand B(\bu_1) \coloneqq \begin{pmatrix}
        0 & \frac{\del_h P'(h_1,A_1)}{2 \rho h_1} \nabla & \frac{\del_A P'(h_1,A_1)}{2 \rho h_1} \nabla\\
        0 & 0 & 0\\
        0 & 0 & 0
    \end{pmatrix}.
\end{equation*}
By virtue of $q > 2$, we find that $\rW^{1,q}(G) \hookrightarrow \rC(\oG)$, see for instance \cite[Thm.~4.6.1]{Tri:78}, and the space $\rW^{1,q}(G)$ further possesses Banach algebra structure.
By \autoref{prop:Hinfty fast ice op}, there is a unique solution $\bv \in \rW^{1,p}(0,T;\rL^q(G)^2) \cap \rL^p(0,T;\rW^{2,q}(G)^2 \cap \rW_0^{1,q}(G)^2) \eqqcolon \bF_1^{\bv}$ to the linear problem
\begin{equation*}
    \del_t \bv - \bigl(\bAFI(\bu_1) - \omega\bigr) \bv = f_1, \tin G \times (0,T), \tand \bv(0) = \bv_0.
\end{equation*}
Now we set $\tf_2 \coloneqq f_2 - h_1 \div \bv$ and $\tf_3 \coloneqq f_3 - A_1 \div \bv$.
Thanks to $\bv \in \bF_1^{\bv}$ and the Banach algebra structure of $\rW^{1,q}(G)$, we find that $\tf_2$, $\tf_3 \in \rL^p(0,T;\rW^{1,q}(G))$.
Due to the maximal $\rL^p$-regularity of the operator~$\omega \Id$ on $\rW^{1,q}(G)$ for every $\omega > 0$, there exist unique solutions $h$, $A \in \rW^{1,p}(0,T;\rW^{1,q}(G))$ to
\begin{equation*}
    \del_t h + \omega \Id h = \tf_2, \tin G \times (0,T), \enspace h(0) = h_0, \tand \del_t A + \omega \Id A = \tf_3, \tin G \times (0,T), \enspace A(0) = A_0.
\end{equation*}
This reveals the maximal $\rL^p$-regularity of the operator matrix $A_1^\LFI(\bu_1) + \omega$ on $\rY_0$, where~$\omega > \omega_1$, and~$\omega_1$ has been made precise in \autoref{prop:Hinfty fast ice op}.

We now show that $B(\bu_1)$ can be viewed as a perturbation of~$A_1^\LFI(\bu_1)$.
The shape of $P'$ as revealed in \autoref{sec:landfast ice model} as well as the assumption that $h_1 \ge \kappa > 0$ imply the existence of a constant $C > 0$ such that $\left\| \frac{\del_h P'(h_1,A_1)}{2 \rho h_1} \nabla h \right\|_{\rL^q(G)} \le C \cdot \| h \|_{\rW^{1,q}(G)}$.
The other term in $B(\bu_1)$ can be estimated in the same way, so in total, we find that $B(\bu_1)$ is a bounded perturbation of~$A_1^\LFI(\bu_1)$.
Thus, for $\omega > \omega_0$, where $\omega_0 \ge \omega_1$ for the above $\omega_1$, we find that $\ALFIhyp(\bu_1) + \omega$ has maximal $\rL^p$-regularity on $\rY_0$ thanks to perturbation theory for maximal $\rL^p$-regularity, see for example \cite[Cor.~2]{KW:01}.
\end{proof}

Next, we invoke the so-called \textit{reference solution}, which captures the initial data in the upcoming fixed point argument, since the maximal regularity constant $C > 0$ appearing in the estimate in \autoref{prop:max reg lin syst} generally depends on $T > 0$ unless $\bu_0 = 0$.
For $\omega_0$ from \autoref{prop:max reg lin syst}, we fix $\omega > \omega_0$ below, and for $\bu_0 \in V$, we consider
\begin{equation}\label{eq:syst ref sol landfast sea-ice}
    \left\{
    \begin{aligned}
        \del_t \bv - \bigl(\bAFI(\bu_0) - \omega\bigr) \bv + B_1(\bu_0) \binom{h}{A}
        &= 0, &&\tin G \times (0,T),\\
        \del_t h + h_0 \div \bv + \omega h
        &= 0, &&\tin G \times (0,T),\\
        \del_t A + A_0 \div \bv + \omega A
        &= 0, &&\tin G \times (0,T),\\
        \bv
        &= 0, &&\ton \del G \times (0,T),\\
        \bv(0) = \bv_0, \enspace h(0) = h_0, \enspace A(0)
        &= A_0, &&\tin G.
    \end{aligned}
    \right.
\end{equation}

The following result on the existence of a reference solution to the linearized system \eqref{eq:syst ref sol landfast sea-ice} is a direct consequence of \autoref{prop:max reg lin syst}.

\begin{coro}\label{cor:ref sol}
Suppose that $p$, $q \in (1,\infty)$ satisfy $\frac{1}{p} + \frac{1}{q} < \frac{1}{2}$, and let $0 < T \le \infty$ as well as $\bu_0 \in V$, where $V \subset \rY_\gamma$ has been made precise in \eqref{eq:open set V para-hyper}.
Then there exists a unique solution $\bu_0^* = (\bv_0^*,h_0^*,A_0^*) \in \bF_1$ to \eqref{eq:syst ref sol landfast sea-ice}, and $\bu_0^*$ has the following properties.
\begin{enumerate}[(a)]
    \item It holds that $C_T^* \coloneqq \| \bu_0^* \|_{\bF_1} \to 0$ as $T \to 0$.
    \item We have $\| \bu_0^* - \bu_0 \|_{\rBUC([0,T];\rY_\gamma)} \to 0$ as $T \to 0$.
    As $V \subset \rY_\gamma$ is open, for $r_0 > 0$ sufficiently small, and it holds that $\overline{B}_{\rY_\gamma}(\bu_0,r_0) \subset V$ for given $\bu_0 \in V$.
    Thus, there exists $T_0 > 0$ sufficiently small so that for all $t \in [0,T_0]$, we have $\bu_0^*(t) \in V$ thanks to the estimate $\sup_{t \in [0,T_0]} \| \bu_0^*(t) - \bu_0 \|_{\rY_\gamma} \le \frac{r_0}{2}$.
\end{enumerate}
\end{coro}

\subsubsection{Reformulation as a fixed point problem, nonlinear estimates, and proof of \autoref{thm:loc strong wp para-hyper}}
\ 

We start this final step by introducing the relevant fixed point problem.
In fact, let $\tbu = (\tbv,\th,\tA)$ be a solution to \eqref{eq:landfast ice Lagrangian coords}, and consider the reference solution $\bu_0^* = (\bv_0^*,h_0^*,A_0^*)$ from \autoref{cor:ref sol}.
We then set $\hbu \coloneqq \tbu - \bu_0^*$, and for $\omega > 0$ as fixed above, we note that $\hbu$ then solves
\begin{equation}\label{eq:landfast sea-ice as fixed point probl}
    \left\{
    \begin{aligned}
        \del_t \hbv - \bigl(\bAFI(\bu_0) - \omega\bigr) \hbv + B_1(\bu_0) \binom{\hh}{\hA}
        &= F_1(\hbu), &&\tin G \times (0,T),\\
        \del_t \hh + h_0 \div \hbv + \omega \hh
        &= F_2(\hbu), &&\tin G \times (0,T),\\
        \del_t \hA + A_0 \div \hbv + \omega \hA
        &= F_3(\hbu), &&\tin G \times (0,T),\\
        \hbv
        &= 0, &&\ton \del G \times (0,T),\\
        \hbv(0) = 0, \enspace \hh(0) = 0, \enspace \hA(0)
        &= 0, &&\tin G.
    \end{aligned}
    \right.
\end{equation}
In \eqref{eq:landfast sea-ice as fixed point probl} above, the terms on the right-hand side are given by
\begin{equation*}
    \begin{aligned}
        F_1(\hbu)
        &\coloneqq \bigl(\tbAFI(\tbu) - \bAFI(\bu_0)\bigr) \tbv - \bigl(\tB(\tbu) - B_1(\bu_0)\bigr) \binom{\th}{\tA} + \omega \tbv - C_{\cor} \bk \times (\tbv - \bv_\ro)\\
        &\quad + \frac{1}{\rho \th}\bigl(f_\ra + f_\ro(\tbv) + f_\rb(\tbv,\th,\tA)\bigr),\\
        F_2(\hbu)
        &\coloneqq h_0 \div \tbv - \th \sum_{j,k=1}^2 (\del_j Y_k) \del_k \tbv_j + \omega \th \tand
        F_3(\hbu)
        \coloneqq A_0 \div \tbv - \tA \sum_{j,k=1}^2 (\del_j Y_k) \del_k \tbv_j + \omega \tA.
    \end{aligned}
\end{equation*}

For brevity of notation, we will also use $\prescript{}{0}{\bF_1}$ to denote the subspace of $\bF_1$ with homogeneous initial data, i.e., $\prescript{}{0}{\bF_1} \coloneqq \{(\bv,h,A) \in \bF_1 : (\bv(0),h(0),A(0)) = 0\}$.
As in \cite[Lemma~5.1]{Bra:25}, we find that given $\bu_0 \in V$, and with $\hbu \in \overline{B}_{\prescript{}{0}{\bF_1}}(0,R)$, for $R \in (0,R_0]$ with $R_0 > 0$ sufficiently small, it follows that $\tbu(t) = \hbu(t) + \bu_0^*(t) \in V$ for all $t \in [0,T]$, where $T \in (0,T_0]$, and $T_0 > 0$ is small enough.

We now make precise the setting for the fixed point argument.
For $T \in (0,T_0]$ and $R \in (0,R_0]$, we set
\begin{equation}\label{eq:set & map fixed point arg}
    \cK_T^R \coloneqq \overline{B}_{\prescript{}{0}{\bF_1}}(0,R) \tand \Phi_T^R \colon \cK_T^R \to \prescript{}{0}{\bF_1}, \twhere \Phi_T^R(\hbu) \coloneqq \bz,
\end{equation}
where $\bz \in \prescript{}{0}{\bF_1}$ is the unique solution to \eqref{eq:landfast sea-ice as fixed point probl} with right-hand sides $F_1(\hbu)$, $F_2(\hbu)$, and $F_3(\hbu)$.
Based on the maximal $\rL^p$-regularity as established in \autoref{prop:max reg lin syst}, we find that $\Phi_T^R$ is well-defined provided $F_1$, $F_2$, and $F_3$ are contained in the data space $\bF_0$ from \eqref{eq:data & max reg space}.
In the following, we address the estimates of the nonlinear terms $F_1$, $F_2$, and $F_3$.

\begin{prop}\label{prop:self-map est nonlin terms}
Let $p$, $q \in (1,\infty)$ be such that $\frac{1}{p} + \frac{1}{q} < \frac{1}{2}$, and $\tbu \coloneqq \hbu + \bu_0^*$, where $\hbu \in \overline{B}_{\prescript{}{0}{\bF_1}}(0,R)$, and~$\bu_0^*$ is the reference solution from \autoref{cor:ref sol}.
Suppose that $\bv_\ra$ and $\bv_\ro$ satisfy the assumptions from \autoref{thm:loc strong wp para-hyper}.
Then there exists $C(R,T) > 0$, tending to zero as $R$, $T \to 0$, such that $\| (F_1(\hbu),F_2(\hbu),F_3(\hbu)) \|_{\bF_0} \le C(R,T)$.
Moreover, $R > 0$ and $T > 0$ can be chosen sufficiently small so that $C C(R,T) \le \frac{R}{2}$ for the maximal regularity constant $C > 0$ from \autoref{prop:max reg lin syst}.
\end{prop}

\begin{proof}
We refer here to \cite[Sec.~5.2]{Bra:25} for the respective estimates of $F_2$ and $F_3$ that can be found in \cite[Prop.~5.7]{Bra:25}.
In this proof, we focus on the main differences compared to the previous work, namely on the difference in the ice strength and the additional basal stress $f_\rb$ arising in the momentum equation.
Thus, recalling that $\tbu = \hbu + \bu_0^*$, we find that the task is to estimate the difference 
\begin{equation}\label{eq:diff to be est mom eq}
    \bigl(\tbAFI(\hbu + \bu_0^*) - \bAFI(\bu_0)\bigr) (\hbv + \bv_0^*) - \bigl(\tB(\hbu + \bu_0^*) - B_1(\bu_0)\bigr) \binom{\hh + h_0^*}{\hA + A_0^*}
\end{equation}
as well as the term $\frac{1}{\rho (\hh + h_0^*)} f_\rb(\hbv + \bv_0^*,\hh + h_0^*,\hA + A_0^*)$ in $\rL^p(0,T;\rL^q(G)^2)$.
With regard to the difference in \eqref{eq:diff to be est mom eq}, the idea is to split the problem into handling the ``geometric'' nonlinearity and the ``intrinsic'' nonlinearity.
More precisely, we split \eqref{eq:diff to be est mom eq} into
\begin{equation*}
    \begin{aligned}
        &\bigl(\tbAFI(\hbu + \bu_0^*) - \bAFI(\hbu + \bu_0^*)\bigr) (\hbv + \bv_0^*) - \bigl(\tB(\hbu + \bu_0^*) - B_1(\hbu + \bu_0^*)\bigr) \binom{\hh + h_0^*}{\hA + A_0^*},\\
        &\bigl(\bAFI(\hbu + \bu_0^*) - \bAFI(\bu_0^*)\bigr)(\hbv + \bv_0^*) - \bigl(B_1(\hbu + \bu_0^*) - B_1(\bu_0^*)\bigr) \binom{\hh + h_0^*}{\hA + A_0^*}, \tand\\
        &\bigl(\bAFI(\bu_0^*) - \bAFI(\bu_0)\bigr) (\hbv + \bv_0^*) - \bigl(B_1(\bu_0^*) - B_1(\bu_0)\bigr) \binom{\hh + h_0^*}{\hA + A_0^*}.
    \end{aligned}
\end{equation*}
In a similar manner as in \cite[Lemma~5.5]{Bra:25}, upon observing that the coefficients of the fast ice operator~$\bAFI$ as introduced in \eqref{eq:lin fast ice op} enjoy the same regularity properties as those of the original Hibler operator $\bA^\rH$ as defined in \cite[Sec.~3]{BDHH:22}, see also \cite[Sec.~3]{Bra:25}, and thus allow for analogous estimates as $\bA^\rH$, and the same remains valid for the off-diagonal terms $B_1$, we argue that there exists a constant $C > 0$ such that
\begin{equation*}
    \begin{aligned}
        \left\| \bigl(\bAFI(\tbu) - \bAFI(\bu_0^*)\bigr)\tbv - \bigl(B_1(\tbu) - B_1(\bu_0^*)\bigr) \binom{\th}{\tA} \right\|_{\rL^p(0,T;\rL^q(G))} 
        &\le C R(R + C_T^*) \tand\\
        \left\| \bigl(\bAFI(\bu_0^*) - \bAFI(\bu_0)\bigr) \tbv - \bigl(B_1(\bu_0^*) - B_1(\bu_0)\bigr) \binom{\th}{\tA} \right\|_{\rL^p(0,T;\rL^q(G))} 
        &\le C \cdot \| \bu_0^* - \bu_0 \|_{\rBUC([0,T];\rY_\gamma)} (R + C_T^*).
    \end{aligned}
\end{equation*}
In other words, the terms associated with the intrinsic nonlinearity can be estimated by terms that decay to zero as $R \to 0$ and $T \to 0$ by \autoref{cor:ref sol}(a) and~(b).

In order to handle \eqref{eq:diff to be est mom eq}, it remains to treat the geometric nonlinearity.
For this purpose, let us observe that the difference $\bigl(\tbAFI(\hbu + \bu_0^*) - \bAFI(\hbu + \bu_0^*)\bigr) (\hbv + \bv_0^*)$ can be expanded and estimated by following the lines of \cite[(5.12)--(5.16)]{Bra:25}, since the modification of the ice strength $\tP = P + T$ as made precise in~\autoref{ssec:lin theory} does not affect the regularity of the coefficients.
Thus, $\bigl(\tbAFI(\hbu + \bu_0^*) - \bAFI(\hbu + \bu_0^*)\bigr) (\hbv + \bv_0^*)$ admits an estimate by a term that is decaying in $R$ and $T$ in $\rL^p(0,T;\rL^q(G)^2)$.
Concerning the off-diagonal terms $\bigl(\tB(\hbu + \bu_0^*) - B_1(\hbu + \bu_0^*)\bigr) \binom{\hh + h_0^*}{\hA + A_0^*}$, for $P'$ as defined in \autoref{ssec:lin theory}, the $h$-term is
\begin{equation*}
    \frac{\del_h P'(\hh + h_0^*,\hA + A_0^*)}{2 \rho (\hh + h_0^*)} \sum_{j=1}^2 (\del_i Y_j - \delta_{ij}) \del_j (\hh + h_0^*).
\end{equation*}
Noting that the term $\del_h P'(\hh + h_0^*,\hA + A_0^*)$ is bounded by some constant, we can proceed as in the proof of \cite[Prop.~5.6]{Bra:25} to infer that
\begin{equation*}
    \left\| \frac{\del_h P'(\hh + h_0^*,\hA + A_0^*)}{2 \rho (\hh + h_0^*)} \sum_{j=1}^2 (\del_i Y_j - \delta_{ij}) \del_j (\hh + h_0^*) \right\|_{\rL^p(0,T;\rL^q(G))} \le C T^{\nicefrac{1}{p'}} (T^{\nicefrac{1}{p}} R + C_T^*).
\end{equation*}
We remark that the corresponding $A$-term can be estimated in the same way.
In summary, we have verified that the difference in \eqref{eq:diff to be est mom eq} can be estimated in $\rL^p(0,T;\rL^q(G)^2)$ by a term that is shrinking to zero as $R \to 0$ and $T \to 0$.

The remainder of the proof is dedicated to discussing the estimate of the term associated with the basal stress $\frac{1}{\rho (\hh + h_0^*)} f_\rb(\hbv + \bv_0^*,\hh + h_0^*,\hA + A_0^*)$.
To this end, we first recall that we consider $R \in (0,R_0]$ and $T \in (0,T_0]$ sufficiently small such that $\tbu(t) = \hbu(t) + \bu_0^*(t) \in V$ for all $t \in [0,T]$.
Consequently, $\th = \hh + h_0^* > \kappa$, which in turn implies that the factor $\frac{1}{\rho(\hh + h_0^*)}$ is bounded, so the estimate reduces to controlling $f_\rb(\hbv + \bv_0^*,\hh + h_0^*,\hA + A_0^*)$.
At this stage, we also observe that the $\rL^p(0,T;\rL^q(G)^2)$-norm of $f_\rb(\hbv + \bv_0^*,\hh + h_0^*,\hA + A_0^*)$ can be bounded from above by
\begin{equation*}
    k_2 \left( \frac{-(\hbv + \bv_0^*)}{|\hbv + \bv_0^*| + v_0} \right)(\hh + h_0^* - h_{\crit}) \exp\left(-\alpha_{\rb}(1 - (\hA + A_0^*)\right).
\end{equation*}
It readily follows that $\frac{-(\hbv + \bv_0^*)}{|\hbv + \bv_0^*| + v_0}$ is bounded.
On the other hand, invoking that $\tbu(t) \in V$ again, we find that $\hA(t) + A_0^*(t) \in (0,1+\delta)$ for all $t \in [0,T]$, which in turn yields that $\exp\left(-\alpha_{\rb}(1 - (\hA + A_0^*)\right)$ is bounded.
Therefore, we only need to estimate $\hh + h_0^* - h_{\crit}$ in $\rL^p(0,T;\rL^q(G))$.
For the first addend, we employ H\"older's inequality as well as $\hbu \in \overline{B}_{\prescript{}{0}{\bF_1}}(0,R)$, which in particular means that $\hh$ has homogeneous initial values, while we directly invoke the definition of the maximal regularity space for the second addend.
With regard to the third addend, we use H\"older's inequality and the fact that $h_{\crit}$ is constant.
Putting together the previous arguments, we find that
\begin{equation*}
    \begin{aligned}
        &\quad \left\| \frac{1}{\rho (\hh + h_0^*)} f_\rb(\hbv + \bv_0^*,\hh + h_0^*,\hA + A_0^*) \right\|_{\rL^p(0,T;\rL^q(G))}\\
        &\le C \cdot \| f_\rb(\hbv + \bv_0^*,\hh + h_0^*,\hA + A_0^*) \|_{\rL^p(0,T;\rL^q(G))}\\
        &\le C \cdot \| \hh + h_0^* - h_{\crit} \|_{\rL^p(0,T;\rL^q(G))}\\
        &\le C \cdot \bigl(\| \hh \|_{\rL^p(0,T;\rL^q(G))} + \| h_0^* \|_{\rL^p(0,T;\rL^q(G))} + \| h_{\crit} \|_{\rL^p(0,T;\rL^q(G))}\bigr)\\
        &\le C \cdot \bigl(T \cdot \| \hh \|_{\rW^{1,p}(0,T;\rW^{1,q}(G))} + C_T^* + T^{\nicefrac{1}{p}} \cdot \| h_{\crit} \|_{\rL^\infty(0,T;\rL^\infty(G))}\bigr)\\
        &\le C(T R + C_T^* + T^{\nicefrac{1}{p}}).
    \end{aligned}
\end{equation*}
Owing to \autoref{cor:ref sol}(a), we find that the term corresponding to the basal stress can be estimated by a term that tends to zero as $R \to 0$ and $T \to 0$.
This completes the proof.
\end{proof}

The next result concerns the Lipschitz estimates of the nonlinear terms.
It can be obtained in the same way as \autoref{prop:self-map est nonlin terms}.
We do not provide a proof and refer to \cite[Sec.~5.2]{Bra:25} and \cite[Sec.~6.4]{Bra:24} instead for more details in a similar context.

\begin{prop}\label{prop:Lipschitz est nonlin terms}
Consider $p$, $q \in (1,\infty)$ so that $\frac{1}{p} + \frac{1}{q} < \frac{1}{2}$, and for the reference solution $\bu_0^*$ from \autoref{cor:ref sol} and $\hbu_i \in \overline{B}_{\prescript{}{0}{\bF_1}}(0,R)$, $i=1,2$, let $\tbu_i \coloneqq \hbu_i + \bu_0^*$.
Additionally, assume that $\bv_\ra$ and $\bv_\ro$ satisfy the assumptions from \autoref{thm:loc strong wp para-hyper}.
Then there is a constant $L(R,T) > 0$, with the property that $L(R,T) \to 0$ as $R \to 0$ and $T \to 0$, such that
\begin{equation*}
    \| (F_1(\hbu_1) - F_1(\hbu_2),F_2(\hbu_1) - F_2(\hbu_2),F_3(\hbu_1) - F_3(\hbu_2)) \|_{\bF_0} \le L(R,T) \cdot \| \hbu_1 - \hbu_2 \|_{\bF_1}.
\end{equation*}
\end{prop}

We are now in a position to prove the local strong well-posedness of the parabolic-hyperbolic landfast ice model.

\begin{proof}[Proof of \autoref{thm:loc strong wp para-hyper}]
First, we will show \autoref{prop:local strong wp in the Lagrangian setting} in the Lagrangian framework, and then we will deduce \autoref{thm:loc strong wp para-hyper} from there by invoking the inverse change of variables associated with the Lagrangian coordinates.

We start by recalling the set $\cK_T^R$ and the map $\Phi_T^R$ from \eqref{eq:set & map fixed point arg}.
For $\hbu \in \cK_T^R$, we use the maximal regularity from \autoref{prop:max reg lin syst} as well as the nonlinear estimates from \autoref{prop:self-map est nonlin terms} to find
\begin{equation*}
    \| \Phi_T^R(\hbu) \|_{\bF_1} \le C \cdot \| (F_1(\hbu),F_2(\hbu),F_3(\hbu)) \|_{\bF_0} \le C \cdot C(R,T) \le \frac{R}{2}
\end{equation*}
for $R > 0$ and $T > 0$ small enough.
On the other hand, \autoref{prop:max reg lin syst} and \autoref{prop:Lipschitz est nonlin terms} yield
\begin{equation*}
    \| \Phi_T^R(\hbu_1) - \Phi_T^R(\hbu_2) \|_{\bF_1} \le C L(R,T) \cdot \| \hbu_1 - \hbu_2 \|_{\bF_1} \le \frac{1}{2} \cdot \| \hbu_1 - \hbu_2 \|_{\bF_1}, \tfor \hbu_1, \, \hbu_2 \in \cK_T^R
\end{equation*}
and $R > 0$ and $T > 0$ sufficiently small.
In other words, the map $\Phi_T^R$ is a self-map and contraction on $\cK_T^R$ for such $R > 0$ and $T > 0$.
The contraction mapping principle thus implies the existence of a unique fixed point $\hbu \in \prescript{}{0}{\bF_1}$ of \eqref{eq:landfast sea-ice as fixed point probl}.
At the same time, by construction, this means that \eqref{eq:landfast ice Lagrangian coords} has a unique solution $\tbu = \hbu + \bu_0^* \in \bF_1$.
The last part of the assertion of \autoref{prop:local strong wp in the Lagrangian setting} then is a result of the embedding 
\begin{equation*}
    \bF_1 = \rW^{1,p}(0,T;\rY_0) \cap \rL^p(0,T;\rY_1) \hookrightarrow \rBUC([0,T];\rY_\gamma),
\end{equation*}
see \cite[Theorem III.4.10.2]{Ama:95}, together with the discussion preceding \eqref{eq:set & map fixed point arg}.

With regard to the assertion of \autoref{thm:loc strong wp para-hyper}, as in \cite[(5.8) and Lemma~5.2]{Bra:25}, we argue that we can consider $T_1 > 0$ sufficiently small such that the inverse $Y$ of the characteristics $X$ exists for all $T \in (0,T_1]$.
Moreover, the regularity of $\tbv$ as revealed in \autoref{prop:local strong wp in the Lagrangian setting} implies that
\begin{equation*}
    X \in \rW^{2,p}(0,T;\rL^q(G)^2) \cap \rW^{1,p}(0,T;\rW^{2,q}(G)^2), 
\end{equation*}
and these properties carry over to the inverse $Y$.
Therefore, we get the variables in the Eulerian framework $\bu = (\bv,h,A)$ from the Lagrangian variables by $\bv(t,x) = \tbv(t,Y(t,x))$, $h(t,x) = \th(t,Y(t,x))$ as well as~$A(t,x) = \tA(t,Y(t,x))$.
Let us observe that the regularity properties of the solution are not affected by the change of variables thanks to the above remark on the regularity of the coordinate transform.
\end{proof}

\section{Global well-posedness close to equilibria}\label{sec:glob wp close to equil}

In this section, we investigate the global strong well-posedness of the landfast ice model close to equilibria, and in the absence of external forcing terms.
The corresponding result can be found in \autoref{thm:glob ex close to equil}.
The idea is to use the generalized principle of linearized stability as recalled in \autoref{lem:gen princ of lin stab}.
With regard to the rheology, we distinguish the cases $k_{\rt} \in [0,1)$ and $k_{\rt} \equiv 1$.
In fact, in the former case, in contrast to \cite[Thm.~2.3]{BDHH:22}, we manage to show the result without assumption on the regularization parameter $\shadedDelta_{\min} > 0$, while in the latter case, such a smallness condition seems indispensable.

More precisely, we focus on the situation when $f_{\ra} = f_{\ro} = f_{\rc} = f_{\rsh} = f_{\rb} = 0$.
The resulting simplified landfast ice problem is given by
\begin{equation}\label{eq:simplified landfast ice model}
    \left\{
    \begin{aligned}
        \del_t \bv + (\bv \cdot \nabla) \bv
        &= \frac{1}{\rho h} \div(\bsigma), &&\tin G \times (0,T),\\
        \del_t h + \div(\bv h)
        &= d_\rh \Delta h, &&\tin G \times (0,T),\\
        \del_t A + \div(\bv A)
        &= d_\rA \Delta A, &&\tin G \times (0,T),\\
        \bv = 0, \enspace \del_\nu h&= \del_\nu A = 0, &&\ton \del G \times (0,T),\\
        \bv(0) = \bv_0, \enspace h(0)
        &= h_0, \enspace A(0) = A_0, &&\tin G.
    \end{aligned}
    \right.
\end{equation}
We focus here on the unweighted case, i.e., for $p$, $q \in (1,\infty)$ such that 
\begin{equation}\label{eq:cond p and q glob small data}
    \frac{1}{p} + \frac{1}{q} < \frac{1}{2},
\end{equation}
we consider the trace space
\begin{equation}\label{eq:trace space glob wp close to equil}
    \rX_\gamma \coloneqq \rX_{\gamma,1} = \rB_{qp,\rD}^{2 - \frac{2}{p}}(G)^2 \times \rB_{qp,\rN}^{2 - \frac{2}{p}}(G) \times \rB_{qp,\rN}^{2 - \frac{2}{p}}(G),
\end{equation}
where we recall that the subscripts $_\rD$ and $_\rN$ indicate Dirichlet and Neumann boundary conditions, respectively.
Analogously, we invoke the open set $V \coloneqq V_1$ as introduced in \eqref{eq:open set V_mu}.

The simplified right-hand side $F_s(\bu)$ capturing the lower-order nonlinearities is given by
\begin{equation*}
    F_s(\bu) = \begin{pmatrix}
        -(\bv \cdot \nabla)\bv\\
        -\div(\bv h)\\
        -\div(\bv A)
    \end{pmatrix}.
\end{equation*}
Recalling the operator matrix $\ALFI$ from \eqref{eq:op matrix landfast ice}, for the principal variable $\bu = (\bv,h,A)$, we reformulate the simplified landfast ice model as a quasilinear evolution equation of the form
\begin{equation}\label{eq:ACP simplified fast ice}
    \left\{
    \begin{aligned}
        \bu'(t) + \ALFI(\bu(t))\bu(t)
        &= F_s(\bu(t)), \tfor t \in (0,T),\\
        \bu(0)
        &= \bu_0.
    \end{aligned}
    \right.
\end{equation}
By $\cE$, we denote the set of equilibrium solutions to \eqref{eq:ACP simplified fast ice}, or, equivalently, to \eqref{eq:simplified landfast ice model}, i.e., 
\begin{equation*}
    \cE \coloneqq \left\{\bu \in V \cap \rX_1 : \ALFI(\bu) \bu = F_s(\bu)\right\}.
\end{equation*}
We observe that $\bu_* = (0,h_*,A_*)$ for $h_* > \kappa$ and $A_* \in (0,1]$ constant is clearly an equilibrium solution, so it holds that $\bu_* \in \cE$.
The theorem below is the second main result of this paper.
It asserts the stability of equilibria of the form of $\bu_*$.
The main difficulty in establishing this stability is that the Neumann Laplacians place zero in the spectrum, so the classical principle of linearized stability is not directly applicable.
Therefore, we show that the local equilibrium set is a two-dimensional manifold, and we verify that zero is a semi-simple eigenvalue.

It is also worth comparing this result to \cite[Thm.~2.3]{BDHH:22}.
Unlike \cite[Thm.~2.3]{BDHH:22}, no smallness assumption on \(\shadedDelta_{\min}\) is required when \(k_{\mathrm t}<1\), because the pressure coupling permits cancellation of the off-diagonal terms after a suitable choice of test functions. 
The endpoint \(k_{\mathrm t}=1\) has no analogue in Hibler’s model at the level of the pressure coupling, since \(P'=P-T\) then vanishes identically. We treat this case through a separate coercivity argument, subject to a smallness condition on \(\shadedDelta_{\min}\).

\begin{thm}\label{thm:glob ex close to equil}
Consider $p$, $q \in (1,\infty)$ such that \eqref{eq:cond p and q glob small data}, let $\bu_* = (0,h_*,A_*)$ with $h_* > \kappa$ and $A_* \in (0,1]$ constant, and recall the trace space $\rX_\gamma$ from \eqref{eq:trace space glob wp close to equil}.
\begin{enumerate}[(i)]
    \item If $k_{\rt} \in [0,1)$, the equilibrium $\bu_*$ is stable in $\rX_\gamma$, and there exists $r = r(d_\rh,d_\rA) > 0$ such that the unique solution $\bu$ to \eqref{eq:ACP simplified fast ice}, or, equivalently, to \eqref{eq:simplified landfast ice model}, for initial data $\bu_0 \in \rX_\gamma$ with $\| \bu_0 - \bu_* \|_{\rX_\gamma} < r$ exists for all times $t \in \bR_+$ and converges to some $\bu_\infty \in \cE$ in $\rX_\gamma$ at an exponential rate as $t \to \infty$.
    \item If $k_{\rt} \equiv 1$, there is $\shadedDelta_* = \shadedDelta_*(d_\rh,d_\rA) > 0$ so that if $\shadedDelta_{\min} < \shadedDelta_*$, the equilibrium $\bu_*$ is stable in $\rX_\gamma$, so the assertion of~(i) on the global-in-time existence and exponential convergence remains valid.
\end{enumerate}
\end{thm}

The proof of \autoref{thm:glob ex close to equil} will be based on the generalized principle of linearized stability as recalled in \autoref{lem:gen princ of lin stab}.
In the remainder of the section, we show that we are in the scope of \autoref{lem:gen princ of lin stab}.
In \autoref{ssec:dyn evol towards equil}, we will elaborate on how the analytical result in \autoref{thm:glob ex close to equil} is reflected in the numerical simulations, see in particular \autoref{fig:ex3}.

With regard to \autoref{lem:gen princ of lin stab}, we invoke the total linearization
\begin{equation*}
    \ALFI_0 \bu = \ALFI(\bu_*) \bu + ((\ALFI)'(\bu)) \bu_* - F_s'(\bu_*) \bu.
\end{equation*}
First, since the velocity component of $\bu_*$ especially equals zero, we find that
\begin{equation*}
    F_s'(\bu_*) \bu = \begin{pmatrix}
        0\\ -h_* \div \bv\\ -A_* \div \bv
    \end{pmatrix}.
\end{equation*}
We further introduce
\begin{equation}\label{eq:shorthand not P}
    \begin{aligned}
        \tP_* 
        &\coloneqq (1+k_{\rt}) h_* P^* \exp(-c^*(1-A_*)),\\
        P_*' 
        &\coloneqq (1-k_{\rt}) h_* P^* \exp(-c^*(1-A_*)),\\
        P_{h,*}'
        &\coloneqq \frac{\del_h P'(h_*,A_*)}{2 \rho h_*} = \frac{(1-k_{\rt})P^* \exp(-c^*(1-A_*))}{2 \rho h_*}, \tand\\
        P_{A,*}'
        &\coloneqq \frac{\del_A P'(h_*,A_*)}{2 \rho h_*} = \frac{c^*(1-k_{\rt})P^* \exp(-c^*(1-A_*))}{2 \rho}.
    \end{aligned}
\end{equation}
Note that if $k_{\rt} \equiv 1$, then $P_*' = P_{h,*}' = P_{A,*}' = 0$.
We find that
\begin{equation*}
    \ALFI(\bu_*) \bu = \begin{pmatrix}
        -\AFID(\bu_*) \bv + P_{h,*}' \nabla h + P_{A,*}' \nabla A\\
        -d_\rh \Delta_\rN h\\
        -d_\rA \Delta_\rN A
    \end{pmatrix},
\end{equation*}
while the shape of $\bu_* = (0,h_*,A_*)$ implies $((\ALFI)'(\bu)) \bu_* = 0$.
Thus, we obtain the total linearization
\begin{equation}\label{eq:tot lin}
    \ALFI_0 \bu = \begin{pmatrix}
        -\AFID(\bu_*) \bv + P_{h,*}' \nabla h + P_{A,*}' \nabla A\\
        -d_\rh \Delta_\rN h + h_* \div \bv\\
        -d_\rA \Delta_\rN A + A_* \div \bv
    \end{pmatrix}.
\end{equation}

In the following, we discuss properties of the linearized fast ice operator $\AFID(\bu_*)$, and of the total linearization $\ALFI_0$.
For the second part, we distinguish the cases $k_{\rt} \in [0,1)$ and $k_{\rt} \equiv 1$, since the latter case requires a smallness assumption on the regularization parameter $\shadedDelta_{\min}$.

\begin{lem}\label{lem:spectral props of tot lin}
Let $\bu_* = (0,h_*,A_*)$, where $h_* > \kappa$ and $A_* \in (0,1]$ are constant.
\begin{enumerate}[(a)]
    \item For $\AFID(\bu_*)$ as introduced in \eqref{eq:lin fast ice op}, it holds that $0 \in \rho(\AFID(\bu_*))$, $s(\AFID(\bu_*)) < 0$, and $-\AFID(\bu_*)$ has maximal $\rL^p$-regularity.
    \item The total linearization $\ALFI_0$ from \eqref{eq:tot lin} has a compact resolvent in $\rX_0$, the spectrum of $\ALFI_0$ is $q$-independent and consists of eigenvalues.
    \begin{enumerate}[(i)]
        \item For $k_{\rt} \in [0,1)$, we have $\sigma(\ALFI_0) \setminus \{0\} \subset \bC_+$ and $\rN(\ALFI_0) = \{0\} \times \bR \times \bR$. 
        \item If $k_{\rt} \equiv 1$, there is $\shadedDelta_* > 0$ so that if $\shadedDelta_{\min} < \shadedDelta_*$, we get $ \sigma(\ALFI_0) \setminus \{0\} \subset \bC_+$ as well as $\rN(\ALFI_0) = \{0\} \times \bR \times \bR$.
    \end{enumerate}
\end{enumerate}
\end{lem}

\begin{proof}
First, we observe that $\bu_* \in \rC^{1,\alpha}(\oG)^4$, so \autoref{prop:Hinfty fast ice op} yields that there exists $\omega_1 \in \bR$ such that for all $\omega > \omega_1$, the shifted operator $-\AFID(\bu_*) + \omega$ has maximal $\rL^p$-regularity on $\rL^q(G)^2$.
From the Rellich-Kondrachov theorem, it follows that the embedding $\rD(\AFID(\bu_*)) = \rW^{2,q}(G)^2 \cap \rW_0^{1,q}(G)^2 \hookrightarrow \rL^q(G)^2$ is compact, so the operator $\AFID(\bu_*)$ has compact resolvent, yielding in turn that the spectrum is $q$-independent and consists of eigenvalues.
Thus, in order to locate the spectrum, it suffices to test the eigenvalue problem
\begin{equation*}
    \lambda \bv - \AFID(\bu_*) \bv.
\end{equation*}
With regard to the form of the coefficients as revealed in \eqref{eq:coeffs principal part fast ice op}, for the above $\tP_*$, we find that
\begin{equation*}
    - \AFID(\bu_*) \bv = -\frac{\tP_*}{\shadedDelta_{\min}} \sum_{j,k,l=1}^2 \bS_{ij}^{kl} \del_k \del_l \bv_j.
\end{equation*}
Integrating by parts, using that $\bv = 0$ on $\del G$, invoking $\shadedDelta^2$ from \eqref{eq:tri^2}, exploiting $\shadedDelta^2(\nabla \bv) \ge \frac{1}{e^2} |\eps(\bv)|^2$, and employing Korn's inequality as well as Poincar\'e's inequality, we get
\begin{equation}\label{eq:est of fast ice op}
    -\int_G \AFID(\bu_*) \bv \cdot \bv \srd \bx = \frac{\tP_*}{\shadedDelta_{\min}} \sum_{i,j,k,l=1}^2 \int_G \bS_{ij}^{kl} \del_k \bv_i \del_l \bv_j \srd \bx = \frac{\tP_*}{\shadedDelta_{\min}} \int_G \shadedDelta^2(\nabla \bv) \srd \bx \ge C \cdot \| \bv \|_{\rH^1(G)}^2
\end{equation}
for some constant $C = C(\shadedDelta_{\min}) > 0$.
Thus, testing the above eigenvalue problem, we find that
\begin{equation*}
    0 = \lambda \cdot \| \bv \|_{\rL^2(G)}^2 - \int_G \AFID(\bu_*) \bv \cdot \bv \srd \bx \ge \lambda \cdot \| \bv \|_{\rL^2(G)}^2 + C(\shadedDelta_{\min}) \cdot \| \bv \|_{\rH^1(G)}^2,
\end{equation*}
implying that $\lambda \in \bR$, $\lambda \le 0$, and $\bv = 0$ if $\lambda = 0$, so $0 \in \rho(\AFID(\bu_*))$ and $s(\AFID(\bu_*)) < 0$.
In particular, this yields that $\AFID(\bu_*)$ has maximal $\rL^p$-regularity on $\rL^q(G)^2$ without shift, completing the proof of~(a).

For~(b), arguing in a similar way as in~(a), we find that $\ALFI_0$ has a compact resolvent, so its spectrum is $q$-independent and consists of eigenvalues.
Again, for locating the spectrum, it is sufficient to test the eigenvalue problem
\begin{equation}\label{eq:eigenvalue probl}
    \lambda \bu + \ALFI_0 \bu = 0.
\end{equation}
Since $k_{\rt} \in [0,1)$ and $k_{\rt} \equiv 1$ require a different treatment, we distinguish the two cases.
Concerning~(i), for $P_{h,*}'$ and $P_{A,*}'$ as specified in \eqref{eq:shorthand not P}, we set $c_1 \coloneqq \frac{P_{h,*}'}{h_*} > 0$, $c_2 \coloneqq \frac{P_{A,*}'}{A_*} > 0$, and $\tbu \coloneqq (\bv,c_1^{\nicefrac{1}{2}} h,c_2^{\nicefrac{1}{2}} A)$, and we note that
\begin{equation}\label{eq:id tested eq}
    \int_G \bu^\top \cdot (\bv,c_1 h,c_2 A)^\top \srd \bx = \| \tbu \|_{\rL^2(G)}^2.
\end{equation}
Integrating by parts, and invoking the Neumann boundary conditions for $h$ and $A$, we find that
\begin{equation}\label{eq:tested bal laws}
    -\int_G d_\rh \Delta_\rN h \cdot c_1 h \srd \bx = d_\rh c_1 \cdot \| \nabla h \|_{\rL^2(G)}^2 \tand -\int_G d_\rA \Delta_\rN A \cdot c_2 A \srd \bx = d_\rA c_2 \cdot \| \nabla A \|_{\rL^2(G)}^2.
\end{equation}
Thus, testing the above eigenvalue problem for $\ALFI_0$, and using \eqref{eq:est of fast ice op}, \eqref{eq:id tested eq}, and \eqref{eq:tested bal laws} as well as the concrete choice of $c_1$, $c_2 > 0$, we find that
\begin{equation}\label{eq:tested eigenvalue eq}
    \begin{aligned}
        0
        &= \lambda \| \tbu \|_{\rL^2(G)}^2 - \int_G \AFID(\bu_*) \bv \cdot \bv \srd \bx + d_\rh c_1 \cdot \| \nabla h \|_{\rL^2(G)}^2 + d_\rA c_2 \cdot \| \nabla A \|_{\rL^2(G)}^2\\
        &\quad + \bigl(P_{h,*}' - c_1 h_*\bigr) \int_G \nabla h \cdot \bv \srd \bx + \bigl(P_{A,*}' - c_2 A_*\bigr) \int_G \nabla A \cdot \bv \srd \bx\\
        &\ge \lambda \| \tbu \|_{\rL^2(G)}^2 + C \left(\| \bv \|_{\rH^1(G)}^2 + \| \nabla h \|_{\rL^2(G)}^2 + \| \nabla A \|_{\rL^2(G)}^2\right).
    \end{aligned}
\end{equation}
From the previous estimate, we deduce that $\lambda \in \bR$ and $\lambda \le 0$.
By virtue of the $q$-independence of the spectrum of $\ALFI_0$, we obtain $\sigma(\ALFI_0) \setminus \{0\} \subset \bC_+$.
From $\lambda = 0$, we derive that $\bv = 0$ as well as $h$ and $A$ constant for $\bu = (\bv,h,A) \in \rN(\ALFI_0)$, as asserted in~(i).

Let us now elaborate on the proof of~(ii).
In this case, it is not possible to cancel the terms related with $h_* \div \bv$ and $A_* \div \bv$ with the ice strength $P'$, since the latter vanishes.
Instead, it is necessary to absorb these terms into the terms resulting from the fast ice operator and the Neumann Laplacians.
In fact, in this case, we test the eigenvalue problem \eqref{eq:eigenvalue probl} by $\bu = (\bv,h,A)$ and obtain
\begin{equation}\label{eq:eigenvalue probl k_T = 1}
    \begin{aligned}
        0
        &= \lambda \| \bu \|_{\rL^2(G)}^2 - \int_G \AFID(\bu_*) \bv \cdot \bv \srd \bx + d_\rh \cdot \| \nabla h \|_{\rL^2(G)}^2 + d_\rA \cdot \| \nabla A \|_{\rL^2(G)}^2\\
        &\quad - h_* \int_G \nabla h \cdot \bv \srd \bx - A_* \int_G \nabla A \cdot \bv \srd \bx.
    \end{aligned}
\end{equation}
Thus, using Young's inequality, for $\gamma > 0$, we find that 
\begin{equation*}
    - h_* \int_G \nabla h \cdot \bv \srd \bx - A_* \int_G \nabla A \cdot \bv \srd \bx \ge - h_* \frac{\gamma}{2} \cdot \| \nabla h \|_{\rL^2(G)}^2 - A_* \frac{\gamma}{2} \cdot \| \nabla A \|_{\rL^2(G)}^2 - \frac{h_* + A_*}{2 \gamma} \cdot \| \bv \|_{\rL^2(G)}^2.
\end{equation*}
We then choose $0 < \gamma < \min\left\{2 \frac{d_\rh}{h_*},2 \frac{d_\rA}{A_*}\right\} \eqqcolon \gamma_*$ sufficiently small.
With regard to \eqref{eq:est of fast ice op}, there exists a constant $C_{\mathrm{K},\mathrm{P}} > 0$ depending on the constants from the Korn and Poincar\'e inequalities such that
\begin{equation*}
    -\int_G \AFID(\bu_*) \bv \cdot \bv \srd \bx \ge \frac{\tP_* C_{\mathrm{K},\mathrm{P}}}{e^2 \shadedDelta_{\min}} \cdot \| \bv \|_{\rH^1(G)}^2.
\end{equation*}
Therefore, if $0 < \shadedDelta_{\min} < \frac{\tP_* C_{\mathrm{K},\mathrm{P}}}{e^2} \cdot \frac{2 \gamma}{h_* + A_*} \eqqcolon \shadedDelta_*$ for $\gamma < \gamma_*$ as specified above, from \eqref{eq:eigenvalue probl k_T = 1}, it follows that there is $\Tilde{C} > 0$ such that
\begin{equation}\label{eq:tested eigenvalue eq k_T = 1}
    0 \ge \lambda \| \bu \|_{\rL^2(G)}^2 + \Tilde{C} \left(\| \bv \|_{\rH^1(G)}^2 + \| \nabla h \|_{\rL^2(G)}^2 + \| \nabla A \|_{\rL^2(G)}^2\right),
\end{equation}
so the assertion of~(ii) can be deduced in the same way as above for~(i).
\end{proof}

In the next lemma, we elaborate on the shape of the set of equilibria $\cE$ close to an equilibrium of the aforementioned form.
Again, the case $k_{\rt} \equiv 1$ requires special treatment.

\begin{lem}\label{lem:shape of equil}
Let $p$, $q \in (1,\infty)$ be such that \eqref{eq:cond p and q glob small data} is valid, and let $\bu_* = (0,h_*,A_*)$ be an equilibrium with $h_* > \kappa$ and $A_* \in (0,1]$ constant.
\begin{enumerate}[(i)]
    \item If $k_{\rt} \in [0,1)$, then there exists $r = r(d_\rh,d_\rA) > 0$ such that in $B_{\rX_\gamma}(\bu_*,r)$, the set of equilibria $\cE$ is a $\rC^1$-manifold in $\rX_1$, and the tangent space of $\cE$ at $\bu_*$ is isomorphic to $\rN(\ALFI_0) = \{0\} \times \bR \times \bR$.
    \item For $k_{\rt} \equiv 1$, there is $\shadedDelta_* = \shadedDelta_*(d_\rh,d_\rA) > 0$ such that if $\shadedDelta_{\min} < \shadedDelta_*$, there is $r = r(d_\rh,d_\rA) > 0$ so that in $B_{\rX_\gamma}(\bu_*,r)$, the set of equilibria $\cE$ is a $\rC^1$-manifold in $\rX_1$, and the tangent space of $\cE$ at~$\bu_*$ is isomorphic to $\rN(\ALFI_0) = \{0\} \times \bR \times \bR$.
\end{enumerate}
\end{lem}

\begin{proof}
Consider an equilibrium $\bu = (\bv,h,A) \in V \cap \rX_1$ such that $\| \bu - \bu_* \|_{\rX_\gamma} < r$, where $r = r(d_\rh,d_\rA) > 0$ will be made precise in the following.
Again, we first treat the case $k_{\rt} \in [0,1)$ and then deal with the case $k_{\rt} \equiv 1$, since in the latter case, we have $P'(h,A) = 0$.

From the resulting relation $\ALFI(\bu) \bu = F_s(\bu)$, it follows that
\begin{equation}\label{eq:rel equil}
    0 = \begin{pmatrix}
        2 \rho h \left(-\AFID(\bu) \bv + (\bv \cdot \nabla) \bv\right) + \nabla P'(h,A)\\
        -d_\rh \Delta_\rN h + \div(\bv h)\\
        -d_\rA \Delta_\rN A + \div(\bv A)
    \end{pmatrix}.
\end{equation}
From $\bu \in V$ such that $\| \bu - \bu_* \|_{\rX_\gamma} < r$, it first follows that there is a constant $c_e > 0$ with
\begin{equation*}
    \tP(h,A) \ge (1+k_{\rt}) \kappa P^* \exp(-C) \eqqcolon \tP_{**} \tand \frac{1}{\shadedDelta(\dbeps(\bv))} \ge \frac{1}{\sqrt{\shadedDelta_{\min}^2 + c_e r^2}}.
\end{equation*}
Thus, an integration by parts joint with the homogeneous Dirichlet boundary conditions of $\bv$, the identity $\shadedDelta^2(\dbeps) = \dbeps^\top \bS \dbeps$, and Poincar\'e's as well as Korn's inequality yield that
\begin{equation*}
    \begin{aligned}
        -\int_G 2 \rho h \AFID(\bu) \bv \cdot \bv \srd \bx 
        &= -\int_G \div\left(\tP(h,A) \frac{\bS \dbeps}{\shadedDelta(\dbeps)}\right) \cdot \bv \srd \bx\\
        &\ge \frac{\tP_{**}}{\sqrt{\shadedDelta_{\min}^2 + c_e r^2}} \int_G \dbeps^\top \bS \dbeps \srd \bx \ge \frac{C}{\sqrt{\shadedDelta_{\min}^2 + c_e r^2}} \cdot \| \bv \|_{\rH^1(G)}^2
    \end{aligned}
\end{equation*}
For the convective term, we use H\"older's inequality together with the embedding $\rX_\gamma \hookrightarrow \rC^1(\oG)^4$, the shape of the equilibrium and the assumption that $\| \bu - \bu_* \|_{\rX_\gamma} < r$ to infer that
\begin{equation*}
    2 \rho \int_G h(\bv \cdot \nabla) \bv \cdot \bv \srd \bx \le C \cdot \| h \|_{\rL^\infty(G)} \cdot \| \nabla \bv \|_{\rL^\infty(G)} \cdot \| \bv \|_{\rL^2(G)}^2 \le C'(1+r)r\cdot \| \bv \|_{\rH^1(G)}^2
\end{equation*}
for some constant $C' > 0$.
Now, we set $c_1 \coloneqq 2 \rho P_{h,*}'$, $c_2 \coloneqq 2 \rho P_{A,*}'$, and integrate by parts to get
\begin{equation*}
    \begin{aligned}
        &\quad \int_G \nabla P'(h,A) \cdot \bv \srd \bx + c_1 \int_G \div(\bv h) h \srd \bx + c_2 \int_G \div(\bv A) A \srd \bx\\
        &= \int_G \bigl(\del_h P'(h,A) - c_1 h\bigr) \nabla h \cdot \bv \srd \bx + \int_G \bigl(\del_A P'(h,A) - c_2 A\bigr) \nabla A \cdot \bv \srd \bx.
    \end{aligned}
\end{equation*}
Therefore, testing \eqref{eq:rel equil} by $(\bv,c_1 h,c_2 A)$, we find that
\begin{equation}\label{eq:tested equil rel}
    \begin{aligned}
        0
        &\ge \left(\frac{C}{\sqrt{\shadedDelta_{\min}^2 + c_e r^2}} - C' (1+r)r\right) \cdot \| \bv \|_{\rH^1(G)}^2 + c_1 d_\rh \cdot \| \nabla h \|_{\rL^2(G)}^2 + c_2 d_\rA \cdot \| \nabla A \|_{\rL^2(G)}^2\\
        &\quad + \int_G \bigl(\del_h P'(h,A) - c_1 h\bigr) \nabla h \cdot \bv \srd \bx + \int_G \bigl(\del_A P'(h,A) - c_2 A\bigr) \nabla A \cdot \bv \srd \bx.
    \end{aligned}
\end{equation}
Straightforward calculations reveal that there exist $C_1$, $C_2 > 0$ such that
\begin{equation*}
    \| \del_h P'(h,A) - c_1 h \|_{\rL^\infty(G)} \le C_1 (1+r)r \tand \| \del_A P'(h,A) - c_2 A \|_{\rL^\infty(G)} \le C_2 (1+r)r.
\end{equation*}
Inserting this into \eqref{eq:tested equil rel}, using H\"older's as well as Young's inequality, and choosing $r > 0$ sufficiently small, we argue that
\begin{equation*}
    \begin{aligned}
        0
        &\ge \left(\frac{C}{\sqrt{\shadedDelta_{\min}^2 + c_e r^2}} - C' (1+r)r\right) \cdot \| \bv \|_{\rH^1(G)}^2 + \bigl(c_1 d_\rh - C_1 (1+r)r) \cdot \| \nabla h \|_{\rL^2(G)}^2\\
        &\quad + \bigl(c_2 d_\rA - C_2 (1+r)r\bigr) \cdot \| \nabla A \|_{\rL^2(G)}^2\\
        &\ge \Tilde{C}\left(\| \bv \|_{\rH^1(G)}^2 + \| \nabla h \|_{\rL^2(G)}^2 + \| \nabla A \|_{\rL^2(G)}^2\right).
    \end{aligned}
\end{equation*}
In other words, an equilibrium $\bu \in V \cap \rX_1$ such that $\| \bu - \bu_* \|_{\rX_\gamma} < r$ for sufficiently small $r > 0$ satisfies $\bv = 0$ as well as $h$ and $A$ constant, so
\begin{equation*}
    B_{\rX_\gamma \cap \cE}(\bu_*,r) = \{0\} \times \bR \times \bR = \rN(\ALFI_0)
\end{equation*}
by \autoref{lem:spectral props of tot lin}(b).
In particular, near $\bu_*$, the set $\cE$ is a two-dimensional $\rC^1$-manifold, and the tangent space of $\cE$ near $\bu_*$ is especially isomorphic to $\rN(\ALFI_0)$.
This completes the proof of~(i).

With regard to~(ii), we observe that it is again not possible to use $P'$ to absorb the terms associated with $\div(\bv h)$ and $\div(\bv A)$, since $P'$ vanishes in the present case $k_{\rt} \equiv 1$.
Since we are interested in $r > 0$ small, we may consider $r \le 1$.
It then follows that there exists a constant $C' > 0$ such that
\begin{equation*}
    \| h \|_{\rL^\infty(G)} \le \left(h_* + \| h - h_* \|_{\rL^\infty(G)}\right) \le \frac{C'}{2} \tand \| A \|_{\rL^\infty(G)} \le \left(A_* + \| A - A_* \|_{\rL^\infty(G)}\right) \le \frac{C'}{2}.
\end{equation*}
Thus, for $\gamma > 0$, an integration by parts joint with Young's inequality and the preceding estimate implies
\begin{equation*}
    \int_G \div(\bv h)h \srd \bx + \int_G \div(\bv A)A \srd \bx \ge -\frac{C'\gamma}{4} \cdot \| \nabla h \|_{\rL^2(G)}^2 - \frac{C' \gamma}{4} \cdot \| \nabla A \|_{\rL^2(G)}^2 - \frac{C'}{2 \gamma} \cdot \| \bv \|_{\rL^2(G)}^2.
\end{equation*}
In the present case, we test \eqref{eq:rel equil} by $(\bv,h,A)$, and the above estimates lead to
\begin{equation*}
    0 \ge \left(\frac{C}{\sqrt{\shadedDelta_{\min}^2 + c_e r^2}} - \frac{C'}{2 \gamma}\right) \cdot \| \bv \|_{\rH^1(G)}^2 + \left(d_\rh - \frac{C' \gamma}{4}\right) \cdot \| \nabla h \|_{\rL^2(G)}^2 + \left(d_\rA - \frac{C' \gamma}{4}\right) \cdot \| \nabla A \|_{\rL^2(G)}^2.
\end{equation*}
We choose $0 < \gamma < \min\left\{\frac{4 d_\rh}{C'},\frac{4 d_\rA}{C'}\right\} \eqqcolon \gamma_*$, $0 < r < \frac{\sqrt{2} \gamma C}{\sqrt{c_e} C'}$, and finally $0 < \shadedDelta_{\min} < \frac{\sqrt{2} \gamma C}{C'} \eqqcolon \shadedDelta_*$.
In total, this yields the existence of a constant $\Tilde{C} > 0$ such that
\begin{equation*}
    0 \ge \Tilde{C} \left(\| \bv \|_{\rH^1(G)}^2 + \| \nabla h \|_{\rL^2(G)}^2 + \| \nabla A \|_{\rL^2(G)}^2\right),
\end{equation*}
so the assertion of~(ii) follows in the same way as that of~(i).
\end{proof}

The following lemma reveals that zero is a semi-simple eigenvalue of the total linearization $\ALFI_0$.

\begin{lem}\label{lem:zero semi-simple ev}
Consider $\bu_* = (0,h_*,A_*)$, where $h_* > \kappa$ and $A_* \in (0,1]$ are constant, and recall the associated total linearization $\ALFI_0$ from \eqref{eq:tot lin}.
\begin{enumerate}[(i)]
    \item If $k_{\rt} \in [0,1)$, then $\rN(\ALFI_0) \oplus \rR(\ALFI_0) = \rX_0$, so zero is a semi-simple eigenvalue of $\ALFI_0$.
    \item If $k_{\rt} \equiv 1$, there exists $\shadedDelta_* > 0$ so that if $\shadedDelta_{\min} < \shadedDelta_*$, then $\rN(\ALFI_0) \oplus \rR(\ALFI_0) = \rX_0$.
\end{enumerate}
\end{lem}

\begin{proof}
Note that in the present proof, we will not strictly distinguish the cases $k_{\rt} \in [0,1)$ and $k_{\rt} \equiv 1$, but we will express when we use modified results for $k_{\rt} \equiv 1$.

For $\rL_0^q(G) = \left\{f \in \rL^q(G) : \int_G f \srd \bx = 0\right\}$ denoting the space of functions with average zero in $\rL^q(G)$, we introduce the modified ground space
\begin{equation}\label{eq:mod ground space}
    \rX_0^{\rm} \coloneqq \rL^q(G)^2 \times \rL_0^q(G) \times \rL_0^q(G),
\end{equation}
while $\ALFI_{0,\rm}$ represents the restriction of the total linearization $\ALFI_0$ to $\rX_0^{\rm}$.
Testing $\ALFI_0 \bu = 0$ by $\bu \in \rD(\ALFI_{0,\rm})$, as in \eqref{eq:tested eigenvalue eq}, see also \eqref{eq:tested eigenvalue eq k_T = 1} for the case $k_{\rt} \equiv 1$, we find that there exists $C > 0$ such that
\begin{equation*}
    0 \ge C \cdot \left(\| \bv \|_{\rH^1(G)}^2 + \| \nabla h \|_{\rL^2(G)}^2 + \| \nabla A \|_{\rL^2(G)}^2\right).
\end{equation*}
Since $h$, $A \in \rL_0^q(G)$, in addition to $\bv = 0$, it also follows that $h = A = 0$.
Note that the operator~$\ALFI_{0,\rm}$ also has a compact resolvent, so the spectrum of this operator is also $q$-independent and consists of eigenvalues, so $0 \in \rho(\ALFI_{0,\rm})$.
By $\rN(\ALFI_0) = \{0\} \times \bR \times \bR$, to verify that $\rX_0 = \rN(\ALFI_0) + \rR(\ALFI_0)$, it suffices to prove that
\begin{equation*}
    \rL^q(G)^2 \times \rL_0^q(G) \times \rL_0^q(G) \subset \rR(\ALFI_0).
\end{equation*}
To this end, let $\bff \in \rL^q(G)^2 \times \rL_0^q(G) \times \rL_0^q(G)$.
Thanks to $0 \in \rho(\ALFI_0)$, we find that there is $\bu \in \rD(\ALFI_{0,\rm})$ such that $\ALFI_0 \bu = \ALFI_{0,\rm} \bu = \bff$, as desired.

To complete the proof, we need to show that $\rN(\ALFI_0) \cap \rR(\ALFI_0) = \{0\}$, so let $\bu \in \rN(\ALFI_0) \cap \rR(\ALFI_0)$.
The above form of $\rN(\ALFI_0)$ reveals that $\bu = (0,c_\rh,c_\rA)$ for constants $c_\rh$ and $c_\rA$.
By $\bu \in \rR(\ALFI_0)$, we find that there is $\tbu \in \rD(\ALFI_0)$ so that $\ALFI_0 \tbu = \bu$, and we split into the mean value zero part and the average part, so
\begin{equation*}
    \tbu = (\tbv,\th,\tA) = (\tbv,\th_\rm,\tA_\rm) + (0,\th_\avg,\tA_\avg) \eqqcolon \tbu_\rm + \tbu_\avg.
\end{equation*}
By construction, we have $\tbu_\rm \in \rX_0^\rm$, so $\tbu_\rm \in \rD(\ALFI_{0,\rm})$, while $\tbu_\avg \in \rN(\ALFI_0)$.
This implies that
\begin{equation*}
    \bu = \ALFI_0 \tbu = \ALFI_0 \tbu_\rm + \ALFI_0 \tbu_\avg = \ALFI_{0,\rm} \tbu_\rm,
\end{equation*}
i.e., $\bu \in \rR(\ALFI_{0,\rm})$, which in turn yields that $c_\rh$, $c_\rA \in \rL_0^q(G)$.
This requires that $c_\rh = c_\rA = 0$, so $\bu = 0$ and thus also $\rN(\ALFI_0) \cap \rR(\ALFI_0) = \{0\}$, finishing the proof.
\end{proof}

The previous lemmas enable us to prove \autoref{thm:glob ex close to equil}.

\begin{proof}[Proof of \autoref{thm:glob ex close to equil}]
We are in the scope of \autoref{lem:gen princ of lin stab}.
In fact, the shape of the set of equilibria $\cE$ as required in \autoref{lem:gen princ of lin stab}(a) and~(b) has been discussed in \autoref{lem:shape of equil}, in \autoref{lem:zero semi-simple ev}, we have shown that zero is a semi-simple eigenvalue of $\ALFI_0$ as needed in \autoref{lem:gen princ of lin stab}(c), and in \autoref{lem:spectral props of tot lin}, we have determined the spectrum of the total linearization $\ALFI_0$ as necessary for \autoref{lem:gen princ of lin stab}(d).
\end{proof}

\section{Time-periodic solutions to the landfast ice model}\label{sec:time-per sols}

In this section, we discuss the existence of a time-periodic strong solution to the landfast ice model in the presence of time-periodic forcing terms as stated in \autoref{thm:time-per sols}.
A particular application is the presence of time-periodic wind forces as in the simulations shown in \autoref{fig:Ex2}.
The main idea here is to use the Arendt-Bu theorem, see \autoref{lem:Arendt-Bu}.
A key difficulty is that in view of the Neumann Laplacian operators in the balance laws, the linearized operator matrix is not invertible.
We thus consider time-periodic solutions as deviations of equilibrium solutions and introduce a modified functional analytic framework involving spaces of functions with average zero as in the proof of \autoref{lem:zero semi-simple ev}.

For a fixed time period $T > 0$, and for suitable $T$-periodic external forcing terms $f_\ice$, $f_\rh$ as well as $f_\rA$, the landfast ice model in the periodic setting is given by 
\begin{equation}\label{eq:time-per landfast ice model}
    \left\{
    \begin{aligned}
        \del_t \bv + (\bv \cdot \nabla) \bv
        &= \frac{1}{\rho h} \div(\bsigma) - C_{\cor} \bk \times (\bv - \bv_\ro)\\
        &\quad + \frac{1}{\rho h}\bigl(f_\ra + f_\ro(\bv) + f_\rb(\bv,h,A)\bigr) + f_\ice, &&\tin G \times \bR,\\
        \del_t h + \div(\bv h)
        &= d_\rh \Delta h + f_\rh, &&\tin G \times \bR,\\
        \del_t A + \div(\bv A)
        &= d_\rA \Delta A + f_\rA, &&\tin G \times \bR,\\
        \bv = 0, \enspace \del_\nu h&= \del_\nu A = 0, &&\ton \del G \times \bR,\\
        \bv(t) = \bv(t+T), \enspace h(t)
        &= h(t+T), \enspace A(t) = A(t+T), &&\tfor t \in \bR.
    \end{aligned}
    \right.
\end{equation}
Next, we invoke a constant equilibrium solution of the form $(0,h_*,A_*)$ to the unforced problem as investigated in \autoref{sec:glob wp close to equil}.
If $\bu = (\bv,h,A)$ is a solution to \eqref{eq:time-per landfast ice model}, then $\tbu \coloneqq \bu - \bu_* = (\tbv,\th,\tA)$ solves
\begin{equation}\label{eq:time-per landfast ice model around equil}
    \left\{
    \begin{aligned}
        \del_t \tbv - \frac{1}{\rho (\th + h_*)} \div(\bsigma)
        &= -(\tbv \cdot \nabla) \tbv - C_{\cor} \bk \times (\tbv - \bv_\ro) + \frac{1}{\rho (\th + h_*)}\bigl(f_\ra\\
        &\qquad + f_\ro(\tbv) + f_\rb(\tbv,\th + h_*,\tA + A_*)\bigr) + f_\ice, &&\tin G \times \bR,\\
        \del_t \th - d_\rh \Delta \th
        &= - h_* \div \tbv - \div(\tbv \th) + f_\rh, &&\tin G \times \bR,\\
        \del_t \tA - d_\rA \Delta \tA
        &= - A_* \div \tbv - \div(\tbv \tA) + f_\rA, &&\tin G \times \bR,\\
        \tbv = 0, \enspace \del_\nu \th&= \del_\nu \tA = 0, &&\ton \del G \times \bR,\\
        \tbv(t) = \tbv(t+T), \enspace \th(t)
        &= \th(t+T), \enspace \tA(t) = \tA(t+T), &&\tfor t \in \bR.
    \end{aligned}
    \right.
\end{equation}

We now reformulate \eqref{eq:time-per landfast ice model around equil} as a time-periodic quasilinear evolution equation.
In fact, due to the lack of invertibility of the Neumann Laplacian operator, we invoke here the modified ground space $\rX_0^\rm$ from~\eqref{eq:mod ground space}, where we recall that in the $h$- and $A$-component, we consider the space of $\rL^q$-functions with average zero, denoted by $\rL_0^q(G)$.
Accordingly, we introduce the modified regularity space
\begin{equation*}
    \rX_1^\rm \coloneqq \rW^{2,q}(G)^2 \cap \rW_0^{1,q}(G)^2 \times \rW_\rN^{2,q}(G) \cap \rL_0^q(G) \times \rW_\rN^{2,q}(G) \cap \rL_0^q(G).
\end{equation*}
When considering $p$, $q \in (1,\infty)$ such that \eqref{eq:cond p and q glob small data}, and upon invoking \cite[Sec.~5]{PSW:18} for the interpolation of the average zero condition, we find that the trace space $\rX_\gamma^\rm = (\rX_0^\rm,\rX_1^\rm)_{1-\nicefrac{1}{p},p}$ is given by
\begin{equation*}
    \rX_\gamma^\rm = \rB_{qp,\rD}^{2-\nicefrac{2}{p}}(G)^2 \times \rB_{qp,\rN}^{2-\nicefrac{2}{p}}(G) \cap \rL_0^q(G) \times \rB_{qp,\rN}^{2-\nicefrac{2}{p}}(G) \cap \rL_0^q(G).
\end{equation*}
As observed in the proof of \autoref{lem:shape of equil}, it holds that $\rX_\gamma^\rm \hookrightarrow \rC^1(\oG)^4$.
To make sure that $h$ and~$A$ attain values in their physically reasonable ranges, we introduce an open subset $W \subset \rB_{qp}^{2-\nicefrac{2}{p}}(G)^4$, so we set
\begin{equation}\label{eq:open set W}
    W \coloneqq \left\{\bu = (\bv,h,A) \in \rB_{qp}^{2-\nicefrac{2}{p}}(G)^4 : h > \kappa \tand A \in (0,1+\delta)\right\}.
\end{equation}

Below, we show that indeed, $\bu = \tbu + \bu_* \in W$ if $\bu_* \in W$, and the remaining part $\tbu = (\tbv,\th,\tA)$ is small enough in the modified trace space $\rX_\gamma^\rm$, or in the maximal regularity space
\begin{equation*}
    \bE_1^\rm \coloneqq \rW^{1,p}(0,T;\rX_0^\rm) \cap \rL^p(0,T;\rX_1^\rm).
\end{equation*}
Thus, let $\bu_* \in W$.
Since $W$ is open, there exists $r_0' > 0$ such that
\begin{equation*}
    \oB_{\rB_{qp}^{2-\nicefrac{2}{p}}(G)^4}(\bu_*,r_0') \subset  W.
\end{equation*}
Moreover, we invoke the embedding $\bE_1^\rm \hookrightarrow \rBUC([0,T];\rX_\gamma^\rm)$, see, e.g., \cite[Thm.~III.4.5.10]{Ama:95}, yielding that
\begin{equation*}
    \sup_{t \in [0,T]} \| \bu(t) - \bu_* \|_{\rB_{qp}^{2-\nicefrac{2}{p}}(G)} \le C \cdot \sup_{t \in [0,T]} \| \tbu(t) \|_{\rX_\gamma^\rm} \le C \cdot \| \tbu \|_{\bE_1^\rm}.
\end{equation*}
Thus, setting 
\begin{equation}\label{eq:r_0}
    r_0 \coloneqq \frac{r_0'}{C},
\end{equation}
and considering $\| \tbu \|_{\bE_1^\rm} \le r_0$, we find that $\bu(t) = \tbu(t) + \bu_* \in W$ for all $t \in [0,T]$.
For completeness, we also introduce the data space $\bE_0^\rm \coloneqq \rL^p(0,T;\rX_0^\rm)$ tailored to the present time-periodic problem.

We now make precise the assumptions on the forcing terms.
Note that we impose a smallness assumption on the maximal basal stress parameter $k_2$ provided the $h$-component $h_*$ of the equilibrium exceeds the critical ice thickness $h_{\crit}$ for the basal stress.
Otherwise, with regard to \eqref{eq:basal stress}, the basal stress is absent.

\begin{asu}\label{ass:atm & ocn vels}
For $p$, $q \in (1,\infty)$, suppose that the surface wind and ocean velocities $\bv_\ra$ and $\bv_\ro$ fulfill $\bv_\ra$, $\bv_\ro \in \rL^{2p}(0,T;\rL^{2q}(G)^2)$ as well as $\bv_\ra(0) = \bv_\ra(T)$ and $\bv_\ro(0) = \bv_\ro(T)$.
Furthermore, assume that there exists $\delta_1 > 0$ such that
\begin{equation*}
    \| \bv_\ra \|_{\rL^{2p}(0,T;\rL^{2q}(G))} + \| \bv_\ro \|_{\rL^{2p}(0,T;\rL^{2q}(G))} < \delta_1.
\end{equation*}
If $h_* \ge h_{\crit}$, assume that there exists $\delta_1 > 0$ such that
\begin{equation*}
    \| \bv_\ra \|_{\rL^{2p}(0,T;\rL^{2q}(G))} + \| \bv_\ro \|_{\rL^{2p}(0,T;\rL^{2q}(G))} + k_2 < \delta_1,
\end{equation*}
where we recall the maximal basal stress parameter $k_2$ from \autoref{sec:landfast ice model}.
\end{asu}

We now state the third main result of this paper on the existence of a time-periodic strong solution.
Similarly as for the global well-posedness close to constant equilibria as asserted in \autoref{thm:glob ex close to equil}, we deal with the case of $k_{\rt} \equiv 1$ separately.
Again, this induces a smallness condition on $\shadedDelta_{\min} > 0$.
Let us also briefly describe the main analytical difficulties of this result.
First, the linearization is non-invertible on the natural space $\rL^q(G)^2 \times \rL^q(G) \times \rL^q(G)$, since the scalar means are neutral modes.
To overcome this, the problem must be reformulated relative to an equilibrium, and the mean-zero subspace has to be shown to be invariant.
Maximal regularity and invertibility of the resulting complete coupled operator must then be established.
In contrast to a purely topological fixed-point argument, the maximal-regularity construction yields local uniqueness in the stated strong-solution class.

\begin{thm}\label{thm:time-per sols}
Let $p$, $q \in (1,\infty)$ satisfy \eqref{eq:cond p and q glob small data}, let $\bu_* = (0,h_*,A_*)$, with $h_* > \kappa$ and $A_* \in (0,1]$ constant, and invoke $r_0 > 0$ from \eqref{eq:r_0}.
Besides, assume that $f_\per = (f_\ice,f_\rh,f_\rA) \colon \bR \to \rX_0^\rm$ is $T$-periodic with $\left.f_\per\right|_{(0,T)} \in \bE_0^\rm$.
\begin{enumerate}[(i)]
    \item Let $k_{\rt} \in [0,1)$.
    Then there are $r_1 > 0$ and $\delta_1 > 0$ so that for all $r \in (0,r_1)$, and if $\bv_\ra$, $\bv_\ro$, and $k_2$ satisfy \autoref{ass:atm & ocn vels} for $\delta_1 > 0$, there exists $\delta_2 > 0$ such that if $\| \left. f_\per \right|_{(0,T)} \|_{\bE_0^\rm} < \delta_2$, there is a strong $T$-periodic solution $\tbu \colon \bR \to \rX_0^\rm$ to \eqref{eq:time-per landfast ice model around equil}, the solution $\tbu$ satisfies $\left. \tbu \right|_{(0,T)} \in \oB_{\bE_1^\rm}(0,r)$, and it is unique in $\oB_{\bE_1^\rm}(0,r)$.
    \item Consider $k_{\rt} \equiv 1$.
    Then there is $\shadedDelta_* > 0$ so that if $\shadedDelta_{\min} < \shadedDelta_*$, the assertion of~(i) remains valid.
\end{enumerate}
\end{thm}

A few remarks on the previous theorem are in order now.

\begin{rem}
\begin{enumerate}[(a)]
    \item For the solution $\tbu \in \bE_1^\rm$ from \autoref{thm:time-per sols}, it holds that $\bu \coloneqq \tbu + \bu_*$, with $\bu_* = (0,h_*,A_*)$ solves the problem \eqref{eq:time-per landfast ice model}.
    \item The discussion around \eqref{eq:r_0}, and the definition of the open set $W$ in \eqref{eq:open set W} reveal that the $h$- and $A$-component of the solution $\bu$ to original problem take values in their physically relevant ranges.
\end{enumerate}
\end{rem}

In the remainder of this section, we show \autoref{thm:time-per sols} by verifying that we are in the scope of \autoref{lem:quasilin time per result}.
This especially involves establishing that the linearization at zero satisfies the properties required by the Arendt-Bu theorem, \autoref{lem:Arendt-Bu}.
In \autoref{ssec:dyn evol to stat landfast ice config}, we will discuss the relation with the numerical simulations as presented in \autoref{fig:Ex2}.

For $\bu_* = (0,h_*,A_*) \in V$, and with $\tbu \in V \coloneqq B_{\rX_\gamma^\rm}(0,r_0)$, the discussion around \eqref{eq:r_0} reveals that $\th(t) + h_* > \kappa$ for all $t \in [0,T]$ by the choice of $r_0$.
In particular, additionally invoking the embedding $\rB_{qp}^{2-\nicefrac{2}{p}}(G) \hookrightarrow \rC^1(\oG)$, we find that the landfast ice operator $\AFID(\tbu + \bu_*)$ as made precise in \eqref{eq:Lq-realization of the fast ice op} is in particular well-defined.
Furthermore, we introduce the $\rL_0^q(G)$-realization of the Neumann Laplacian operator.
It is defined by
\begin{equation*}
    \Delta_{\rN,\rm} f \coloneqq \Delta f, \tfor f \in \rD(\Delta_{\rN,\rm}) = \rW_{\rN}^{2,q}(G) \cap \rL_0^q(G).
\end{equation*}
The operator matrix $\ALFI_\per \colon V \to \cL(\rX_1^\rm,\rX_0^\rm)$ associated to the present problem is then given by
\begin{equation*}
    \ALFI_\per(\tbu) \coloneqq \begin{pmatrix}
        -\AFID(\tbu + \bu_*) & \frac{\del_h P'(\th+h_*,\tA+A_*)}{2 \rho(\th + h_*)} \nabla & \frac{\del_A P'(\th+h_*,\tA+A_*)}{2 \rho(\th + h_*)} \nabla\\
        h_* \div & -d_\rh \Delta_{\rN,\rm} & 0\\
        A_* \div & 0 & -d_\rA \Delta_{\rN,\rm}
    \end{pmatrix}.
\end{equation*}
We also make precise the right-hand side $F_\per \colon \bR \times V \to \rX_0^\rm$ capturing the terms of lower order in \eqref{eq:time-per landfast ice model around equil}.
For $f_\per = (f_\ice,f_\rh,f_\rA)^\top$, it takes the form
\begin{equation*}
    F_\per(t,\tbu) \coloneqq \begin{pmatrix}
        -(\tbv \cdot \nabla) \tbv - C_{\cor} \bk \times (\tbv - \bv_\ro) + \frac{1}{\rho (\th + h_*)}\bigl(f_\ra + f_\ro(\tbv) + f_\rb(\tbv,\th + h_*,\tA + A_*)\bigr)\\
        -\div(\tbv \th)\\
        -\div(\tbv \tA)
    \end{pmatrix} + f_\per(t).
\end{equation*}
The above pieces of notation allow us to reformulate the time-periodic problem \eqref{eq:time-per landfast ice model around equil} as a quasilinear evolution equation
\begin{equation*}
    \left\{
    \begin{aligned}
        \tbu'(t) + \ALFI_\per(\tbu(t)) \tbu(t)
        &= F_\per(t,\tbu(t)), &&\tfor t \in \bR,\\
        \tbu(t)
        &= \tbu(t+T), &&\tfor t \in \bR.
    \end{aligned}
    \right.
\end{equation*}

For the analysis of the time-periodic problem, we first investigate the linearization of the operator matrix $\ALFI_\per(\tbu)$ at zero.
In fact, invoking the pieces of notation $P_{h,*}'$ and $P_{A,*}'$ from \eqref{eq:shorthand not P}, we define
\begin{equation}\label{lin:op matrix time-per}
    \ALFI_{\per,0} \coloneqq \ALFI_\per(0) = \begin{pmatrix}
        -\AFID(\bu_*) & P_{h,*}' \nabla & P_{A,*}' \nabla\\
        h_* \div & -d_\rh \Delta_{\rN,\rm} & 0\\
        A_* \div & 0 & -d_\rA \Delta_{\rN,\rm}
    \end{pmatrix}.
\end{equation}
Note that $\ALFI_{\per,0}$ coincides with the total linearization $\ALFI_0$ from \eqref{eq:tot lin} when the latter is restricted to the closed subspace $\rX_0^\rm \subset \rX_0$.
We capitalize on this fact in the analysis of $\ALFI_{\per,0}$, see the lemma below.

\begin{lem}\label{lem:max reg & inv of lin op matrix}
Let $\bu_* = (0,h_*,A_*)$, where $h_* > \kappa$ and $A_* \in (0,1]$ are constant.
\begin{enumerate}[(i)]
    \item If $k_{\rt} \in [0,1)$, then the operator matrix $\ALFI_{\per,0}$ from \eqref{lin:op matrix time-per} has maximal $\rL^p$-regularity on $\rX_0^\rm$, and it holds that $0 \in \rho(\ALFI_{\per,0})$.
    \item In the case $k_{\rt} \equiv 1$, there exists $\shadedDelta_* > 0$ such that if $\shadedDelta_{\min} < \shadedDelta_*$, the operator matrix $\ALFI_{\per,0}$ has maximal $\rL^p$-regularity on $\rX_0^\rm$, and $0 \in \rho(\ALFI_{\per,0})$.
\end{enumerate}
\end{lem}

\begin{proof}
With regard to the above observation, and as a by-product of the proof of \autoref{lem:zero semi-simple ev}, where we especially invoke the case distinction of $k_{\rt} \in [0,1)$ and $k_{\rt} \equiv 1$, we find that $0 \in \rho(\ALFI_{\per,0})$.
In order to prove the maximal $\rL^p$-regularity, we split the operator as follows
\begin{equation*}
    \ALFI_{\per,0} = \ALFI_{\per,1} + \ALFI_{\per,2} \coloneqq \begin{pmatrix}
        -\AFID(\bu_*) & P_{h,*}' \nabla & P_{A,*}' \nabla\\
        0 & -d_\rh \Delta_{\rN,\rm} & 0\\
        0 & 0 & -d_\rA \Delta_{\rN,\rm}
    \end{pmatrix} + \begin{pmatrix}
        0 & 0 & 0\\
        h_* \div & 0 & 0\\
        A_* \div & 0 & 0
    \end{pmatrix}.
\end{equation*}
Upon invoking that $\Delta_{\rN,\rm}$ admits maximal $\rL^p$-regularity on $\rL_0^q(G)$, see, for example, \cite[Lemma~2.3.20]{Bra:24}, similarly as in \autoref{prop:bdd Hoo-calculus op matrix landfast ice}, we obtain the existence of $\omega_0 \in \bR$ such that for all $\omega > \omega_0$, the shifted operator $\ALFI_{\per,1} + \omega$ has maximal $\rL^p$-regularity on $\rX_0^\rm$.

It remains to handle $\ALFI_{\per,2}$ by means of a perturbation argument.
In fact, for $\nu$ denoting the outer unit normal vector to $\dG$, the divergence theorem together with the homogeneous Dirichlet boundary conditions of $\bv$ first yields that
\begin{equation}\label{eq:avg zero off-diag term}
    \int_G h_* \div \tbv \srd \bx = h_* \int_{\del G} \tbv \cdot \nu \srd S = 0,
\end{equation}
implying that $h_* \div \tbv \in \rL_0^q(G)$, and likewise $A_* \div \tbv \in \rL_0^q(G)$.
For every $\alpha > 0$, we deduce from interpolation and Young's inequality that there exists $C(\alpha) > 0$ such that
\begin{equation*}
    \| h_* \div \tbv \|_{\rL_0^q(G)} \le C \cdot \| \tbv \|_{\rW^{1,q}(G)} \le C(\alpha) \cdot \| \tbu \|_{\rX_0^\rm} + \alpha \cdot \| \tbu \|_{\rX_1^\rm}.
\end{equation*}
In total, it follows that $\ALFI_{\per,2}$ is a relatively $\ALFI_{\per,1}$-bounded perturbation with relative bound zero, so perturbation theory for maximal $\rL^p$-regularity, see, for instance, \cite[Cor.~2]{KW:01}, implies that there exists $\omega_0' \in \bR$ such that for all $\omega > \omega_0'$, the operator $\ALFI_{\per,0} + \omega = \ALFI_{\per,1} + \ALFI_{\per,2} + \omega$ has maximal $\rL^p$-regularity on $\rX_0^\rm$.
Thanks to $0 \in \rho(\ALFI_{\per,0})$, this property even follows without shift, completing the proof.
\end{proof}

Next, we discuss the estimates of the nonlinearities.

\begin{lem}\label{lem:ests of op matrix per}
Consider $p$, $q \in (1,\infty)$ satisfying \eqref{eq:cond p and q glob small data} as well as $\bu_* = (0,h_*,A_*)$ with $h_* > \kappa$ and $A_* \in (0,1]$ constant, and let $r_0 > 0$ as chosen in \eqref{eq:r_0}, yielding that $\bu(t) = \tbu(t) + \bu_* \in W$ for $t \in [0,T]$ for all $\tbu(t) \in V$ or $\tbu \in \oB_{\bE_1^\rm}(0,r_0)$.
Then the following assertions are valid.
\begin{enumerate}[(a)]
    \item The operators $\ALFI_\per \colon V \to \cL(\rX_1^\rm,\rX_0^\rm)$ are closed and linear.
    \item For every $r \in (0,r_0)$, there exists a constant $L(r) > 0$ such that
    \begin{equation*}
        \| (\ALFI_\per(\tbu_1) - \ALFI_\per(\tbu_2))\tbu \|_{\bE_0^\rm} \le L(r) \cdot \| \tbu_1 - \tbu_2 \|_{\bE_1^\rm} \cdot \| \tbu \|_{\bE_1^\rm}
    \end{equation*}
    for all $\tbu_1$, $\tbu_2 \in \oB_{\bE_1^\rm}(0,r)$ and $\tbu \in \bE_1^\rm$.
\end{enumerate}
\end{lem}

\begin{proof}
Analogously as in the proof of \autoref{lem:max reg & inv of lin op matrix}, we find that there is $\omega_0 \in \bR$ such that $\ALFI_\per(\tbu) + \omega$ has maximal $\rL^p$-regularity on $\rX_0^\rm$ for all $\omega > \omega_0$ and $\tbu \in V$.

The Lipschitz estimate is similar to the one \autoref{lem:nonlin ests} upon noting that $\tbu(t) + \bu_* \in W$ for all $t \in [0,T]$, i.e., we find the existence of $C_A(r) > 0$ with $\| (\ALFI_\per(\tbu_1) - \ALFI_\per(\tbu_2)) \tbu \|_{\rX_0^\rm} \le C_A(r) \cdot \| \tbu_1 - \tbu_2 \|_{\rX_\gamma^\rm} \cdot \| \tbu \|_{\rX_1^\rm}$.
Together with $\bE_1^\rm \hookrightarrow \rBUC([0,T];\rX_\gamma^\rm)$, it then follows that there exists a constant $L(r) > 0$ such that
\begin{equation*}
    \| (\ALFI_\per(\tbu_1) - \ALFI_\per(\tbu_2))\tbu \|_{\bE_0^\rm} \le C_A(r) \cdot \| \tbu_1 - \tbu_2 \|_{\rBUC([0,T];\rX_\gamma^\rm)} \cdot \| \tbu \|_{\bE_1^\rm} \le L(r) \cdot \| \tbu_1 - \tbu_2 \|_{\bE_1^\rm} \cdot \| \tbu \|_{\bE_1^\rm}. \qedhere
\end{equation*}
\end{proof}

It remains to show estimates for the right-hand side $F_\per$.
At this point, we invoke the assumptions on the data from \autoref{ass:atm & ocn vels}.
The nonlinear estimates of $F_\per$ are addressed in the next lemma.

\begin{lem}\label{lem:ests RHS per}
Let $p$, $q \in (1,\infty)$ be such that \eqref{eq:cond p and q glob small data} is valid, consider $\bu_* = (0,h_*,A_*)$, where $h_* > \kappa$ and $A_* \in (0,1]$, and let $r_0 > 0$ as above such that $\bu(t) = \tbu(t) + \bu_* \in W$ for all $t \in [0,T]$, and for all $\tbu \in \oB_{\bE_1^\rm}(0,r_0)$.
Besides, suppose that the surface wind and ocean velocities $\bv_\ra$ and $\bv_\ro$ satisfy \autoref{ass:atm & ocn vels}, and $\left.f_\per\right|_{(0,T)} \in \bE_0^\rm$.
Then the following is valid.
\begin{enumerate}[(a)]
    \item For all $\tbu \in \oB_{\bE_1^\rm}(0,r_0)$, it holds that $F_\per(\tbu) \in \bE_0^\rm$.
    \item For every $r \in (0,r_0)$ and $\delta_1 > 0$, there exists a constant $C_F(r,\delta_1) > 0$ such that
    \begin{equation*}
        \| F_\per(\tbu_1) - F_\per(\tbu_2) \|_{\bE_0^\rm} \le C_F(r,\delta_1) \cdot \| \tbu_1 - \tbu_2 \|_{\bE_1^\rm}.
    \end{equation*}
\end{enumerate}
\end{lem}

\begin{proof}
Let us start by estimating the bilinear part $F_{\per,1}$ of $F_\per$, i.e.,
\begin{equation*}
    F_{\per,1}(\tbu) = \begin{pmatrix}
        -(\tbv \cdot \nabla) \tbv\\
        -\div(\tbv \th)\\
        -\div(\tbv \tA)
    \end{pmatrix}.
\end{equation*}
First, H\"older's inequality and the embedding $\rB_{qp}^{2-\nicefrac{2}{p}}(G) \hookrightarrow \rW^{1,q}(G) \hookrightarrow \rL^\infty(G)$ yield that
\begin{equation}\label{eq:est convective term}
    \| -(\tbv \cdot \nabla) \tbv \|_{\rL^q(G)} \le \| \tbv \|_{\rL^\infty(G)} \cdot \| \nabla \tbv \|_{\rL^q(G)} \le C \cdot \| \tbv \|_{\rW^{1,q}(G)}^2 \le C \cdot \| \tbu \|_{\rX_\gamma^\rm}^2.
\end{equation}
Similarly as in \eqref{eq:avg zero off-diag term}, we find that $\div(\tbv \th)$, $\div(\tbv \tA) \in \rL_0^q(G)$.
Analogous arguments as in \eqref{eq:est convective term} lead to
\begin{equation*}
    \| \div(\tbv \th) \|_{\rL_0^q(G)} \le C \cdot \| \bu \|_{\rX_\gamma^\rm}^2 \tand \| \div(\tbv \tA) \|_{\rL_0^q(G)} \le C \cdot \| \tbu \|_{\rX_\gamma^\rm}^2.
\end{equation*}
In particular, thanks to the embedding $\bE_1^\rm \hookrightarrow \rBUC([0,T];\rX_\gamma^\rm)$, we find that $\| F_{\per,1}(\tbu) \|_{\bE_0^\rm} \le C \cdot \| \tbu \|_{\bE_1^\rm}^2$, so $F_{\per,1}(\tbu) \in \bE_0^\rm$ for all $\tbu \in \oB_{\bE_1^\rm}(0,r_0)$.
Likewise, for $\tbu_1$, $\tbu_2 \in \oB_{\bE_1^\rm}(0,r)$, we infer that
\begin{equation*}
    \| F_{\per,1}(\tbu_1) - F_{\per,1}(\tbu_2) \|_{\bE_0^\rm} \le C \cdot \bigl(\| \tbu_1 \|_{\bE_1^\rm} + \| \tbu_2 \|_{\bE_1^\rm}\bigr) \cdot \| \tbu_1 - \tbu_2 \|_{\bE_1^\rm} \le C \cdot r \cdot \| \tbu_1 - \tbu_2 \|_{\bE_1^\rm}.
\end{equation*}

We proceed with the estimates of the remaining part $F_{\per,2}$ given by
\begin{equation*}
    F_{\per,2}(\tbu) = \begin{pmatrix}
        - C_{\cor} \bk \times (\tbv - \bv_\ro) + \frac{1}{\rho (\th + h_*)}\bigl(f_\ra + f_\ro(\tbv) + f_\rb(\tbv,\th + h_*,\tA + A_*)\bigr)\\
        0\\
        0
    \end{pmatrix} + f_\per(t).
\end{equation*}
Note that $f_\per$ is independent of $\tbu$ and already satisfies $\left.f_\per\right|_{(0,T)} \in \bE_0^\rm$ by assumption.
Besides, the choice of $r_0$ implies that $\th(t) + h_* > \kappa$ for all $t \in [0,T]$, so we find that $\frac{1}{\rho(\th + h_*)} \in \rL^\infty(0,T;\rL^\infty(G))$.
On the other hand, from H\"older's inequality, Young's inequality, and \autoref{ass:atm & ocn vels}, it follows that
\begin{equation}\label{eq:ests of f_a & f_o}
    \begin{aligned}
        \| f_\ra \|_{\rL^p(0,T;\rL^q(G))}
        &\le C \cdot \| \bv_\ra \|_{\rL^{2p}(0,T;\rL^{2q}(G))}^2 < C \delta_1^2 \tand\\
        \| f_\ro(\tbv) \|_{\rL^p(0,T;\rL^q(G))} 
        &\le C\bigl(\delta_1^2 + \| \tbv \|_{\rL^{2p}(0,T;\rL^{2q}(G))}^2\bigr).
    \end{aligned}
\end{equation}
Invoking that $\bu(t) = \tbu(t) + \bu_* \in W$ for all $t \in [0,T]$ by the choice of $r_0$, as in \eqref{eq:est of basal stress}, we find that
\begin{equation}\label{eq:est of f_b}
    \| f_\rb(\tbv,\th + h_*,\tA + A_*) \|_{\rL^p(0,T;\rL^q(G))} \le C \cdot \bigl(\kappa + \| \th \|_{\rL^p(0,T;\rL^q(G))}\bigr) \le C \cdot\bigl(1 + \| \tbu \|_{\bE_1^\rm}\bigr).
\end{equation}
In summary, we have established that $F_{\per,2}(\tbu) \in \bE_0^\rm$ for all $\tbu \in \oB_{\bE_1^\rm}(0,r_0)$.

For the Lipschitz estimate of $F_{\per,2}$, we consider $\tbu_1$, $\tbu_2 \in \oB_{\bE_1^\rm}(0,r)$, and we invoke that
\begin{equation*}
    \bE_1^\rm \hookrightarrow \rBUC([0,T];\rX_\gamma^\rm) \hookrightarrow \rL^\infty(0,T;\rL^\infty(G)^4).
\end{equation*}
The mean value theorem and the above observation that $\th_i(t) + h_* > \kappa$ for all $t \in [0,T]$ imply that
\begin{equation}\label{eq:est of diff of fractions}
    \left\|\frac{1}{\rho(\th_1 + h_*)} - \frac{1}{\rho(\th_2 + h_*)}\right\|_{\rL^\infty(0,T;\rL^\infty(G))} \le C \cdot \| \th_1 - \th_2 \|_{\rL^\infty(0,T;\rL^\infty(G))} \le C \cdot \| \tbu_1 - \tbu_2 \|_{\bE_1^\rm}.
\end{equation}
Joint with the estimate \eqref{eq:ests of f_a & f_o}$_1$ of $f_\ra$, this yields that
\begin{equation*}
    \left\|\left(\frac{1}{\rho(\th_1 + h_*)} - \frac{1}{\rho(\th_2 + h_*)}\right)f_\ra\right\|_{\rL^p(0,T;\rL^q(G))} \le C \delta_1^2 \cdot \| \tbu_1 - \tbu_2 \|_{\bE_1^\rm}.
\end{equation*}
Regarding the ocean force, we first split the term as follows:
\begin{equation*}
    \begin{aligned}
        &\quad \left\|\frac{1}{\rho(\th_1 + h_*)} f_\ro(\tbv_1) - \frac{1}{\rho(\th_2 + h_*)}f_\ro(\tbv_2)\right\|_{\rL^p(0,T;\rL^q(G))}\\
        &\le \left\|\left(\frac{1}{\rho(\th_1 + h_*)} - \frac{1}{\rho(\th_2 + h_*)}\right)f_\ro(\tbv_1)\right\|_{\rL^p(0,T;\rL^q(G))} + \left\|\frac{1}{\rho(\th_2 + h_*)}(f_\ro(\tbv_1) - f_\ro(\tbv_2))\right\|_{\rL^p(0,T;\rL^q(G))}.
    \end{aligned}
\end{equation*}
Combining \eqref{eq:est of diff of fractions} with the estimate \eqref{eq:ests of f_a & f_o}$_2$ of $f_\ro$, we estimate the first addend by
\begin{equation*}
    \left\|\left(\frac{1}{\rho(\th_1 + h_*)} - \frac{1}{\rho(\th_2 + h_*)}\right)f_\ro(\tbv_1)\right\|_{\rL^p(0,T;\rL^q(G))} \le C(\delta_1^2 + r^2) \cdot \| \tbu_1 - \tbu_2 \|_{\bE_1^\rm}.
\end{equation*}
For the second addend, H\"older's inequality and \autoref{ass:atm & ocn vels} yield that
\begin{equation*}
    \begin{aligned}
        \| f_\ro(\tbv_1) - f_\ro(\tbv_2) \|_{\rL^p(0,T;\rL^q(G))}
        &\le C \cdot \bigl(\| \bv_\ro \|_{\rL^{2p}(0,T;\rL^{2q}(G))} + \| \tbu_1 \|_{\bE_1^\rm} + \| \tbu_2 \|_{\bE_1^\rm}\bigr) \cdot \| \tbu_1 - \tbu_2 \|_{\bE_1^\rm}\\
        &\le C(\delta_1 + r) \cdot \| \tbu_1 - \tbu_2 \|_{\bE_1^\rm}.
    \end{aligned}
\end{equation*}
Concatenating the previous estimates, we find that
\begin{equation*}
    \left\|\frac{1}{\rho(\th_1 + h_*)} f_\ro(\tbv_1) - \frac{1}{\rho(\th_2 + h_*)}f_\ro(\tbv_2)\right\|_{\rL^p(0,T;\rL^q(G))} \le C(\delta_1^2 + r^2 + \delta_1 + r) \cdot \| \tbu_1 - \tbu_2 \|_{\bE_1^\rm}.
\end{equation*}
With regard to the basal stress $f_\rb$, we proceed in a similar way, namely, we first split this term into
\begin{equation*}
    \begin{aligned}
        &\quad \left\|\frac{1}{\rho(\th_1 + h_*)} f_\rb(\tbv_1,\th_1 + h_*,\tA_1 + A_*) - \frac{1}{\rho(\th_2 + h_*)}f_\rb(\tbv_2,\th_2 + h_*,\tA_2 + A_*)\right\|_{\rL^p(0,T;\rL^q(G))}\\
        &\le \left\|\left(\frac{1}{\rho(\th_1 + h_*)} - \frac{1}{\rho(\th_2 + h_*)}\right)f_\rb(\tbv_1,\th_1 + h_*,\tA_1 + A_*)\right\|_{\rL^p(0,T;\rL^q(G))}\\
        &\quad + \left\|\frac{1}{\rho(\th_2 + h_*)}(f_\rb(\tbv_1,\th_1 + h_*,\tA_1 + A_*) - f_\rb(\tbv_2,\th_2 + h_*,\tA_2 + A_*))\right\|_{\rL^p(0,T;\rL^q(G))}.
    \end{aligned}
\end{equation*}
For the first addend, we combine \eqref{eq:est of diff of fractions} with the estimate \eqref{eq:est of f_b} of $f_\rb$ to infer that
\begin{equation*}
    \begin{aligned}
        \left\|\left(\frac{1}{\rho(\th_1 + h_*)} - \frac{1}{\rho(\th_2 + h_*)}\right)f_\rb(\tbv_1,\th_1 + h_*,\tA_1 + A_*)\right\|_{\rL^p(0,T;\rL^q(G))} 
        \le C (1 + r) \cdot \| \tbu_1 - \tbu_2 \|_{\bE_1^\rm}.
    \end{aligned}
\end{equation*}
Concerning the second addend, we proceed similarly as in \eqref{eq:diff of basal stress} to conclude that
\begin{equation*}
    \begin{aligned}
        \| f_\rb(\tbv_1,\th_1 + h_*,\tA_1 + A_*) - f_\rb(\tbv_2,\th_2 + h_*,\tA_2 + A_*) \|_{\rL^p(0,T;\rL^q(G))} \le C(r + 1) \cdot \| \tbu_1 - \tbu_2 \|_{\bE_1^\rm}.
    \end{aligned}
\end{equation*}
Putting together these estimates, we argue that
\begin{equation*}
    \begin{aligned}
        &\quad \left\|\frac{1}{\rho(\th_1 + h_*)} f_\rb(\tbv_1,\th_1 + h_*,\tA_1 + A_*) - \frac{1}{\rho(\th_2 + h_*)}f_\rb(\tbv_2,\th_2 + h_*,\tA_2 + A_*)\right\|_{\rL^p(0,T;\rL^q(G))}\\
        &\le C(r + 1)\| \tbu_1 - \tbu_2 \|_{\bE_1^\rm}.
    \end{aligned}
\end{equation*}
In total, we have proved the estimates as asserted in the lemma.
\end{proof}

We are now in position to prove the main result in this section on the existence of a time-periodic strong solution to the landfast ice model subject to time-periodic forcing terms.

\begin{proof}[Proof of \autoref{thm:time-per sols}]
The proof is based on an application of the abstract result \autoref{lem:quasilin time per result}.
\autoref{lem:max reg & inv of lin op matrix} implies that the linearized operator $\ALFI_{\per,0}$ has maximal $\rL^p$-regularity on $\rX_0^\rm$ and satisfies $0 \in \rho(\ALFI_{\per,0})$, so it lies in the scope of the Arendt-Bu theorem as recalled in \autoref{lem:Arendt-Bu}, see also \autoref{ass:time-per quasilin}(c).
In \autoref{lem:ests of op matrix per}, we showed that $\ALFI_\per$ satisfies \autoref{ass:time-per quasilin}(a), and in \autoref{lem:ests RHS per}, we proved that the right-hand side $F_\per$ fulfills the estimates required by \autoref{ass:time-per quasilin}(b).
In particular, the proof of \autoref{lem:ests RHS per} reveals that for $r > 0$ and $\delta_1 > 0$ sufficiently small, the Lipschitz constant $C_F(r,\delta_1)$ satisfies $C_F < \delta_1'$, where $\delta_1' > 0$ is sufficiently small.

At the same time, we observe that
\begin{equation*}
    F_{\per,2}(0) = \begin{pmatrix}
        C_{\cor} \bk \times \bv_\ro + \frac{1}{\rho h_*}\bigl(f_\ra + f_\ro(0) + f_\rb(0,h_*,A_*)\bigr)\\
        0\\
        0
    \end{pmatrix} + f_\per(t).
\end{equation*}
Thus, choosing $\delta_1 > 0$ smaller, and considering $\delta_2 > 0$ small enough, we find that $\| F_\per(0) \|_{\bE_0^\rm} < \delta_2'$ for $\delta_2' > 0$ small by virtue of \autoref{ass:atm & ocn vels} as well as the assumption on $f_\per$.
The existence of a time-periodic solution on the time period $(0,T)$ then is a consequence of \autoref{ass:time-per quasilin}.
The solution can be extended to the whole real line by employing the periodicity condition.
\end{proof}

\section{Numerical simulations}\label{sec:numerical simulations}

In this section, we examine the influence of tensile strength and basal drag parameterizations on the formation of landfast ice and compare this extended model to the classical viscous–plastic formulation used to describe pack ice.
Moreover, we compare the numerical simulations with the analytical results obtained in the previous sections.
\autoref{sec:Ex1} is concerned with a comparison of the landfast ice model with Hibler's viscous-plastic sea-ice model.
In \autoref{ssec:dyn evol towards equil}, we provide numerical simulations associated with the convergence to equilibrium solutions, while the focus of \autoref{ssec:dyn evol to stat landfast ice config} is the case of constant (and thus time-periodic) wind forcing and its relation to the analytical result.

We consider a square domain $
G = [0,512\,\mathrm{km}] \times [0,512\,\mathrm{km}]$ with zero Dirichlet boundary conditions on~$\partial G $. 
Note that the geometric setup of the domain slightly differs from the one considered in \autoref{sec:well-posedness}, \autoref{sec:glob wp close to equil}, and \autoref{sec:time-per sols}, where a domain with $\rC^2$-boundary is considered.
The case of a domain of the form as $G$, though with periodic boundary conditions, for sea-ice has been considered in \cite[Sec.~5]{BBH:26}.
The domain is initially fully covered by sea-ice with  varying thickness given by 
\begin{equation*}
    A(x,y,0) = 1, \quad h(x,y,0) = 2.5\,\mathrm{m} - \sin\!\left(\frac{x}{512 \,\mathrm{km}}\pi\right).
\end{equation*}
Initially, the sea-ice is at rest, $\bv(x,y,0) = \bm{0} \,\mathrm{m/s}$.
Following \cite{LDBRSF:16}, we select the model parameters
\begin{equation*}
    k_{\rt} = 0.15, 
    \quad
    h_{\crit} = 2.5\,\mathrm{m}, 
    \quad
    k_2 = 5\,\mathrm{N/m^3}, 
    \quad
    \alpha_{\rb} = 20.
\end{equation*}
All other parameters are listed in \autoref{tab:1}. 
The system is driven by a constant wind field of magnitude
\begin{equation*}
    |\bv_a| = 20\,\mathrm{m/s},
\end{equation*}
blowing from left to right across $G$. 
Ocean currents and Coriolis forces are neglected for simplicity.
\begin{figure}
    \centering   \includegraphics[width=0.95\linewidth]{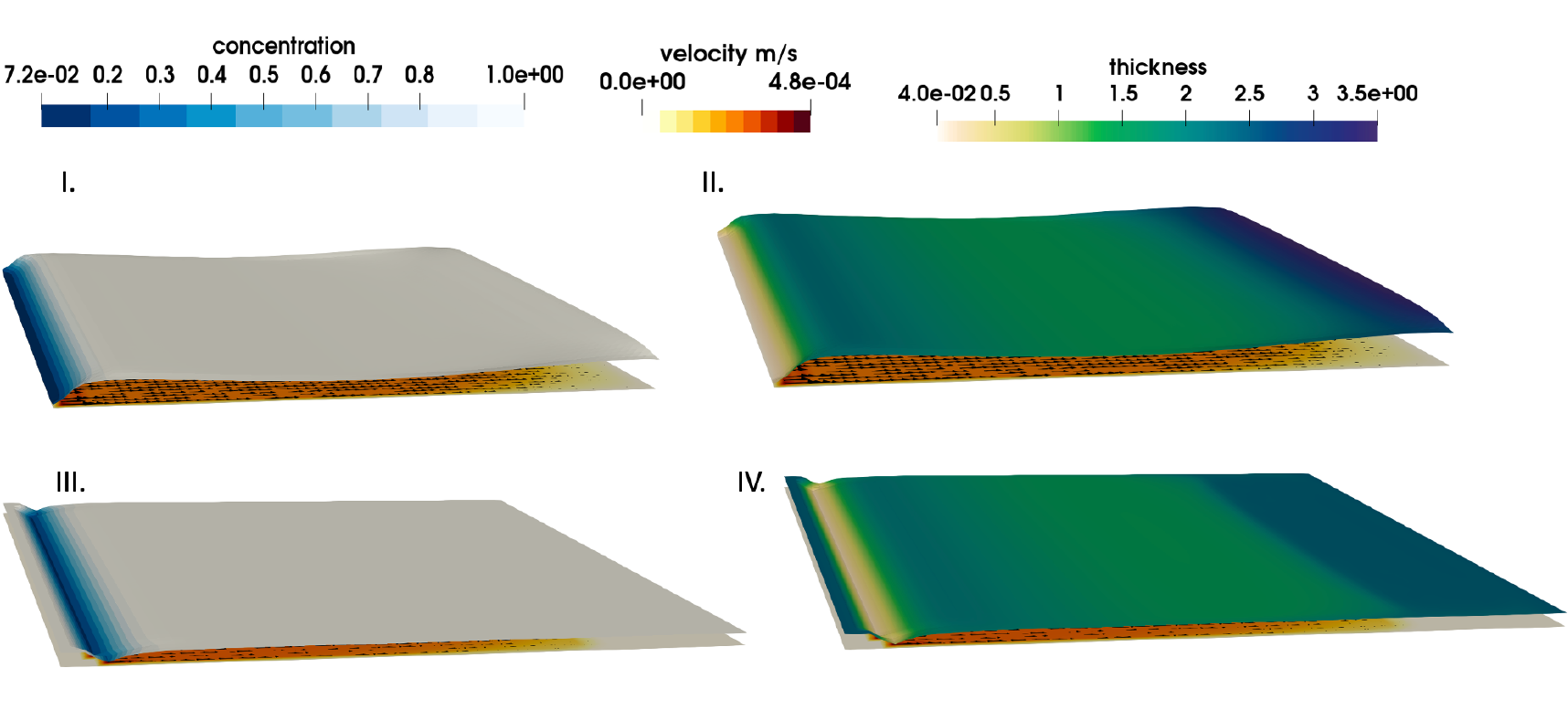}
    \caption{Comparison of sea-ice concentration, thickness and velocity after 2 days of simulation using the standard viscous–plastic model (I, II) and its landfast ice extension (III, IV), both under dominant rightward wind forcing. 
    While in the standard case (I, II), the ice drifts with the wind, the landfast ice extension (III, IV) leads to the formation of stationary sea-ice along the coast and an opening of sea-ice (polynya) to its left.    }
    \label{fig:ex1}
\end{figure}

The domain is discretized using an equidistant quadrilateral mesh with grid size $8\,\mathrm{km}$, and the system is advanced in time with a time step of $\Delta t = 30\,\mathrm{min}$. The sea-ice dynamics are approximated using a time-splitting approach: first, the sea-ice momentum equation~\eqref{eq:mom} is solved, followed by an update of the tracers~\eqref{eq:transport}. The momentum equation is discretized in space using piecewise linear finite elements and in time with an implicit Euler scheme (see \cite{MK:26}). The resulting nonlinear system is solved using a modified Newton–multigrid method \cite{MR:17,MR:17b}. 

The sea-ice concentration and thickness are advected with an upwind scheme, where we set $d_\rh = d_\rA = 0$.
Further implementation details can be found in \cite{MK:26}.

\subsection{A landfast-ice extension of the viscous–plastic sea-ice model}\label{sec:Ex1}
\

In the standard viscous–plastic sea-ice model (without the basal drag term~\eqref{eq:basal stress} and with tensile strength $T=0$ in~\eqref{eq:stress tensor}), the ice is transported in the direction of the wind, leading to a reduction of ice concentration along the left boundary of the domain (see Panel~I in \autoref{fig:ex1}).

In contrast, in the landfast ice configuration, i.e., when the viscous–plastic rheology is augmented by tensile strength~\eqref{eq:stress tensor} and the momentum equation by a basal drag term~\eqref{eq:basal stress}, the ice remains stationary near the left boundary (coast). This behavior results from the grounding condition ($h > h_{\crit}$) and the increased resistance of the ice to tensile stresses.

Consequently, an opening in the ice cover forms in the left part of the domain, where the sea-ice concentration drops to nearly zero (see Panel~III in \autoref{fig:ex1}). Such openings, known as polynyas, typically arise under offshore wind forcing, as considered here. To the right of the polynya, the sea-ice concentration increases again and accumulates due to the imposed Dirichlet boundary conditions.

In the standard viscous–plastic model, the ice thickness is transported with the wind in the same way as the ice concentration (see Panel~II in \autoref{fig:ex1}). As a result, the thickness decreases toward the left boundary and accumulates along the right boundary of the domain. The slight decrease in thickness visible in the center of the domain can be attributed to the sinusoidal variation in the initial condition.

In the landfast ice configuration (Panel~IV in \autoref{fig:ex1}), the ice thickness remains anchored near the left boundary, decreases within the opening (polynya), exhibits variations toward the center due to the imposed initial condition, and accumulates again at the right boundary. Compared to the standard viscous–plastic model, the accumulated thickness at the right boundary is smaller in the landfast ice case (compare Panels~II and~IV in \autoref{fig:ex1}). This behavior can be attributed to the increased stiffness of the ice in the landfast ice configuration.

This behavior of the sea-ice concentration and thickness is also reflected in the sea-ice velocity field: the velocity vanishes in a strip along the left boundary, increases toward the center of the domain, and decreases again toward the right boundary (see Panel~II and Panel~IV in \autoref{fig:ex1}). Owing to the enhanced resistance to tensile stresses and an additional basal drag in the landfast ice formulation, the ice becomes effectively stiffer and accumulates over a broader region along the right boundary compared to the standard viscous–plastic model. This is evident from the wider near-zero velocity region in Panel~IV relative to Panel~II. 

\subsection{Dynamic evolution toward an equilibrium state without forcing}\label{ssec:dyn evol towards equil}
\
 
 \begin{figure}
     \centering
\includegraphics[width=0.9\linewidth]{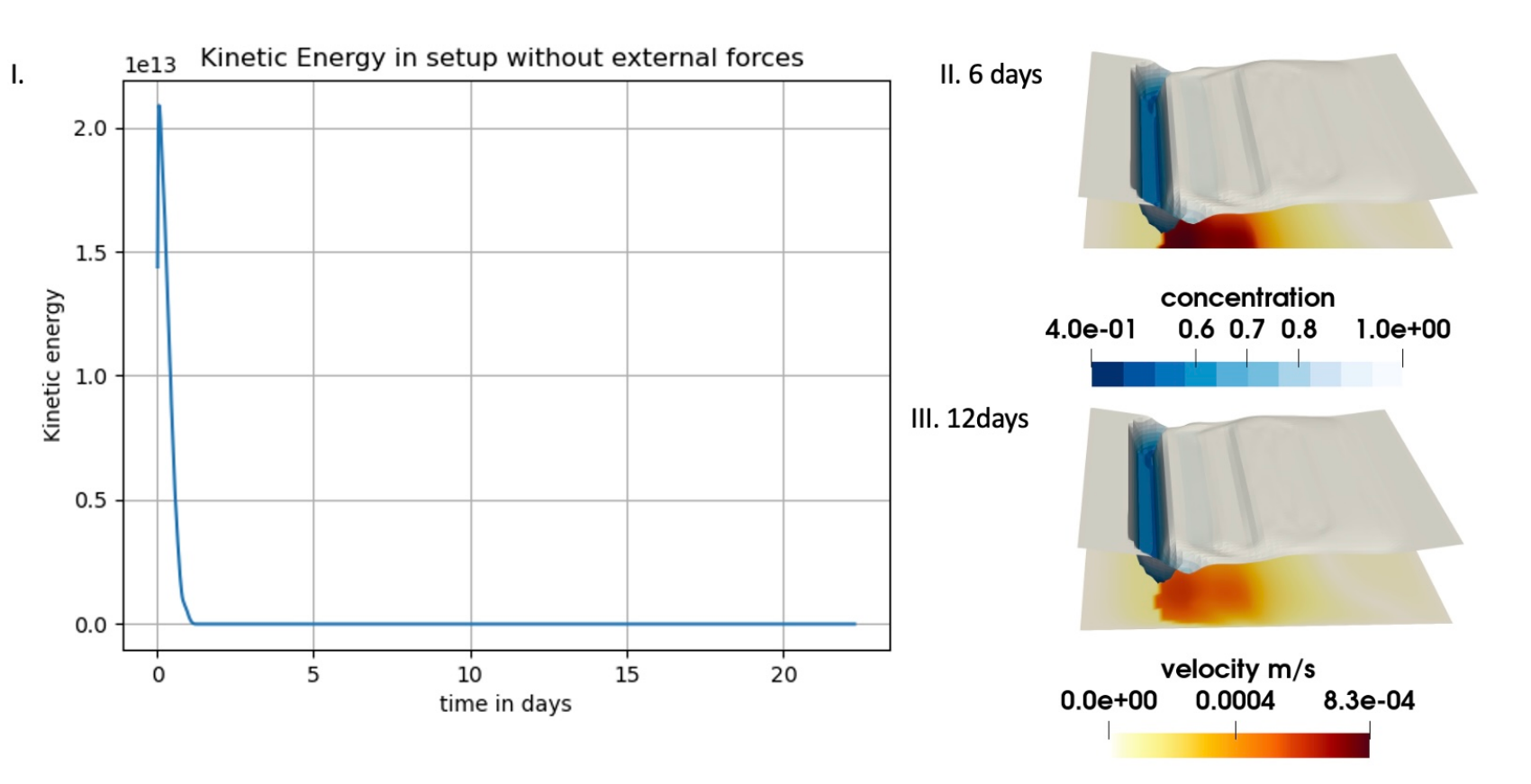}
     \caption{Sea-ice evolution without external forcing. The kinetic energy is presented in Panel I. 
     Panel II and Panel III show the sea-ice concentration after 6 and 12 simulated days respectively. 
     In the absence of external forces the sea-ice velocity approaches zero over time, which results in a stable sea-ice cover over time. }
     \label{fig:ex3}
 \end{figure}
In this section, we study the evolution of landfast ice with initial velocity
\[
\bv(x,y,0)= \sin\!\left(\frac{\pi x}{512\, \mathrm{km}}\right)\sin\!\left(\frac{\pi y}{512\, \mathrm{km}}\right)\cdot 0.05\,\text{m/s}.
\]
Wind and  ocean drag, Coriolis term, and forcing due to sea surface height variations are neglected, i.e., 
\[
f_{\ra} = f_{\ro} = f_{\rc} = f_{\rsh} = 0.
\]
The initial conditions for concentration and thickness are chosen as in the previous example:
\[
A(x,y,0)=1, \quad h(x,y,0)= 2.5\, \mathrm{m} - 0.5 \sin\!\left(\frac{\pi x}{512\, \mathrm{km}}\right).
\]

All remaining parameters are set as in \autoref{sec:Ex1}. To evaluate the results over a time span of $I=[0,22]$ days, we consider the kinetic energy 
\[
    E_\text{kin}(t)=\int_G \frac{1}{2}\rho h(x,y,t) |\bv(x,y,t)|^2 \srd x \srd y,
\]
and analyze the fast ice concentration at selected time instances.

Since no external forcing is applied, the initial ice velocity (with maximal magnitude of $0.05\,\mathrm{m/s}$) decreases over time and approaches zero. 
This behavior is reflected in the evolution of the kinetic energy, see Panel~I in \autoref{fig:ex3}. In particular, a rapid decay is observed at first two simulated days; after that the sea-ice velocity reaches values of the order of $10^{-4}$  and continues to decrease thereafter, see Panel III and Panel IV in \autoref{fig:Ex2}.

As in the previous examples, a characteristic opening in the sea-ice cover forms in the left part of the domain, see Panel II and Panel III in \autoref{fig:ex3}. 
This polynya opening  is  similar to the patterns observed under offshore wind conditions. As indicated by the decay of kinetic energy, the ice concentration distribution changes only marginally after time 2 simulated days, since the ice velocity becomes very small throughout the domain.  The numerical approximation of this stable equilibrium, characterized by a sea-ice velocity close to zero, reflects the existence result of a global solution established in \autoref{thm:glob ex close to equil}.
As in the present numerical setup, no external forcing is assumed in \autoref{thm:glob ex close to equil}.
Note that the initial sea-ice velocity is small, the initial ice concentration is constant, and the initial mean ice thickness is oscillating around a constant, so the situation seems comparable to the assumptions in \autoref{thm:glob ex close to equil}.
The observation that the system approaches a stationary state aligns with the assertion of \autoref{thm:glob ex close to equil} that the solution converges to an equilibrium solution.

\subsection{Dynamic evolution toward a stationary configuration under constant wind forcing}\label{ssec:dyn evol to stat landfast ice config}
\

\begin{figure}
    \centering
    \includegraphics[width=0.9\linewidth]{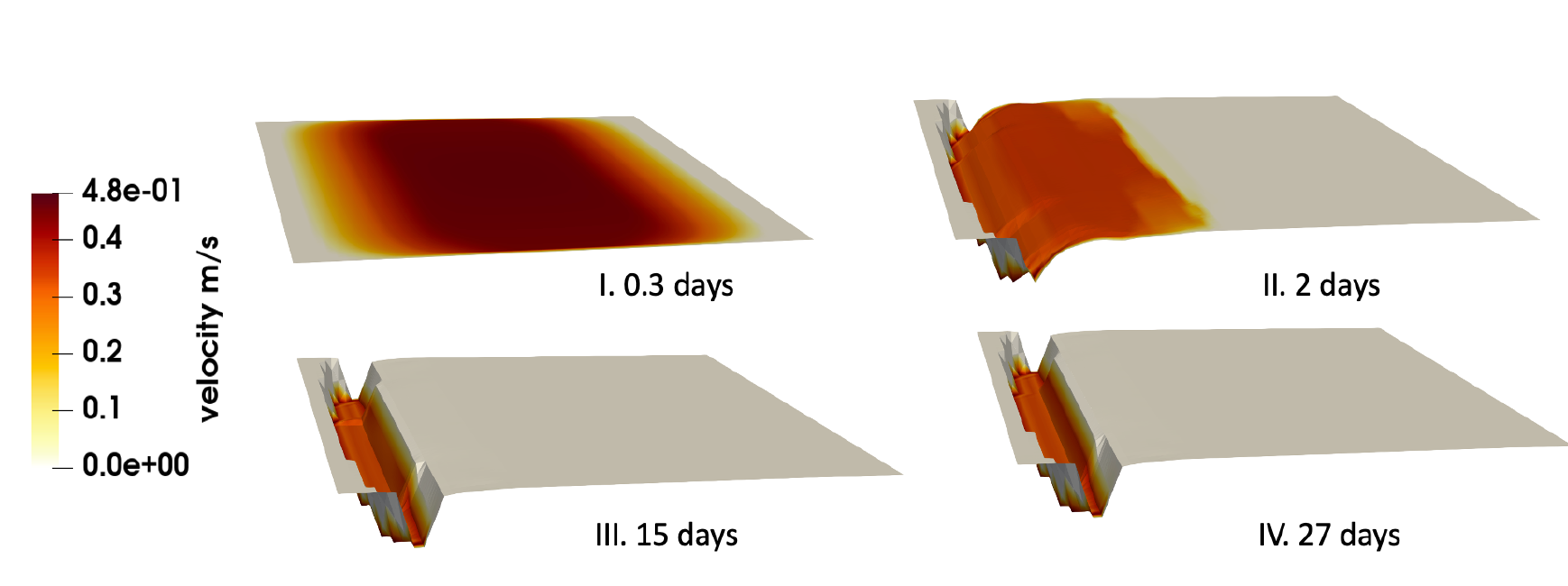}
    \caption{Sea-ice evolution under constant rightward wind forcing over 27 simulated days.
    Panels I–IV show the velocity field scaled by the sea-ice concentration. 
    In Panel~I, the scaling factor is 1, as the entire domain is initially ice-covered. 
    Over time, the system approaches a stationary state, indicated by near-zero velocities in ice-covered regions (gray plateau in Panels III and IV), while non-zero velocities occur only in ice-free areas (opening in Panels III and IV).}
    \label{fig:Ex2}
\end{figure}

We consider the temporal evolution of sea-ice dynamics over 27 simulated days. For illustration purposes, we modify the initial ice thickness to
\[
h(x,y,0) = 2.5 \, \mathrm{m} - 0.5 \sin\!\left(\frac{x}{512 \,\mathrm{km}}\pi\right).
\]
All other settings are identical to those in \autoref{sec:Ex1}. In this example, we analyze the ice dynamics based on the sea-ice velocity field. To relate the velocity to the ice distribution, \autoref{fig:Ex2} shows the sea-ice velocity, $\bv$, scaled by the sea-ice concentration, $A$.

At the initial time, the entire domain is covered by a continuous ice layer, which is reflected by a scaling factor of one throughout the domain (see Panel~I in \autoref{fig:Ex2}). 
Due to the dominant rightward wind forcing, an open-water region (polynya) forms (Panel~II in \autoref{fig:Ex2}), as already observed in the previous example in \autoref{sec:Ex1}. 

As time progresses, the system evolves toward a stationary configuration characterized by vanishing ice velocity to the right of the polynya (Panels~III and~IV).
This state indicates that the basal stress $f_{\rb}$, in combination with internal rheological stresses, has reached an equilibrium, where it effectively counteracts the external wind forcing.

These results provide a numerical illustration for the mathematical framework established in \autoref{sec:time-per sols}.
Specifically, since a constant force is a special case of a $T$-periodic force for any $T > 0$, the observed stationary equilibrium provides a qualitative numerical counterpart to the stationary, and hence time-periodic, states covered by the analytical framework.
The present computational regime is not intended as a direct numerical verification of the small-data theorem.
This provides a qualitative connection between our result on the existence of a time-periodic solution and the physical formation of stable landfast ice.

\section{Conclusion and outlook}

To the best of our knowledge, the present work provides the first rigorous PDE analysis of a regularized evolution system based on the landfast-ice model introduced by Lemieux et al.~\cite{LDBRSF:16}.

In a similar way as in \cite{BDHH:22} on the analysis of a regularized version of Hibler's viscous-plastic sea-ice model~\cite{Hib:79}, we used an approach by maximal $\rL^p$-regularity and quasilinear evolution equations to study the local strong well-posedness and global strong well-posedness close to equilibria.
However, compared to \cite{BDHH:22}, we accounted for the tensile strength in the stress tensor, and we treated the basal stress as a state-dependent lower-order nonlinearity.
In particular, the modified stress led to the vanishing of $P' = P - T$ and hence of the pressure-gradient coupling when $k_\rt \equiv 1$, which required a more careful spectral analysis due to the lack of cancellations in the case $k_\rt \equiv 1$.
We overcame this by imposing a smallness condition on~$\shadedDelta_{\min}$.
On the other hand, in contrast to \cite{BDHH:22}, we were able to establish the global well-posedness without a smallness condition on $\shadedDelta_{\min}$ in the non-degenerate case $k_\rt < 1$.

We also showed the local strong well-posedness of a parabolic-hyperbolic variant of the landfast model by introducing Lagrangian coordinates and showing maximal $\rL^p$-regularity in an anisotropic ground space.

The last analytical result of this paper is the existence and local uniqueness of a strong time-periodic solution to the landfast model in the presence of small time-periodic forcing terms.
The lack of invertibility required us to invoke the spaces of $\rL^q$-functions with average zero.
We were then able to show the maximal $\rL^p$-regularity of the linearized operator matrix close to equilibria, and together with the invertibility, we could infer the maximal periodic regularity, paving the way for the existence of a unique time-periodic strong solution by the contraction mapping principle.

Finally, we performed numerical simulations to provide qualitative numerical illustrations of the analytical picture.

Major remaining open problems are the analysis of the long-time dynamics and the construction of time-periodic solutions for the parabolic–hyperbolic system \(d_h=d_A=0\).
To the best of our knowledge, both questions also remain open for the original Hibler model.
A further important problem is the singular limit \(\shadedDelta_{\min}\to0\) in the stress regularization.

\medskip

{\bf Acknowledgements.}
This project has received funding from the European Regional Development Fund (grant FEM Power~II, ZS2024/06/18815) under the European Union's Horizon Europe Research and Innovation Program, which is gratefully acknowledged.
Felix Brandt would like to thank the German National Academy of Sciences Leopoldina for support through the Leopoldina Fellowship Program with grant number~LPDS 2024-07. 
Carolin Mehlmann acknowledges funding by the Deutsche Forschungsgemeinschaft (DFG, German Research Foundation) (SPP~1158: project number~463061012).

The authors would like to thank the anonymous referee for their careful reading and insightful comments that helped us improve the manuscript.

\appendix

\section{$\cR$-sectoriality, maximal $\rL^p$-regularity and boundedness of the $\Hinfty$-calculus}\label{sec:op theoret concepts}

The purpose of this section is to recall some operator theoretic concepts that are used throughout this manuscript.
These concepts comprise $\cR$-sectoriality as well as maximal $\rL^p$-regularity along with their relation as well as the boundedness of the $\Hinfty$-calculus.  

We start with the concept of {\em UMD} or {\em unconditional martingale differences} spaces.
Note that this is a probabilistic concept, but it admits a characterization in terms of the Hilbert transform.
In fact, a Banach space $\rX_0$ is a UMD space if and only if the Hilbert transform is bounded on $\rL^p(\bR;\rX_0)$, where $p \in (1,\infty)$.
We remark that many classical function spaces including $\rL^p$-spaces, Sobolev spaces $\rW^{s,p}$ or Besov spaces $\rB_{pq}^s$ enjoy this property if  $p$, $q \in (1,\infty)$, and this property transfers to closed subspaces of these spaces, see also \cite[Sec.~III.4.4 and III.4.5]{Ama:95} for more information.

In the theory of maximal $\rL^p$-regularity on UMD spaces, the concept of $\cR$-sectoriality is of prime importance.
This requires in turn the notion of $\cR$-boundedness, which we introduce in the following.
Let $\rX_0$ as well as $\rY_0$ be Banach spaces.
A family of operators $\cT(\rX_0,\rY_0)$ from $\rX_0$ to $\rY_0$ is referred to as {\em $\cR$-bounded} if there are $C > 0$ and $p \in [1,\infty)$ so that for all $N \in \bN$, $T_j \in \cT$, $x_j \in \rX_0$, and for each independent, symmetric, $\{-1,1\}$-valued random variable $\eps_j$ on a probability space $(\Omega,\mathcal{A},\mu)$, we have the inequality
\begin{equation}\label{eq:ineq R-bddness}
    \left\| \sum_{j=1}^N \eps_j T_j x_j \right\|_{\rL^p(\Omega;\rY_0)} \le C \cdot \left\| \sum_{j=1}^N \eps_j x_j \right\|_{\rL^p(\Omega;\rX_0)}.
\end{equation}
We call the smallest $C > 0$ so that \eqref{eq:ineq R-bddness} is satisfied the {\em $\cR$-bound} of $\cT$, and we use $\cR(\cT)$ to denote it.

Before making precise the concept of $\cR$-sectoriality, we recall the definition of a sectorial operator.
Indeed, a closed linear operator $A \colon \rX_1 \subset \rX_0 \to \rX_0$ is sectorial, for which we also use the notation $A \in \cS(\rX_0)$, if
\begin{enumerate}[(a)]
    \item the operator is densely defined and has dense range, so $\overline{\rD(A)} = \overline{\mathrm{R}(A)} = \rX_0$,
    \item it holds that $(-\infty,0) \subset \rho(A)$, 
    \item and there exists $M \in (0,\infty)$ such that $\| t(t+A)^{-1} \|_{\cL(\rX_0)} \le M$ for all $t > 0$.
\end{enumerate}
For $\Sigma_\theta \coloneqq \{z \in \bC \setminus \{0\} : |\arg z| < \theta\}$ representing the sector of angle $\theta \in (0,\pi]$ in the complex plane, the spectral angle $\phi_A$ of a sectorial operator $A$ is defined by
\begin{equation*}
    \phi_A \coloneqq \inf\left\{\phi : \rho(-A) \supset \Sigma_{\pi-\phi}, \enspace \sup_{\lambda \in \Sigma_{\pi-\phi}} \| \lambda (\lambda + A)^{-1} \| < \infty\right\}.
\end{equation*}
We remark that sectoriality of a closed operator is closely related to the generation of analytic semigroups, namely, if $A \in \cS(\rX_0)$ with $\phi_A < \nicefrac{\pi}{2}$, then $-A$ generates a bounded analytic semigroup of angle $\nicefrac{\pi}{2} - \phi_A$.

We are now in a position to introduce the concept of $\cR$-sectoriality of an operator.

\begin{defn}\label{def:R-sect}
A sectorial operator $A$ on a Banach space $\rX_0$ is {\em $\cR$-sectorial} provided
\begin{equation*}
    \cR_A(0) \coloneqq \cR\left\{t(t+A)^{-1} : t > 0\right\} < \infty,
\end{equation*}
and the $\cR$-angle of an $\cR$-sectorial operator $A$ is defined by
\begin{equation*}
    \phi_A^\cR \coloneqq \inf\{\theta \in (0,\pi) : \cR_A(\pi-\theta) < \infty\}, \twhere \cR_A(\theta) \coloneqq \cR\left\{\lambda(\lambda+A)^{-1} : |\arg \lambda| \le \theta\right\}.
\end{equation*}
The class of $\cR$-sectorial operators on a Banach space $\rX_0$ will be denoted by $\cR \cS(\rX_0)$ in the following.
\end{defn}

We now recall the concept of maximal $\rL^p$-regularity.
To this end, let $J = \bR_+$ or $J = (0,T)$ for some $T > 0$, let $\rX_0$ be a Banach space, and for a closed and densely defined operator $A \colon \rX_1 \subset \rX_0 \to \rX_0$, consider the inhomogeneous initial value problem
\begin{equation}\label{eq:inhom IVP}
    \left\{
    \begin{aligned}
        u'(t) + A u(t)
        &= f(t), &&\tfor t \in J,\\
        u(0)
        &= u_0
    \end{aligned}
    \right.
\end{equation}
in $\rL^p(J;\rX_0)$, where $p \in (1,\infty)$.

\begin{defn}\label{def:max Lp-reg}
A closed and densely defined operator $A \colon \rX_1 \subset \rX_0 \to \rX_0$ has maximal $\rL^p(J)$-regularity provided for all $f \in \rL^p(J;\rX_0)$, there exists a unique solution $u \in \rW^{1,p}(J;\rX_0) \cap \rL^p(J;\rX_1)$ solving \eqref{eq:inhom IVP} a.e.\ in $J$ with $u_0 = 0$.
In this case, the closed graph theorem implies the existence of a constant $C > 0$ such that
\begin{equation*}
    \| u \|_{\rL^p(J;\rX_0)} + \| u' \|_{\rL^p(J;\rX_0)} + \| A u \|_{\rL^p(J;\rX_0)} \le C \cdot \| f \|_{\rL^p(J;\rX_0)}
\end{equation*}
for all $f \in \rL^p(J;\rX_0)$.
\end{defn}

As already indicated, we now invoke the celebrated result on the characterization of maximal $\rL^p$-regularity on UMD spaces in terms of $\cR$-sectoriality due to Weis \cite[Thm.~4.2]{Wei:01}.

\begin{lem}\label{lem:char max Lp reg R-sect}
Let $\rX_0$ be a UMD space, $p \in (1,\infty)$, and consider a sectorial operator $A$ on $\rX_0$ with spectral angle $\phi_A < \nicefrac{\pi}{2}$.
Then $A$ has maximal $\rL^p(J)$-regularity if and only if $A$ is $\cR$-sectorial with $\cR$-angle $\phi_A^\cR < \nicefrac{\pi}{2}$, and in the case $J = \bR_+$, it additionally holds that $0 \in \rho(A)$.
\end{lem}

We remark that the additional assumption $0 \in \rho(A)$ in the case $J = \bR_+$ stems from the fact that it is quite restrictive to include the norm $\| u \|_{\rL^p(J;\rX_0)}$ in this case, since it already requires that the operator is invertible.

We conclude this section with the operator theoretic concept of the boundedness of the $\Hinfty$-calculus.

\begin{defn}\label{def:Hoo-calculus}
An operator $A \in \cS(\rX_0)$ with $\phi_A \in [0,\pi)$ and $\theta \in (\phi_A,\pi)$ has a {bounded $\Hinfty$-calculus}, denoted by $A \in \Hinfty(\rX)$, if there exists a constant $C > 0$ with
\begin{equation*}
    \| f(A) \|_{\cL(\rX_0)} \le C \cdot \| f \|_{\rH^\infty(\Sigma_\theta)} \tforall f \in \rH_0^\infty(\Sigma_\theta),
\end{equation*}
where $\| f \|_{\rH^\infty(\Sigma_\theta)} \coloneqq \sup\{|f(\lambda)| : |\arg \lambda | < \theta \}$, and $\rH_0^\infty(\Sigma_\theta)$ are the bounded and holomorphic functions on $\Sigma_\theta$ so that there are $C$, $\eps > 0$ with
\begin{equation*}
    |f(\lambda)| \le C \cdot \left|\frac{\lambda}{(1+\lambda)^2}\right|^\eps, \tfor \lambda \in \Sigma_\theta.
\end{equation*}
The infimum of such $\theta$ is denoted by $\phi_A^\infty$ and called the $\Hinfty$-angle of $A$.
\end{defn}

The inclusions below reveal the relations of the concepts invoked above, see also \cite[Sec.~4.4]{DHP:03} for further information.
It holds that
\begin{equation}\label{eq:rels op theoretic concepts}
    A \in \Hinfty(\rX_0) \subset \cR \cS(\rX_0) \subset \cS(\rX_0), \twith \phi_A^\infty \ge \phi_A^\cR \ge \phi_A.
\end{equation}
In view of \eqref{eq:rels op theoretic concepts} and \autoref{lem:char max Lp reg R-sect}, the boundedness of the $\Hinfty$-calculus with an angle strictly less than~$\nicefrac{\pi}{2}$ already implies the maximal $\rL^p$-regularity on a UMD space $\rX_0$.

\section{Boundary value problems and quasilinear theory}\label{sec:det theory}

The purpose of this section is to recall several analytical preliminaries for convenience of the reader.
These concepts comprise parabolic boundary value problems, quasilinear existence theory, the generalized principle of linearized stability for the stability analysis of equilibria, and existence theory for time-periodic evolution equations based on the Arendt-Bu theorem.

Throughout this section, for $\theta \in (0,\pi]$, we denote by $\Sigma_\theta$ the sector of angle $\theta$ in the complex plane as introduced in \autoref{sec:op theoret concepts}.
First, we recall some theory related to elliptic differential operators as well as the implications for the associated $\rL^q$-realizations.
For $x \in \Omega$, where $\Omega \subset \bR^d$ is a bounded domain with sufficiently regular boundary, a Banach space $\rE$ and coefficients $a_\alpha \in \cL(\rE)$, we consider a second order differential operator $\cA(x,\rD)$ with principal part $\cA_{\#}(x,\rD)$ and associated symbol $\cA_{\#}(x,\xi)$ given by
\begin{equation}\label{eq:diff op of second order}
    \cA(x,\rD) = \sum_{|\alpha| \le 2} a_\alpha(x) \rD^\alpha, \enspace \cA_{\#}(x,\rD) = \sum_{|\alpha| = 2} a_\alpha(x) \rD^\alpha \tand \cA_{\#}(x,\xi) = \sum_{|\alpha| = 2} a_\alpha(x) \xi^{\alpha},
\end{equation}
respectively.
In the situation of a differential operator acting on a $\bC^d$-valued function, the principal part and corresponding symbol take the form
\begin{equation}\label{eq:Cd-valued diff op}
    [\cA_{\#}(x,\rD)v(x)]_i = \sum_{j,k,l=1}^d a_{ij}^{kl}(x) \rD_k \rD_l v_j(x) \tand (\cA_{\#}(x,\xi))_{ij} = \sum_{k,l=1}^d a_{ij}^{kl}(x) \xi_k \xi_l.
\end{equation}
The operator $\cA(x,\rD)$ is said to be
\begin{enumerate}[(i)]
    \item \textit{parameter-elliptic} of angle $\phi \in (0,\pi]$ if $\sigma(\cA_{\#}(x,\xi)) \subset \Sigma_\phi$ for all $x \in \oOmega$ and for all $\xi \in \bR^d$ such that $|\xi| = 1$, and the \textit{angle of ellipticity} of $\cA$ is $\phi_{\cA} \coloneqq \inf\left\{\phi \in (0,\pi] : \sigma(\cA_{\#}(x,\xi)) \subset \Sigma_\phi\right\}$, and
    \item \textit{strongly elliptic} if $E$ is a Hilbert space with inner product $(\cdot,\cdot)_{\rE}$, and if there is a constant $c > 0$ so that $\re(\cA_{\#}(x,\xi)w | w)_{\rE} \ge c \cdot \| w \|_{\rE}^2$ for all $x \in \oOmega$, for all $\xi \in \bR^d$ with $|\xi| = 1$, and for all $w \in \rE$.
\end{enumerate}
For $m \in \{0,1\}$, $b_\beta \in \cL(\rE)$ and $x \in \del\Omega$, we consider boundary differential operators $\cB(x,\rD)$ with principal part $\cB_{\#}(x,\rD)$ given by
\begin{equation}\label{eq:boundary diff op}
    \cB(x,\rD) = \sum_{|\beta| \le m} b_\beta(x) \rD^\beta \tand \cB_{\#}(x,\rD) = \sum_{|\beta| = m} b_\beta(x) \rD^\beta.
\end{equation}
In this context, we make precise the so-called \textit{Lopatinskii-Shapiro condition}.

\begin{defn}\label{def:Lopatinskii Shapiro}
Let $\cA(x,\rD)$ as in \eqref{eq:diff op of second order} be a parameter-elliptic differential operator with angle of ellipticity $\phi_{\cA} \in [0,\pi)$, and let $\cB(x,\rD)$ denote a boundary differential operator as in \eqref{eq:boundary diff op}.
Then the \textit{Lopatinskii-Shapiro condition} is satisfied if for every $x_0 \in \dOmega$, the ODE problem in $\bR_+$ given by 
\begin{equation*}
     (\lambda + \cA_{\#}(x_0,\xi',\rD_y)) v(y) = 0, \enspace y > 0, \enspace \cB_{\#}(x_0,\xi',\rD_y)v(0) = g
\end{equation*}
has a unique solution $v \in \rC_0(\bR_+;\rE)$, where the subscript $_0$ indicates that the function is decaying to zero as $x \to \infty$, for all $g \in \rE$, $\lambda \in \overline{\Sigma}_{\phi_{\cA}}$ and $\xi' \in \bR^d$ such that $|\xi'| + |\lambda| \neq 0$.
\end{defn}

In the $\bC^d$-valued situation, there is another type of ellipticity that guarantees the validity of the Lopatinskii-Shapiro condition when considering Dirichlet or Neumann boundary conditions.
This stronger property is called \textit{strong normal ellipticity} and was first introduced by Bothe and Pr\"uss \cite{BP:07}.

\begin{defn}\label{def:strong normal ell}
Let $\cA(x,\rD)$ be a differential operator acting on {$\bC^d$-valued} functions with symbol of the principal part $\cA_{\#}(x,\xi)$ as made precise in \eqref{eq:Cd-valued diff op}.
Then the differential operator $\cA(x,\rD)$ is referred to as \textit{strongly normally elliptic} if $\cA(x,\rD)$ is strongly elliptic, and if it additionally holds that
\begin{equation*}
    \re \sum_{i,j,k,l=1}^d a_{ij}^{kl}(x) (\xi_l u_j - \nu_l v_j) \overline{(\xi_k u_i - \nu_k v_i)} > 0
\end{equation*}
for all $x \in \oOmega$, for all $\xi$, $\nu \in \bR^d$ with $|\xi| = |\nu| = 1$ and $(\xi|\nu) = 0$, and for all $u$, $v \in \bC^d$ with $\im(u|v) \neq 0$.
\end{defn}

For a proof of the lemma below asserting the validity of the Lopatinskii-Shapiro condition for a strongly normally elliptic $\bC^d$-valued differential operator with Dirichlet or Neumann boundary conditions, we refer to \cite[Section~3]{BP:07}, see also the discussion in \cite[Section~6.2.5]{PS:16}.

\begin{lem}\label{lem:strong normal ell implies Lopatinskii-Shapiro}
Let $\cA(x,\rD)$ denote a strongly normally elliptic $\bC^d$-valued operator, and consider Dirichlet or Neumann boundary conditions, meaning that $\cB(x,\rD) = \tr_{\dOmega}$ or $\cB(x,\rD) = \del_\nu$.
Then $(\cA,\cB)$ satisfies the Lopatinskii-Shapiro condition for all $x \in \dOmega$.
\end{lem}

We now invoke some smoothness and ellipticity conditions.
\begin{enumerate}
    \item[(S)] Let $a_\alpha$ and $b_\beta$ denote the coefficients of the differential operator $\cA$ and the boundary differential operator $\cB$ as introduced in \eqref{eq:diff op of second order} and \eqref{eq:boundary diff op}, respectively.
    We assume that
    \begin{enumerate}[(i)]
        \item $a_\alpha \in \rC^{0,\beta}(\oOmega;\cL(\rE))$ for $|\alpha| = 2$ for some $\beta \in (0,1)$,
        \item $a_\alpha \in \rL^\infty(\Omega;\cL(\rE))$ for $|\alpha| < 2$, and
        \item $b_\beta \in \rC^{2-m}(\dOmega;\cL(\rE))$ for $|\beta| \le m$, where $m \in \{0,1\}$.
    \end{enumerate}
    \item[(E)] We assume the existence of $\phi_{\cA} \in [0,\pi)$ such that
    \begin{enumerate}[(i)]
        \item $\cA$ is parameter-elliptic of angle $\phi_{\cA}$ for all $x \in \oOmega$, and
        \item $(\cA,\cB)$ satisfies the Lopatinskii-Shapiro condition for all $x \in \dOmega$.
    \end{enumerate}
\end{enumerate}

We denote by $A_B$ the \textit{$\rL^q(\Omega;\rE)$-realization of $\cA(x,\rD)$}, so $\rD(A_B) = \left\{u \in \rW^{2,q}(\Omega;\rE) : \cB(x,\rD) u = 0 \right\}$.
The next result, which is due to Denk, Dore, Hieber, Pr\"uss and Venni \cite{DDHPV:04}, establishes the bounded $\cH^\infty$-calculus of $A_B$, where the latter concept has been recalled in \autoref{def:Hoo-calculus}.
For the corresponding result on the $\cR$-sectoriality, see also \autoref{def:R-sect} for this notion, under slightly weaker regularity assumptions on the top-order coefficients, we refer to \cite[Theorem~8.2]{DHP:03}.
The second part of the assertion on the maximal $\rL^p$-regularity is classical, and it is a direct consequence of \autoref{lem:char max Lp reg R-sect}.

\begin{lem}[{\cite[Theorem~2.3]{DDHPV:04}}]\label{lem:max reg boundary value problem}
Consider a Hilbert space $\rE$, $d \in \bN$ and $q \in (1,\infty)$, and let $\Omega \subset \bR^d$ be a bounded domain with $\rC^2$-boundary.
Moreover, suppose that the boundary value problem $(\cA,\cB)$ satisfies the smoothness and ellipticity conditions (S) and (E) from above for some $\phi_{\cA} \in [0,\pi)$.

Then for every $\phi > \phi_{\cA}$, there is $\mu_\phi$ such that $A_B + \mu_\phi \in \cH^\infty(\rL^q(\Omega;\rE))$ with $\phi_{A_B + \mu_\phi}^\infty \le \phi_{\cA}$.

In particular, $A_B + \mu_\phi$ is $\cR$-sectorial with $\cR$-angle $\phi_{A_B + \mu_\phi}^\cR \le \phi_{A_B + \mu_\phi}^\infty \le \phi_{\cA}$, and if $\phi_{\cA} < \nicefrac{\pi}{2}$, then $A_B + \mu_\phi$ has the property of maximal regularity in $\rL^p(\bR_+;\rL^q(\Omega;\rE))$ for every $p \in (1,\infty)$.
\end{lem}

Let us now recall some quasilinear theory.
We follow here \cite[Ch.~5]{PS:16}.
More precisely, we consider Banach spaces $\rX_0$ and $\rX_1$ such that $\rX_1 \hookrightarrow \rX_0$, and an open subset $V_\mu \subset \rX_{\gamma,\mu}$ of the interpolation space
\begin{equation*}
    \rX_{\gamma,\mu} \coloneqq (\rX_0,\rX_1)_{\mu-\nicefrac{1}{p},p}, \tfor \mu \in (\nicefrac{1}{p},1].
\end{equation*}
For $(A,F) \colon V_\mu \to \cL(\rX_1,\rX_0) \times \rX_0$ and $u_0 \in V_\mu$, we study the evolution equation 
\begin{equation}\label{eq:quasilin ACP}
    u'(t) + A(u) u = F(u), \tfor t > 0, \enspace u(0) = u_0.
\end{equation}

The local strong well-posedness result for \eqref{eq:quasilin ACP} now reads as follows.
The assertion of~(a) can be found in \cite[Thm.~5.1.1]{PS:16}, while the assertion of~(b) is in \cite[Cor.~5.1.2]{PS:16}.
We refer here also to \cite{LPW:14}.
For the underlying concept of maximal $\rL^p$-regularity in time-weighted spaces, we refer to the article by Pr\"uss and Simonett \cite{PS:04}.
Time-weighted Lebesgue and Sobolev spaces have been defined in \autoref{sec:well-posedness}.

\begin{lem}\label{lem:loc wp quasilin ACP}
Let $p \in (1,\infty)$ and $u_0 \in V_\mu$, and assume $(A,F) \in \rC^{0,1}(V_\mu;\cL(\rX_1,\rX_0) \times \rX_0)$ for some $\mu \in (\nicefrac{1}{p},1]$.
Moreover, suppose that $A(u_0)$ has maximal $\rL^p$-regularity on $\rX_0$.

\begin{enumerate}[(a)]
    \item Then there are $T' = T'(u_0) > 0$ and $r = r(u_0) > 0$ with $\oB_{\rX_{\gamma,\mu}}(u_0,r) \subset V_\mu$ such that for every initial value $w_1 \in \oB_{\rX_{\gamma,\mu}}(u_0,r)$, the abstract Cauchy problem \eqref{eq:quasilin ACP} admits a unique solution
    \begin{equation*}
        u(\cdot;w_1) \in \bE_{1,\mu}(0,T') \coloneqq \rW_\mu^{1,p}(0,T';\rX_0) \cap \rL_\mu^p(0,T';\rX_1) \cap \rC([0,T'];V_\mu)
    \end{equation*}
    on $(0,T')$, and there is a constant $c = c(u_0) > 0$ so that for all $w_1$, $w_2 \in \oB_{\rX_{\gamma,\mu}}(u_0,r)$, we have
    \begin{equation*}
        \| u(\cdot;w_1) - u(\cdot;w_2) \|_{\bE_{1,\mu}(0,T')} \le c \cdot \| w_1 - w_2 \|_{\rX_{\gamma,\mu}}.
    \end{equation*}
    In addition, for every $\tau \in (0,T')$, it holds that $u \in \bE_1(\tau,T') \coloneqq \bE_{1,1}(\tau,T') \hookrightarrow \rC([\tau,T'];\rX_1)$, so the solution regularizes instantaneously.
    \item If $A(v)$ has maximal $\rL^p$-regularity on $\rX_0$ for all $v \in V_\mu$, then the solution $u(t)$ of the abstract Cauchy problem \eqref{eq:quasilin ACP} has a maximal time interval of existence $J(u_0) = [0,t_+(u_0))$ that is characterized by the following alternatives: 
    \newline
    (i) global existence: $t_+(u_0) = \infty$; (ii) $\liminf_{t \to t_+(u_0)} \dist_{\rX_{\gamma,\mu}}(u(t),\del V_\mu) = 0$; (iii) $\lim_{t \to t_+(u_0)} u(t)$ does not exist in $\rX_{\gamma,\mu}$.
\end{enumerate}
\end{lem}

Next, we elaborate on the stability of equilibria to the abstract Cauchy problem \eqref{eq:quasilin ACP}.
In fact, we consider here the real interpolation space $\rX_\gamma \coloneqq \rX_{\gamma,1}$, along with an open subset $V \subset \rX_\gamma$.
By $\cE \subset V \cap \rX_1$, we then denote the set of equilibrium solutions to \eqref{eq:quasilin ACP}, i.e., $u \in \cE$ satisfies $A(u) u = F(u)$.
Upon assuming that $A$ and $F$ are Fr\'echet-differentiable, denoted by $(A,F) \in \rC^1(V;\cL(\rX_1,\rX_0) \times \rX_0)$, we invoke the total linearization of $A$ at an equilibrium $u_* \in \cE$.
It is given by 
\begin{equation}\label{eq:abstr tot lin}
    A_0 v \coloneqq A(u_*) v +(A'(u_*)v)u_* - F'(u_*)v \tfor v \in \rX_1.
\end{equation}

The subsequent generalized principle of linearized stability discusses the stability of equilibria.
We refer here to \cite[Thm.~5.3.1]{PS:16}, see also \cite{PSZ:09}.

\begin{lem}\label{lem:gen princ of lin stab}
Consider $p \in (1,\infty)$, and assume that $u_* \in V \cap \rX_1$ is an equilibrium solution to \eqref{eq:quasilin ACP}, and suppose that $(A,F) \in \rC^1(V;\cL(\rX_1,\rX_0) \times \rX_0)$.
Moreover, assume that $A(u_*)$ has maximal $\rL^p$-regularity on $\rX_0$, and recall the total linearization $A_0$ of $A$ around $u_*$ as made precise in \eqref{eq:abstr tot lin}.
The equilibrium is referred to as {\em normally stable} if
\begin{enumerate}[(a)]
    \item near $u_*$, the set of equilibria $\cE$ is a $\rC^1$-manifold in $\rX_1$ of dimension $m \in \bN$,
    \item the tangent space for $\cE$ at $u_*$ is isomorphic to $\rN(A_0)$, 
    \item zero is a semi-simple eigenvalue of $A_0$, that is, $\rN(A_0) \oplus \rR(A_0) = \rX_0$, and
    \item $\sigma(A_0) \setminus \{0\} \subset \bC_+$.
\end{enumerate}
Then the equilibrium $u_*$ is stable in $\rX_\gamma$, and there is $r > 0$ such that the unique solution $u(t)$ of \eqref{eq:quasilin ACP} with initial value $u_0 \in \rX_\gamma$ satisfying $\| u_0 - u_* \|_{\rX_\gamma} < r$ exists on $\bR_+$ and converges at an exponential rate in $\rX_\gamma$ to some $u_\infty \in \cE$ as $t \to \infty$.
\end{lem}

Below, we recall some theory regarding time-periodic problems.
First, we invoke the notion of maximal periodic $\rL^p$-regularity.
Assume that $\rX_0$ and $\rX_1$ are Banach spaces such that $\rX_1 \hookrightarrow \rX_0$ densely, and consider the generator $A \colon \rX_1 \subset \rX_0 \to \rX_0$ of a $\rC_0$-semigroup on $\rX_0$.
We say that $A$ has maximal periodic $\rL^p$-regularity if for all $f \in \rL^p(0,2\pi;\rX_0) \eqqcolon \bE_0$, there is a unique solution $u \in \rW^{1,p}(0,2 \pi;\rX_0) \cap \rL^p(0,2\pi;\rX_1)$ to the inhomogeneous periodic abstract Cauchy problem
\begin{equation*}
    \left\{
    \begin{aligned}
        u'(t) + A u(t)
        &= f(t), \tfor t \in (0,2\pi),\\
        u(0)
        &= u(2 \pi).
    \end{aligned}
    \right.
\end{equation*}
In this case, the closed graph theorem yields that there is a constant $C > 0$ with 
\begin{equation*}
    \| u \|_{\bE_1} \le C \cdot \| f \|_{\bE_0}.
\end{equation*}
The following result, which is due to Arendt and Bu \cite{AB:02}, provides a characterization of maximal $\rL^p$-regularity in terms of maximal $\rL^p$-regularity for initial value problems and a spectral condition on $A$.

\begin{lem}\label{lem:Arendt-Bu}
Let $\rX_0$ be a Banach space, and assume that $A \colon \rX_1 \to \rX_0$ generates a $\rC_0$-semigroup $(\mre^{t A})$.
Then~$A$ has maximal periodic $\rL^p$-regularity if and only if $A$ has maximal $\rL^p$-regularity as recalled in \autoref{def:max Lp-reg}, and $1 \in \rho(\mre^{2 \pi A})$.
\end{lem}

By a rescaling argument, the Arendt-Bu theorem remains valid for time intervals of the form $(0,T)$, where $T > 0$.
We also observe that the spectral mapping theorem for generators of analytic semigroups, see, e.g., \cite[Cor.IV.3.12]{EN:00}, shows that the condition $1 \in \rho(\mre^{2 \pi A})$ is equivalent with $0 \in \rho(A)$. 

Below, we discuss the solvability of time-periodic quasilinear abstract Cauchy problems of the form
\begin{equation}\label{eq:time-per quasilin ACP}
    \left\{
    \begin{aligned}
        u'(t) + A(u(t)) u(t)
        &= F(t,u(t)), \tfor t \in (0,T),\\
        u(0) 
        &= u(T),
    \end{aligned}
    \right.
\end{equation}
where for $V \subset \rX_\gamma = (\rX_0,\rX_1)_{1-\nicefrac{1}{p},p}$ open, we suppose that $A \colon V \to \cL(\rX_1,\rX_0)$ and $F \colon [0,T] \times V \to \rX_0$.

We make the following assumptions on the nonlinearities $A$ and $F$, and on the linearization $A(0)$.

\begin{asu}\label{ass:time-per quasilin}
Consider $r_0 > 0$ such that for every $u \in \oB_{\bE_1}(0,r_0)$, it holds that $u(t) \in V$ for all $t \in [0,T]$.
Moreover, suppose that the following are valid.
\begin{enumerate}[(a)]
    \item The operators $A \colon V \to \cL(\rX_1,\rX_0)$ are a family of closed linear operators, and for all $r \in (0,r_0)$, there is $L(r) > 0$ with $\| A(u_1) - A(u_2) \|_{\bE_0} \le L(r) \cdot \| u_1 - u_2 \|_{\bE_1} \cdot \| v \|_{\bE_1}$ for all $u_1$, $u_2 \in \oB_{\bE_1}(0,r)$ and $v \in \bE_1$.
    \item The right-hand side $F \colon [0,T] \times V \to \rX_0$ satisfies $F(\cdot,u(\cdot)) \in \bE_0$ for all $u \in \oB_{\bE_1}(0,r)$, and for each $r \in (0,r_0)$, there exists $C(r)$ such that $\| F(\cdot,u_1(\cdot)) - F(\cdot,u_2(\cdot)) \|_{\bE_0} \le C(r) \cdot \| u_1 - u_2 \|_{\bE_1}$ for all $u_1$, $u_2 \in \oB_{\bE_1}(0,r)$.
    Also, assume that $F(0,u) = F(T,u)$ for all $u \in \oB_{\bE_1}(0,r)$.
    \item For all $p \in (1,\infty)$, the operator $A_0 \coloneqq A(0)$ satisfies $0 \in \rho(A_0)$, and $A_0$ has maximal $\rL^p$-regularity.
\end{enumerate}
\end{asu}

The following quasilinear existence result can be obtained in a similar way as the quasilinear version of the Arendt-Bu theorem from \cite[Thm.~3.3]{HS:20}, see also \cite[Cor.~7.1.4]{Bra:24}.

\begin{lem}\label{lem:quasilin time per result}
Suppose that \autoref{ass:time-per quasilin} is fulfilled.
Moreover, assume that for some $r \in (0,r_0)$, the constant $C(r) > 0$ satisfies $C(r) < \delta_1$ for $\delta_1 > 0$ sufficiently small.

Then there are $r' \in (0,r)$ and $\delta_2 = \delta_2(r') > 0$ so that if $\| F(\cdot,0) \|_{\bE_0} < \delta_2$, there exists a solution $u \in \oB_{\bE_1}(0,r')$ to \eqref{eq:time-per quasilin ACP}.
Moreover, $u$ is unique in $\oB_{\bE_1}(0,r')$.
\end{lem}

\end{document}